\documentclass[11pt,american,english]{article}
\usepackage{lmodern}

\usepackage[T1]{fontenc}
\usepackage[latin9]{inputenc}
\usepackage{geometry}
\usepackage{mathtools}
\usepackage{bm}
\usepackage{dsfont}
\usepackage{amsmath}
\usepackage{amsthm}
\usepackage{amssymb}
\usepackage{graphicx}
\usepackage{wasysym}

\makeatletter

\newcommand{\lyxmathsym}[1]{\ifmmode\begingroup\def\b@ld{bold}
  \text{\ifx\math@version\b@ld\bfseries\fi#1}\endgroup\else#1\fi}

\numberwithin{figure}{section}
\numberwithin{equation}{section}
\theoremstyle{plain}
\newtheorem{thm}{\protect\theoremname}[section]
\theoremstyle{plain}
\newtheorem{lem}[thm]{\protect\lemmaname}
\theoremstyle{definition}
\newtheorem{defn}[thm]{\protect\definitionname}
\theoremstyle{plain}
\newtheorem{conjecture}[thm]{\protect\conjecturename}
\theoremstyle{plain}
\newtheorem{prop}[thm]{\protect\propositionname}
\theoremstyle{remark}
\newtheorem{rem}[thm]{\protect\remarkname}
\theoremstyle{remark}
\newtheorem{claim}[thm]{\protect\claimname}
\theoremstyle{plain}
\newtheorem{cor}[thm]{\protect\corollaryname}
\theoremstyle{definition}
\newtheorem{problem}[thm]{\protect\problemname}
\theoremstyle{remark}
\newtheorem*{acknowledgement*}{\protect\acknowledgementname}
\theoremstyle{definition}
\newtheorem{example}[thm]{\protect\examplename}

\usepackage{float,psfrag,epsfig,color,url,hyperref}
\hypersetup{
  colorlinks,
  linkcolor={blue!50!black},
  citecolor={blue!50!black},
  urlcolor={blue!80!black}
}
\usepackage{comment,url,algorithm,algorithmic,graphicx,relsize}
\usepackage{amssymb,amsfonts,amsmath,amsthm,amscd,dsfont,mathrsfs,mathtools,nicefrac,pifont}
\usepackage{upgreek}
\usepackage[dvipsnames]{xcolor}
\usepackage{epstopdf,bbm,mathtools,enumitem}
\usepackage[toc,page]{appendix}
\usepackage{dsfont,tikz}
\usepackage{aligned-overset}
\usepackage[mathscr]{euscript}
\allowdisplaybreaks
\usepackage{custom}
\makeatletter
\renewcommand{\paragraph}{%
  \@startsection{paragraph}{4}%
  {\z@}{1.25ex \@plus 1ex \@minus .2ex}{-1em}%
  {\normalfont\normalsize\bfseries}%
}
\makeatother
\usepackage{wrapfig}

\usepackage{subfig}

\makeatother

\usepackage{babel}
\addto\captionsamerican{\renewcommand{\acknowledgementname}{Acknowledgement}}
\addto\captionsamerican{\renewcommand{\claimname}{Claim}}
\addto\captionsamerican{\renewcommand{\conjecturename}{Conjecture}}
\addto\captionsamerican{\renewcommand{\corollaryname}{Corollary}}
\addto\captionsamerican{\renewcommand{\definitionname}{Definition}}
\addto\captionsamerican{\renewcommand{\examplename}{Example}}
\addto\captionsamerican{\renewcommand{\lemmaname}{Lemma}}
\addto\captionsamerican{\renewcommand{\problemname}{Problem}}
\addto\captionsamerican{\renewcommand{\propositionname}{Proposition}}
\addto\captionsamerican{\renewcommand{\remarkname}{Remark}}
\addto\captionsamerican{\renewcommand{\theoremname}{Theorem}}
\addto\captionsenglish{\renewcommand{\acknowledgementname}{Acknowledgement}}
\addto\captionsenglish{\renewcommand{\claimname}{Claim}}
\addto\captionsenglish{\renewcommand{\conjecturename}{Conjecture}}
\addto\captionsenglish{\renewcommand{\corollaryname}{Corollary}}
\addto\captionsenglish{\renewcommand{\definitionname}{Definition}}
\addto\captionsenglish{\renewcommand{\examplename}{Example}}
\addto\captionsenglish{\renewcommand{\lemmaname}{Lemma}}
\addto\captionsenglish{\renewcommand{\problemname}{Problem}}
\addto\captionsenglish{\renewcommand{\propositionname}{Proposition}}
\addto\captionsenglish{\renewcommand{\remarkname}{Remark}}
\addto\captionsenglish{\renewcommand{\theoremname}{Theorem}}
\providecommand{\acknowledgementname}{Acknowledgement}
\providecommand{\claimname}{Claim}
\providecommand{\conjecturename}{Conjecture}
\providecommand{\corollaryname}{Corollary}
\providecommand{\definitionname}{Definition}
\providecommand{\examplename}{Example}
\providecommand{\lemmaname}{Lemma}
\providecommand{\problemname}{Problem}
\providecommand{\propositionname}{Proposition}
\providecommand{\remarkname}{Remark}
\providecommand{\theoremname}{Theorem}

\begin{document}
\def\balign#1\ealign{\begin{align}#1\end{align}}
\def\baligns#1\ealigns{\begin{align*}#1\end{align*}}
\def\balignat#1\ealign{\begin{alignat}#1\end{alignat}}
\def\balignats#1\ealigns{\begin{alignat*}#1\end{alignat*}}
\def\bitemize#1\eitemize{\begin{itemize}#1\end{itemize}}
\def\benumerate#1\eenumerate{\begin{enumerate}#1\end{enumerate}}

\newenvironment{talign*}
 {\let\displaystyle\textstyle\csname align*\endcsname}
 {\endalign}
\newenvironment{talign}
 {\let\displaystyle\textstyle\csname align\endcsname}
 {\endalign}

\def\balignst#1\ealignst{\begin{talign*}#1\end{talign*}}
\def\balignt#1\ealignt{\begin{talign}#1\end{talign}}

\let\originalleft\left
\let\originalright\right
\renewcommand{\left}{\mathopen{}\mathclose\bgroup\originalleft}
\renewcommand{\right}{\aftergroup\egroup\originalright}

\def\Gronwall{Gr\"onwall\xspace}
\def\Holder{H\"older\xspace}
\def\Ito{It\^o\xspace}
\def\Nystrom{Nystr\"om\xspace}
\def\Schatten{Sch\"atten\xspace}
\def\Matern{Mat\'ern\xspace}

\def\tinycitep*#1{{\tiny\citep*{#1}}}
\def\tinycitealt*#1{{\tiny\citealt*{#1}}}
\def\tinycite*#1{{\tiny\cite*{#1}}}
\def\smallcitep*#1{{\scriptsize\citep*{#1}}}
\def\smallcitealt*#1{{\scriptsize\citealt*{#1}}}
\def\smallcite*#1{{\scriptsize\cite*{#1}}}

\def\blue#1{\textcolor{blue}{{#1}}}
\def\green#1{\textcolor{green}{{#1}}}
\def\orange#1{\textcolor{orange}{{#1}}}
\def\purple#1{\textcolor{purple}{{#1}}}
\def\red#1{\textcolor{red}{{#1}}}
\def\teal#1{\textcolor{teal}{{#1}}}

\def\mbi#1{\boldsymbol{#1}} 
\def\mbf#1{\mathbf{#1}}
\def\mrm#1{\mathrm{#1}}
\def\tbf#1{\textbf{#1}}
\def\tsc#1{\textsc{#1}}

\def\mbiA{\mbi{A}}
\def\mbiB{\mbi{B}}
\def\mbiC{\mbi{C}}
\def\mbiDelta{\mbi{\Delta}}
\def\mbif{\mbi{f}}
\def\mbiF{\mbi{F}}
\def\mbih{\mbi{g}}
\def\mbiG{\mbi{G}}
\def\mbih{\mbi{h}}
\def\mbiH{\mbi{H}}
\def\mbiI{\mbi{I}}
\def\mbim{\mbi{m}}
\def\mbiP{\mbi{P}}
\def\mbiQ{\mbi{Q}}
\def\mbiR{\mbi{R}}
\def\mbiv{\mbi{v}}
\def\mbiV{\mbi{V}}
\def\mbiW{\mbi{W}}
\def\mbiX{\mbi{X}}
\def\mbiY{\mbi{Y}}
\def\mbiZ{\mbi{Z}}

\def\textsum{{\textstyle\sum}} 
\def\textprod{{\textstyle\prod}} 
\def\textbigcap{{\textstyle\bigcap}} 
\def\textbigcup{{\textstyle\bigcup}} 

\def\reals{\mathbb{R}} 
\def\integers{\mathbb{Z}} 
\def\rationals{\mathbb{Q}} 
\def\naturals{\mathbb{N}} 
\def\complex{\mathbb{C}} 

\def\what#1{\widehat{#1}}

\def\twovec#1#2{\left[\begin{array}{c}{#1} \\ {#2}\end{array}\right]}
\def\threevec#1#2#3{\left[\begin{array}{c}{#1} \\ {#2} \\ {#3} \end{array}\right]}
\def\nvec#1#2#3{\left[\begin{array}{c}{#1} \\ {#2} \\ \vdots \\ {#3}\end{array}\right]} 

\def\maxeig#1{\lambda_{\mathrm{max}}\left({#1}\right)}
\def\mineig#1{\lambda_{\mathrm{min}}\left({#1}\right)}

\def\Re{\operatorname{Re}} 
\def\indic#1{\mbb{I}\left[{#1}\right]} 
\def\logarg#1{\log\left({#1}\right)} 
\def\polylog{\operatorname{polylog}}
\def\maxarg#1{\max\left({#1}\right)} 
\def\minarg#1{\min\left({#1}\right)} 
\def\Earg#1{\E\left[{#1}\right]}
\def\Esub#1{\E_{#1}}
\def\Esubarg#1#2{\E_{#1}\left[{#2}\right]}
\def\bigO#1{\mathcal{O}\left(#1\right)} 
\def\littleO#1{o(#1)} 
\def\P{\mbb{P}} 
\def\Parg#1{\P\left({#1}\right)}
\def\Psubarg#1#2{\P_{#1}\left[{#2}\right]}
\def\Trarg#1{\Tr\left[{#1}\right]} 
\def\trarg#1{\tr\left[{#1}\right]} 
\def\Var{\mrm{Var}} 
\def\Vararg#1{\Var\left[{#1}\right]}
\def\Varsubarg#1#2{\Var_{#1}\left[{#2}\right]}
\def\Cov{\mrm{Cov}} 
\def\Covarg#1{\Cov\left[{#1}\right]}
\def\Covsubarg#1#2{\Cov_{#1}\left[{#2}\right]}
\def\Corr{\mrm{Corr}} 
\def\Corrarg#1{\Corr\left[{#1}\right]}
\def\Corrsubarg#1#2{\Corr_{#1}\left[{#2}\right]}
\newcommand{\info}[3][{}]{\mathbb{I}_{#1}\left({#2};{#3}\right)} 
\newcommand{\staticexp}[1]{\operatorname{exp}(#1)} 
\newcommand{\loglihood}[0]{\mathcal{L}} 


\providecommand{\arccos}{\mathop\mathrm{arccos}}
\providecommand{\dom}{\mathop\mathrm{dom}}
\providecommand{\diag}{\mathop\mathrm{diag}}
\providecommand{\tr}{\mathop\mathrm{tr}}
\providecommand{\card}{\mathop\mathrm{card}}
\providecommand{\sign}{\mathop\mathrm{sign}}
\providecommand{\conv}{\mathop\mathrm{conv}} 
\def\rank#1{\mathrm{rank}({#1})}
\def\supp#1{\mathrm{supp}({#1})}

\providecommand{\minimize}{\mathop\mathrm{minimize}}
\providecommand{\maximize}{\mathop\mathrm{maximize}}
\providecommand{\subjectto}{\mathop\mathrm{subject\;to}}

\def\openright#1#2{\left[{#1}, {#2}\right)}

\ifdefined\nonewproofenvironments\else
\ifdefined\ispres\else
 
\fi
\makeatletter
\@addtoreset{equation}{section}
\makeatother
\def\theequation{\thesection.\arabic{equation}}

\newcommand{\cmark}{\ding{51}}

\newcommand{\xmark}{\ding{55}}

\newcommand{\eq}[1]{\begin{align}#1\end{align}}
\newcommand{\eqn}[1]{\begin{align*}#1\end{align*}}
\renewcommand{\Pr}[1]{\mathbb{P}\left( #1 \right)}
\newcommand{\Ex}[1]{\mathbb{E}\left[#1\right]}

\newcommand{\matt}[1]{{\textcolor{Maroon}{[Matt: #1]}}}
\newcommand{\kook}[1]{{\textcolor{blue}{[Kook: #1]}}}
\definecolor{OliveGreen}{rgb}{0,0.6,0}
\newcommand{\sv}[1]{{\textcolor{OliveGreen}{[Santosh: #1]}}}

\global\long\def\on#1{\operatorname{#1}}%

\global\long\def\bw{\mathsf{Ball\ walk}}%
\global\long\def\sw{\mathsf{Speedy\ walk}}%
\global\long\def\gw{\mathsf{Gaussian\ walk}}%
\global\long\def\ps{\mathsf{Proximal\ sampler}}%
\global\long\def\dw{\mathsf{Dikin\ walk}}%

\global\long\def\chr{\mathsf{Coordinate\ Hit\text{-}and\text{-}Run}}%
\global\long\def\har{\mathsf{Hit\text{-}and\text{-}Run}}%
\global\long\def\gc{\mathsf{Gaussian\ cooling}}%
\global\long\def\ino{\mathsf{\mathsf{In\text{-}and\text{-}Out}}}%
\global\long\def\tgc{\mathsf{Tilted\ Gaussian\ cooling}}%
\global\long\def\PS{\mathsf{PS}}%
\global\long\def\psunif{\mathsf{PS}_{\textup{unif}}}%
\global\long\def\psexp{\mathsf{PS}_{\textup{exp}}}%
\global\long\def\psann{\mathsf{PS}_{\textup{ann}}}%
\global\long\def\psgauss{\mathsf{PS}_{\textup{Gauss}}}%
\global\long\def\eval{\mathsf{Eval}}%
\global\long\def\mem{\mathsf{Mem}}%

\global\long\def\O{O}%
\global\long\def\Otilde{\widetilde{O}}%
\global\long\def\Omtilde{\widetilde{\Omega}}%

\global\long\def\E{\mathbb{E}}%
\global\long\def\Z{\mathbb{Z}}%
\global\long\def\P{\mathbb{P}}%
\global\long\def\N{\mathbb{N}}%

\global\long\def\R{\mathbb{R}}%
\global\long\def\Rd{\mathbb{R}^{d}}%
\global\long\def\Rdd{\mathbb{R}^{d\times d}}%
\global\long\def\Rn{\mathbb{R}^{n}}%
\global\long\def\Rnn{\mathbb{R}^{n\times n}}%
\global\long\def\C{\mathbb{C}}%

\global\long\def\psd{\mathbb{S}_{+}^{d}}%
\global\long\def\pd{\mathbb{S}_{++}^{d}}%

\global\long\def\defeq{\stackrel{\mathrm{{\scriptscriptstyle def}}}{=}}%

\global\long\def\veps{\varepsilon}%
\global\long\def\lda{\lambda}%
\global\long\def\vphi{\varphi}%
\global\long\def\K{\mathcal{K}}%

\global\long\def\half{\frac{1}{2}}%
\global\long\def\nhalf{\nicefrac{1}{2}}%
\global\long\def\texthalf{{\textstyle \frac{1}{2}}}%
\global\long\def\ltwo{L^{2}}%

\global\long\def\ind{\mathds{1}}%
\global\long\def\op{\mathsf{op}}%
\global\long\def\ch{\mathsf{Ch}}%
\global\long\def\kls{\mathsf{KLS}}%
\global\long\def\ts{\mathsf{TS}}%
\global\long\def\hs{\textup{HS}}%

\global\long\def\cpi{C_{\mathsf{PI}}}%
\global\long\def\clsi{C_{\mathsf{LSI}}}%
\global\long\def\cch{C_{\mathsf{Ch}}}%
\global\long\def\clch{C_{\mathsf{logCh}}}%
\global\long\def\cexp{C_{\mathsf{exp}}}%
\global\long\def\cgauss{C_{\mathsf{Gauss}}}%

\global\long\def\chooses#1#2{_{#1}C_{#2}}%

\global\long\def\vol{\on{vol}}%

\global\long\def\law{\on{law}}%

\global\long\def\tr{\on{tr}}%

\global\long\def\diag{\on{diag}}%

\global\long\def\diam{\on{diam}}%

\global\long\def\poly{\on{poly}}%

\global\long\def\polylog{\on{polylog}}%

\global\long\def\Diag{\on{Diag}}%

\global\long\def\inter{\on{int}}%

\global\long\def\esssup{\on{ess\,sup}}%

\global\long\def\proj{\on{Proj}}%

\global\long\def\e{\mathrm{e}}%

\global\long\def\id{\mathrm{id}}%

\global\long\def\supp{\on{supp}}%

\global\long\def\spanning{\on{span}}%

\global\long\def\rows{\on{row}}%

\global\long\def\cols{\on{col}}%

\global\long\def\rank{\on{rank}}%

\global\long\def\T{\mathsf{T}}%

\global\long\def\bs#1{\boldsymbol{#1}}%

\global\long\def\eu#1{\EuScript{#1}}%

\global\long\def\mb#1{\mathbf{#1}}%

\global\long\def\mbb#1{\mathbb{#1}}%

\global\long\def\mc#1{\mathcal{#1}}%

\global\long\def\mf#1{\mathfrak{#1}}%

\global\long\def\ms#1{\mathscr{#1}}%

\global\long\def\mss#1{\mathsf{#1}}%

\global\long\def\msf#1{\mathsf{#1}}%

\global\long\def\textint{{\textstyle \int}}%
\global\long\def\Dd{\mathrm{D}}%
\global\long\def\D{\mathrm{d}}%
\global\long\def\grad{\nabla}%
 
\global\long\def\hess{\nabla^{2}}%
 
\global\long\def\lapl{\triangle}%
 
\global\long\def\deriv#1#2{\frac{\D#1}{\D#2}}%
 
\global\long\def\pderiv#1#2{\frac{\partial#1}{\partial#2}}%
 
\global\long\def\de{\partial}%
\global\long\def\lagrange{\mathcal{L}}%
\global\long\def\Div{\on{div}}%

\global\long\def\Gsn{\mathcal{N}}%
 
\global\long\def\BeP{\textnormal{BeP}}%
 
\global\long\def\Ber{\textnormal{Ber}}%
 
\global\long\def\Bern{\textnormal{Bern}}%
 
\global\long\def\Bet{\textnormal{Beta}}%
 
\global\long\def\Beta{\textnormal{Beta}}%
 
\global\long\def\Bin{\textnormal{Bin}}%
 
\global\long\def\BP{\textnormal{BP}}%
 
\global\long\def\Dir{\textnormal{Dir}}%
 
\global\long\def\DP{\textnormal{DP}}%
 
\global\long\def\Expo{\textnormal{Expo}}%
 
\global\long\def\Gam{\textnormal{Gamma}}%
 
\global\long\def\GEM{\textnormal{GEM}}%
 
\global\long\def\HypGeo{\textnormal{HypGeo}}%
 
\global\long\def\Mult{\textnormal{Mult}}%
 
\global\long\def\NegMult{\textnormal{NegMult}}%
 
\global\long\def\Poi{\textnormal{Poi}}%
 
\global\long\def\Pois{\textnormal{Pois}}%
 
\global\long\def\Unif{\textnormal{Unif}}%

\global\long\def\bpar#1{\bigl(#1\bigr)}%
\global\long\def\Bpar#1{\Bigl(#1\Bigr)}%

\global\long\def\abs#1{|#1|}%
\global\long\def\babs#1{\bigl|#1\bigr|}%
\global\long\def\Babs#1{\Bigl|#1\Bigr|}%

\global\long\def\snorm#1{\|#1\|}%
\global\long\def\bnorm#1{\bigl\Vert#1\bigr\Vert}%
\global\long\def\Bnorm#1{\Bigl\Vert#1\Bigr\Vert}%

\global\long\def\sbrack#1{[#1]}%
\global\long\def\bbrack#1{\bigl[#1\bigr]}%
\global\long\def\Bbrack#1{\Bigl[#1\Bigr]}%

\global\long\def\sbrace#1{\{#1\}}%
\global\long\def\bbrace#1{\bigl\{#1\bigr\}}%
\global\long\def\Bbrace#1{\Bigl\{#1\Bigr\}}%

\global\long\def\Abs#1{\left\lvert #1\right\rvert }%
\global\long\def\Par#1{\left(#1\right)}%
\global\long\def\Brack#1{\left[#1\right]}%
\global\long\def\Brace#1{\left\{  #1\right\}  }%

\global\long\def\inner#1{\langle#1\rangle}%
 
\global\long\def\binner#1#2{\left\langle {#1},{#2}\right\rangle }%

\global\long\def\norm#1{\lVert#1\rVert}%
\global\long\def\onenorm#1{\norm{#1}_{1}}%
\global\long\def\twonorm#1{\norm{#1}_{2}}%
\global\long\def\infnorm#1{\norm{#1}_{\infty}}%
\global\long\def\fronorm#1{\norm{#1}_{\text{F}}}%
\global\long\def\nucnorm#1{\norm{#1}_{*}}%
\global\long\def\staticnorm#1{\|#1\|}%
\global\long\def\statictwonorm#1{\staticnorm{#1}_{2}}%

\global\long\def\mmid{\mathbin{\|}}%

\global\long\def\otilde#1{\widetilde{O}(#1)}%
\global\long\def\wtilde{\widetilde{W}}%
\global\long\def\wt#1{\widetilde{#1}}%

\global\long\def\KL{\msf{KL}}%
\global\long\def\dtv{d_{\textrm{\textup{TV}}}}%
\global\long\def\FI{\msf{FI}}%
\global\long\def\tv{\msf{TV}}%
\global\long\def\TV{\msf{TV}}%

\global\long\def\cov{\on{cov}}%
\global\long\def\var{\on{Var}}%
\global\long\def\ent{\on{Ent}}%

\global\long\def\cred#1{\textcolor{red}{#1}}%
\global\long\def\cblue#1{\textcolor{blue}{#1}}%
\global\long\def\cgreen#1{\textcolor{green}{#1}}%
\global\long\def\ccyan#1{\textcolor{cyan}{#1}}%

\global\long\def\iff{\Leftrightarrow}%
 
\global\long\def\textfrac#1#2{{\textstyle \frac{#1}{#2}}}%

\definecolor{c063cb8}{RGB}{0,64,192}
\definecolor{cff0000}{RGB}{255,0,0}
\definecolor{c009200}{RGB}{0,128,0}
\definecolor{c008000}{RGB}{0,128,0}
\definecolor{cff6666}{RGB}{255,128,128}
\definecolor{c008600}{RGB}{0,128,0} 
\title{The Localization Method for High-dimensional Inequalities\date{}\author{Yunbum Kook\\ Georgia Tech\\ \texttt{yb.kook@gatech.edu} 
\and
Santosh S. Vempala\\ Georgia Tech\\ \texttt{vempala@gatech.edu}}}
\maketitle
\begin{abstract}
We survey the localization method for proving inequalities in high
dimension, pioneered by Lov\'asz and Simonovits (1993), and its stochastic
extension developed by Eldan (2012). The method has found applications
in a surprisingly wide variety of settings, ranging from its original
motivation in isoperimetric inequalities to optimization, concentration
of measure, and bounding the mixing rate of Markov chains. At heart,
the method converts a given instance of an inequality (for a set or
distribution in high dimension) into a highly structured instance,
often just one-dimensional.

\end{abstract}
\thispagestyle{empty}\pagebreak\tableofcontents{}

\setcounter{page}{0}
\thispagestyle{empty}\newpage{}

\section{Introduction}

High-dimensional inequalities are a rich subject with many motivating
conjectures from a variety of fields. As an example, consider the
following statement: \medskip{}

\emph{Any convex body $\K$ in $\Rn$ with the property that the uniform
distribution over it has zero mean and identity covariance contains
a ball of radius at least $\sqrt{(n+2)/n}$. }

\medskip{}

This statement is a (true) $n$-dimensional inequality with constraints.
Given that there are infinitely many convex bodies, and any convex
body can be made to satisfy the constraints after an affine transformation,
how does one prove such a statement?

As a second example, take the classical isoperimetry of a Gaussian:
the extremal isoperimetric subset for any Gaussian, i.e., one that
minimizes 

\[
\frac{\gamma^{+}(\partial S)}{\min\{\gamma(S),\gamma(\Rn\setminus S)\}}
\]
over all measurable subsets $S$, is given by a halfspace (the numerator
is the boundary measure). The extreme symmetry of the Gaussian allows
proofs that proceed by symmetrization. Consider the following substantial
generalization: \medskip{}

\emph{For any logconcave distribution in any dimension, there exists
a halfspace that gives an asymptotically optimal isoperimetric subset.
}\medskip{}

Here, by \emph{asymptotically optimal}, we mean ``within a universal
constant factor'', and this statement is the famous Kannan--Lov\'asz--Simonovits
(KLS) conjecture~\cite{KLS95isop}. At the moment, the current best
bound is that halfspaces are within $\O(\sqrt{\log n})$ of an optimal
subset, and the method surveyed here has played an important role
in the progress so far. In fact, localization has been a principal
tool for progress or resolution of many well-known inequalities in
high dimension. We first recall some of these inequalities along with
relevant definitions.

\section{A Selection of High-Dimensional Inequalities}

We begin with notation and definitions that will be used throughout.

\subsection{Preliminaries}

We use $\pi,\nu$ or $p$ to indicate a probability measure (equivalently,
a distribution) over $\Rn$, and use the same symbol for a measure
and its corresponding density (with respect to the Lebesgue measure)
when there is no confusion. Unless specified otherwise, $b=\E_{\pi}X$
(or $\mu$) and $\Sigma=\E_{\pi}[(X-b)^{\otimes2}]=\E_{\pi}[(X-b)(X-b)^{\T}]$
(or $\cov\pi$) denote the barycenter and covariance of $\pi$, respectively.

\paragraph{Notation.}

For $t>0$, we reserve $\gamma_{t}$ for the centered Gaussian distribution
$\msf N(0,tI_{n})$ (in $\Rn$) with covariance matrix $tI_{n}$.
The indicator function of a set $S\subseteq\Rn$ is denoted by $\ind_{S}(x):=[x\in S]$,
and $\mu|_{S}$ refers to a distribution $\mu$ truncated to $S$
(i.e., $\mu|_{S}\propto\mu\cdot\ind_{S}$). For two probability measures
$\mu,\pi$, we use $\mu\pi$ to denote the new distribution with density
proportional to $\mu$ times $\pi$.

The notation $a\lesssim b$ means $a=\O(b)$, $a\gtrsim b$ means
$a=\Omega(b)$, and $a\asymp b$ means $a=\O(b)$ and $a=\Omega(b)$
both hold. Lastly, $a=\Otilde(b)$ means $a=O(b\polylog b)$. For
$a,b\in\R$, we denote $a\wedge b:=\min(a,b)$ and $a\vee b:=\max(a,b)$.
For a positive semi-definite matrix $\Sigma$, let $\norm{\Sigma}$
denote the operator norm of $\Sigma$ (i.e., the largest eigenvalue
of $\Sigma$). For two random variables $X$ and $Y$, we use $X\overset{d}{=}Y$
to mean that $X$ and $Y$ follow the same distribution. 

\paragraph{Logconcavity.}

We call a function $f:\Rn\to[0,\infty)$ \emph{logconcave }(or log-concave)
if $-\log f$ is convex in $\Rn$, and a (Borel) probability measure
$\pi$ (or distribution) logconcave if it has a logconcave density
function with respect to the Lebesgue measure. For example, this class
includes uniform distributions over convex bodies, Gaussian distributions
$\msf N(b,\Sigma)$, exponential distributions, and Laplacian distributions.
We assume that any  distributions considered in this work are non-degenerate;
otherwise, we could simply work on a suitable affine subspace. For
$t\geq0$, $\pi$ is called \emph{$t$-strongly logconcave} if $-\log\pi$
is $t$-strongly convex (i.e., $-\log\pi-\frac{t}{2}\,\norm{\cdot}^{2}$
is convex).

A distribution $\pi$ is called \emph{isotropic} if it has zero mean
and identity covariance, i.e., $\E_{\pi}X=0$ and $\E_{\pi}[X^{\otimes2}]=I_{n}$.
One can readily notice that $T(x):\Sigma^{-1/2}(x-b)$ makes $\pi$
isotropic (i.e., the push-forward $T_{\#}\pi$ of $\pi$ via $T$
is isotropic when $\E_{\pi}X=b$ and $\E_{\pi}[(X-b)^{\otimes2}]=\Sigma$). 

\paragraph{Classical results in convex geometry.}

We recall several classical results from convex geometry. The following
fundamental properties of logconcave functions/distributions were
proven by Dinghas, Leindler, and Pr\'ekopa~\cite{Din57,Lei72,Pr73a,Pr73b}.
\begin{thm}
All marginals as well as the distribution function of a logconcave
function are logconcave. The convolution of two logconcave functions
is logconcave.
\end{thm}

We will also need a reverse H\"older inequality (also known as Borell's
lemma \cite{Borell74convex}), where the specific constant in the
RHS is computed in \cite[Theorem 5.22]{LV07geometry}
\begin{lem}
[Reverse H\"older] If $\pi$ is a logconcave distribution over $\Rn$,
then for any $k\geq1$,
\[
\norm X_{L^{k}(\pi)}:=(\E_{\pi}[\norm X^{k}])^{1/k}\leq2k\,\E_{\pi}\norm X=2k\,\norm X_{L^{1}(\pi)}\,.
\]
\end{lem}

We refer interested readers to \cite{LV07geometry,BGV14convex} for
further properties of logconcave functions and measures.

\subsection{Three conjectures in convex geometry}

We review three open problems in asymptotic convex geometry---the
KLS conjecture, the thin-shell conjecture (resolved in July 2025),
and the slicing conjecture (resolved in December 2024). We mention
the excellent lecture notes \cite{KL24isop} and survey \cite{GPT25isotropic}
which describe these conjectures from different perspectives. In particular,
\cite{GPT25isotropic} discusses several consequences of recent breakthroughs
for the asymptotic theory of high-dimensional convex bodies. Early
surveys addressed algorithmic connections, especially to high-dimensional
sampling \cite{VemSurvey,lee2019kannan}. 

\subsubsection{The Kannan-Lov\'asz-Simonovits (KLS) conjecture}
\begin{defn}
[Cheeger and KLS constant] For a probability measure $\pi$ over
$\Rn$, the \emph{Cheeger isoperimetric constant} for $\pi$ is defined
by 
\[
\psi_{\ch}(\pi)\defeq\inf_{S\subset\Rn}\frac{\pi^{+}(\partial S)}{\min\{\pi(S),\pi(S^{c})\}}\,,
\]
where $\pi^{+}(\partial S):=\liminf_{\veps\to0^{+}}\frac{\pi(S+\epsilon B^{n})-\pi(S)}{\veps}$
and $B^{n}$ is the unit Euclidean ball in $\Rn$. The KLS constant
for $\pi$ is simply the inverse of the Cheeger constant: 
\[
\psi_{\kls}(\pi)\defeq\psi_{\ch}^{-1}(\pi)\,.
\]
It will be convenient to only consider \emph{isotropic} logconcave
measures, so we define
\[
\psi_{\ch}(n)\defeq\inf_{\text{iso. }\pi}\psi_{\ch}(\pi)\,,\quad\&\quad\psi_{\kls}(n)=\sup_{\text{iso. }\pi}\psi_{\kls}(\pi)\,,
\]
where both infimum and supremum run over all isotropic logconcave
measures in $\Rn$.
\end{defn}

The KLS conjecture \cite{KLS95isop} posits that these parameters
are bounded by a universal constant, i.e., the boundary measure of
any subset is at least a universal constant times the smaller of the
measure of the subset or its complement, for any logconcave measure. 
\begin{conjecture}
[KLS conjecture] \label{conj:KLS-cheeger} For any logconcave probability
measure $\pi$ over $\Rn$,
\[
\psi_{\kls}(\pi)\asymp\norm{\cov\pi}^{1/2}\,.
\]
Equivalently,
\[
\psi_{\kls}(n)\asymp1\,.
\]
\end{conjecture}

The conjecture was originally motivated by the analysis of the $\bw$
for sampling uniformly from convex bodies and implies the thin-shell
conjecture, which in turn implies the slicing conjecture. This geometric
isoperimetric constant is closely connected to a functional inequality
called the \emph{Poincar\'e inequality.}
\begin{defn}
[Poincar\'e inequality]A probability measure $\pi$ over $\Rn$ is
said to satisfy a \emph{Poincar\'e inequality} with constant $\cpi(\pi)>0$
if for any locally-Lipschitz function $f:\Rn\to\R$, 
\begin{equation}
\var_{\pi}f\leq\cpi(\pi)\,\E_{\pi}[\norm{\nabla f}^{2}]\,,\tag{{\ensuremath{\msf{PI}}}}\label{eq:PI}
\end{equation}
where $\var_{\pi}f=\E_{\pi}[(f-\E_{\pi}f)^{2}]=\E_{\pi}[f^{2}]-(\E_{\pi}f)^{2}$
is the variance of $f$ with respect to $\pi$.
\end{defn}

The next proposition establishes their equivalence for logconcave
measures.
\begin{prop}
[Equivalence between Cheeger and Poincar\'e] The Cheeger implies
the Poincar\'e, and it can be reversed for logconcave distributions;
\[
\frac{1}{4}\leq\frac{\psi_{\kls}^{2}(\pi)}{\cpi(\pi)}\leq9\,,
\]
where the first inequality is due to Cheeger \cite{cheeger1970lower},
and the second is the Buser--Ledoux inequality \cite{buser1982note,Led04spectral}.
\end{prop}

Using this equivalence between \eqref{eq:PI} and the Cheeger constant,
we can state the KLS conjecture in terms of Poincar\'e constants.
\begin{conjecture}
[KLS conjecture via PI] \label{conj:KLS-PI} For any logconcave probability
measure $\pi$ over $\Rn$, it holds that 
\begin{equation}
\norm{\cov\pi}\leq\cpi(\pi)\lesssim\norm{\cov\pi}\,.\tag{\ensuremath{\msf{KLS}}}\label{eq:KLS-PI}
\end{equation}
\end{conjecture}

We note that $\norm{\cov\pi}\leq\cpi(\pi)$ always holds (due to a
test function $(x-\mu)\cdot\theta$ for a suitable unit vector $\theta\in\mbb S^{n-1}$).
Then, the KLS conjecture says that a linear function almost \emph{saturates}
in a sense that $\cpi(\pi)\lesssim\norm{\cov\pi}$. 

As mentioned earlier, we can check that it suffices to focus on \emph{isotropic}
logconcave measures:
\begin{prop}
Let $\cpi(n):=\sup\cpi(\pi)$ where the supremum runs over all isotropic
logconcave distributions $\pi$. Then,
\[
\cpi(\pi)\leq\cpi(n)\,\norm{\cov\pi}\,.
\]
\end{prop}

\begin{proof}
For $\Sigma:=\cov\pi$, let $X\sim\pi$, $Y=\Sigma^{-1/2}\,(X-\E_{\pi}X)$,
and $\nu:=\law Y$. Then, $\nu$ is isotropic and logconcave, and
\begin{align*}
\var_{\pi}f & =\var_{\nu}(f\circ\Sigma^{1/2})\leq\cpi(n)\,\E_{\nu}[\norm{\nabla(f\circ\Sigma^{1/2})}^{2}]\\
 & \leq\cpi(n)\,\E_{\nu}[\norm{\Sigma}\,\norm{\nabla f\circ\Sigma^{1/2}}^{2}]\leq\cpi(n)\,\norm{\Sigma}\,\E_{\pi}[\norm{\nabla f}^{2}]\,.
\end{align*}
The claim follows from the definition of \eqref{eq:PI} for $\pi$.
\end{proof}
The current best bound on $\cpi(n)$ is due to Klartag, established
following a long line of progress \cite{LV24eldan,Chen21almost,KL22Bourgain,Klartag23log}.
\begin{thm}
[\cite{Klartag23log}]$\cpi(n)\lesssim\log n$ (equivalently, $\psi_{\kls}(n)\lesssim\sqrt{\log n}$).
\end{thm}

\subsubsection{The thin-shell conjecture}

We move onto a second conjecture called the \emph{thin-shell conjecture,}
developed by Anttila, Ball, and Perissinaki \cite{anttila2003central},
and Bobkov and Koldobsky \cite{BK03central}. As we will see, the
KLS conjecture implies the thin-shell conjecture.

\begin{wrapfigure}{r}{0.4\textwidth}
\centering
\begin{tikzpicture}[y=0.80pt, x=0.80pt, yscale=-0.250000, xscale=0.250000, inner sep=0pt, outer sep=0pt]   \path[draw=c063cb8,dash pattern=on 8.00pt off 8.00pt,line cap=butt,miter     limit=4.00,draw opacity=0.761,nonzero rule,line width=1.000pt]     (355.9839,452.3622) ellipse (5.8864cm and 5.4603cm);   \path[draw=c063cb8,line cap=butt,miter limit=4.00,draw opacity=0.762,nonzero     rule,line width=1.000pt] (354.5455,452.3622) ellipse (6.4141cm and 6.0934cm);   \path[draw=cff0000,line join=miter,line cap=butt,miter limit=4.00,even odd     rule,line width=1.000pt] (92.0000,424.3622) -- (466.0000,224.3622);   \path[draw=cff0000,line join=miter,line cap=butt,miter limit=4.00,even odd     rule,line width=1.000pt] (466.0000,224.3622) -- (560.0000,282.3622) ..     controls (530.4225,628.8411) and (530.0000,629.8270) .. (530.0000,629.8270) ..     controls (455.3521,679.1228) and (454.0000,679.3622) .. (454.0000,679.3622) --     (108.0000,524.3622);   \path[draw=cff0000,line join=miter,line cap=butt,miter limit=4.00,even odd     rule,line width=1.000pt] (92.0000,424.3622) .. controls (107.4930,525.7706)     and (108.0000,524.3622) .. (108.0000,524.3622);
\end{tikzpicture}
\caption*{\textbf{Figure.} Does a thin shell contain most of the measure?}
\end{wrapfigure}
\begin{defn}
[Thin-shell constant] For an isotropic probability measure $\pi$
over $\Rn$, the thin-shell constant for $\pi$ is 
\[
\sigma_{\ts}^{2}(\pi)=\sigma_{\pi}^{2}\defeq\frac{1}{n}\,\var_{\pi}(\norm{\cdot}^{2})\,.
\]
The \emph{thin-shell constant} refers to 
\[
\sigma_{\ts}^{2}(n)=\sigma_{n}^{2}\defeq\sup_{\text{iso. }\pi}\sigma_{\pi}^{2}\,,
\]
where the supremum runs over all the isotropic logconcave distributions
over $\Rn$.
\end{defn}

We are now ready to state the thin-shell conjecture.
\begin{conjecture}
[Thin-shell conjecture] \label{conj:Thin-shell} The thin-shell constant
is bounded by a universal constant:
\begin{equation}
\sigma_{n}^{2}=\O(1)\,.\tag{\ensuremath{\msf{TS}}}\label{eq:Ts}
\end{equation}
\end{conjecture}

To parse its meaning, let us consider an isotropic Gaussian $X\sim\msf N(0,I_{n})$.
It is well-known from a standard concentration result that $\norm X\in n^{1/2}\pm\O(1)$
takes up most of the measure. In other words, most of the mass is
concentrated around a thin shell of width $\O(1)$. The thin-shell
conjecture essentially asks if this phenomenon generalizes to any
isotropic logconcave distributions.

As mentioned earlier, it can be shown that \eqref{eq:KLS-PI} implies
\eqref{eq:Ts}.
\begin{prop}
[KLS implies thin-shell] $\sigma_{n}^{2}\leq4\cpi(n)\asymp\psi_{\kls}^{2}(n)$.
\end{prop}

\begin{proof}
For isotropic $\pi$ and $f(\cdot)=\norm{\cdot}^{2}$, \eqref{eq:PI}
reads 
\[
\var_{\pi}(\norm{\cdot}^{2})\leq\cpi(\pi)\,\E[\norm{\nabla f}^{2}]=4\cpi(\pi)\,\E[\norm{\cdot}^{2}]=4n\cpi(\pi)\,.
\]
Thus, $\sigma_{\pi}^{2}\leq4\cpi(\pi)$, and taking supremum over
all isotropic logconcave $\pi$, we obtain $\sigma_{n}^{2}\leq4\cpi(\pi)$.
\end{proof}
A much stronger bound was established by Guan \cite{guan2024note}:
$\sigma_{n}\lesssim\log\psi_{\kls}(n)\lesssim\log\log n$. Subsequently,
Klartag and Lehec \cite{klartag2025thin} proved the conjecture. 
\begin{thm}
[\textup{\cite{klartag2025thin}}] $\sigma_{n}=O(1)$.
\end{thm}

\subsubsection{Bourgain's slicing conjecture}

The slicing conjecture \cite{Bourgain1986}, also known as Bourgain's
hyperplane conjecture, has been a driving force for much of the development
of modern convex geometry. This geometric problem asks whether for
any convex body $\K\subset\Rn$ of volume one, there exists a hyperplane
$H\subset\Rn$ such that 
\[
\vol_{n-1}(\K\cap H)\gtrsim1.
\]
\begin{wrapfigure}{r}{0.4\textwidth}
\centering
\begin{tikzpicture}[y=0.80pt, x=0.80pt, yscale=-0.3000000, xscale=0.3000000, inner sep=0pt, outer sep=0pt]   \path[draw=c063cb8,line join=miter,line cap=butt,miter limit=4.00,draw     opacity=0.761,even odd rule,line width=1.000pt] (89.3829,473.8448) --     (175.3188,681.7452) -- (495.0000,557.3622);   \path[draw=c063cb8,line join=miter,line cap=butt,miter limit=4.00,draw     opacity=0.761,even odd rule,line width=1.000pt] (495.0000,557.3622) --     (537.9999,447.9389);   \path[draw=c063cb8,line join=miter,line cap=butt,miter limit=4.00,draw     opacity=0.761,even odd rule,line width=1.000pt] (89.3829,473.8448) --     (242.2085,331.3622);   \path[draw=c063cb8,line join=miter,line cap=butt,miter limit=4.00,draw     opacity=0.761,even odd rule,line width=1.000pt] (242.2085,331.3622) ..     controls (533.8042,447.9601) and (537.9999,447.9389) .. (537.9999,447.9389);   \path[cm={{0.7791,-0.6269,0.69866,0.71545,(0.0,0.0)}},draw=cff0000,line     cap=butt,miter limit=4.00,fill opacity=0.761,nonzero rule,line     width=1.000pt,rounded corners=0.0000cm] (-166.3427,364.3955) rectangle     (-141.6752,757.1052);   \path[fill=black,line join=miter,line cap=butt,line width=1.000pt]     (200.0334,723.5007) node[above right] (text5305) {$\mathrm{vol}(S) \gtrsim \mathrm{vol}(\K)$?};   \path[fill=black,line join=miter,line cap=butt,line width=1.000pt]     (446.5659,395.5517) node[above right] (text4142) {$\K$};   \path[fill=cff0000,fill opacity=0.316] (265.2715,507.4381) .. controls     (211.4740,452.3657) and (166.9055,406.6835) .. (166.2304,405.9221) --     (165.0030,404.5378) -- (173.2521,396.8505) .. controls (177.7892,392.6224) and     (181.5887,389.2011) .. (181.6955,389.2475) .. controls (182.0090,389.3836) and     (386.3381,598.6835) .. (386.2313,598.7591) .. controls (385.9773,598.9388) and     (363.5273,607.5756) .. (363.3213,607.5729) .. controls (363.1915,607.5709) and     (319.0691,562.5105) .. (265.2715,507.4381) -- cycle;   \path[fill=black,line join=miter,line cap=butt,line width=1.000pt]     (254.6037,566.1661) node[above right] (text4142-9) {$S$};
\end{tikzpicture}
\caption*{\textbf{Figure.} Is there a ``nice'' hyperplane?}
\end{wrapfigure}In other words, it asks if a convex body of unit volume always has
a constant volume $(n-1)$-dimensional slice. 
\begin{conjecture}
[Slicing conjecture] Consider the slicing constant defined by 
\[
\frac{1}{L_{n}}:=\inf_{\K\subset\Rn}\sup_{H}\vol_{n-1}(\K\cap H)\,,
\]
where the infimum runs over all convex bodies $\K$ of volume one,
and the supremum runs over all hyperplanes $H$. Then, 
\begin{equation}
L_{n}=\O(1)\,.\tag{\ensuremath{\msf{Sc}}}\label{eq:slicing}
\end{equation}
\end{conjecture}

It turns out that \eqref{eq:Ts} implies \eqref{eq:slicing}: $L_{n}\lesssim\sigma_{n}$
\cite{eldan2011approximately}. Following Guan's analysis \cite{guan2024note},
the conjecture was resolved by Klartag and Lehec \cite{klartag2024affirmativeresolutionbourgainsslicing},
and Bizeul \cite{bizeul2025slicing}.
\begin{thm}
[\cite{klartag2024affirmativeresolutionbourgainsslicing,bizeul2025slicing}]
$L_{n}=\O(1)$.
\end{thm}

Therefore, summarizing the status of the three conjectures, we have
\[
\text{\fbox{\ \ensuremath{O(1)=L_{n}\asymp\sigma_{n}\leq\log\psi_{\kls}\asymp\log\cpi(n)\lesssim\log\log n}\ }}
\]

\paragraph{Equivalent formulations.}

One can formulate \eqref{eq:slicing} in terms of the maximum of its
density. Let $\pi_{S}$ denote the uniform distribution over a compact
set $S$. 

Let $\K$ be a convex body for which
\[
\vol\K=1\,,\quad\cov\pi_{\K}=L_{\K}^{2}I_{n}\ \text{for some }L_{\K}>0\,.
\]
Clearly, $\pi_{L_{\K}^{-1}\K}$ is isotropic, and its density (over
its support) is simply 
\[
\frac{1}{\vol(L_{\K}^{-1}\K)}=L_{\K}^{n}\,.
\]
Then, \eqref{eq:slicing} is equivalent to
\[
\bar{L}_{n}:=\sup_{\K}L_{\K}=\sup_{\K}\sup_{\Rn}\bpar{\frac{1}{\vol(L_{\K}^{-1}\K)}}^{1/n}=\O(1)\,,
\]
where the supremum runs over all isotropic convex bodies. Its natural
logconcave generalization would be 
\[
\hat{L}_{n}:=\sup_{\pi}\sup_{\Rn}\pi^{1/n}=\O(1)\,,
\]
where the supremum runs over all isotropic logconcave distributions.

We close this section with a classical result in convex geometry (see
\cite{KM22slicing}).
\begin{thm}
\label{thm:slicing-equiv} The following are equivalent:
\begin{itemize}
\item $L_{n}=\O(1)$.
\item $\bar{L}_{n}=\O(1)$.
\item $\hat{L}_{n}=\O(1)$.
\end{itemize}
\end{thm}

\subsection{Large deviation inequalities}

The following concentration result due to Paouris has been influential
and useful in asymptotic convex geometry \cite{Paouris2006}.
\begin{thm}
\label{thm:Paouris} For any isotropic logconcave distribution $\pi$
in $\Rn$, there exists a universal constant $c>0$ such that 
\[
\P_{\pi}(\norm X_{2}-\sqrt{n}\geq t\sqrt{n})\le\exp(-ctn^{1/2})\quad\text{for any }t\geq1\,.
\]
\end{thm}

The next theorem is a common generalization of the above and Levy
concentration of Lipschitz functions, due to Lee and Vempala \cite{LV24eldan}.
\begin{thm}
For any $L$-Lipschitz function $g$ in $\Rn$ and any isotropic logconcave
measure $\pi$, and any $t\ge1,$ there exists a universal constant
$c>0$ such that
\[
\P_{\pi}(|g(X)-\bar{g}|\geq Lt)\leq\exp\bpar{-\frac{ct^{2}}{t+\sqrt{n}}}\,,
\]
where $\bar{g}$ is the median or mean of $g(x)$ for $x\sim\pi$.
\end{thm}

\subsection{Strongly logconcave isoperimetry}

For $\D\pi\propto\exp(-V)\,\D x$, the probability measure $\pi$
is $t$-strongly logconcave if $V-\frac{t}{2}\,\norm{\cdot}^{2}$
is convex. This is a strictly smaller class than logconcave distributions
and enjoys stronger properties. For instance, it is a classical result
that the Poincar\'e constant of $t$-strongly logconcave distributions
is bounded by $1/t$. This can be seen as a consequence of the \emph{Brascamp--Lieb
inequality}: for a probability measure $\pi\propto e^{-V}$ with strictly
convex $V$ and any $f:\Rn\to\R$,
\begin{equation}
\var_{\pi}f\leq\E_{\pi}\inner{\nabla f,(\hess V)^{-1}\,\nabla f}\,.\label{eq:Brascamp-Lieb}
\end{equation}
This $1/t$-bound can also be explained by the fact that the Poincar\'e
inequality is implied by the stronger log-Sobolev inequality (i.e.,
$\cpi\leq\clsi$).
\begin{defn}
[Log-Sobolev constant] A probability measure $\pi$ over $\R^{n}$
is said to satisfy a \emph{log-Sobolev inequality} (LSI) with constant
$\clsi(\pi)>0$ if for any locally-Lipschitz function $f:\Rn\to\R$,
\begin{equation}
\ent_{\pi}(f^{2})\leq2\clsi(\pi)\,\E_{\pi}[\norm{\nabla f}^{2}]\,,\tag{\ensuremath{\msf{LSI}}}\label{eq:lsi}
\end{equation}
where $\ent_{\pi}(f^{2}):=\E_{\pi}[f^{2}\log f^{2}]-\E_{\pi}[f^{2}]\log\E_{\pi}[f^{2}]$.
\end{defn}

It is also classical that $\clsi(\pi)\leq1/t$ for any $t$-strongly
logconcave distributions $\pi$, and this can be established in multiple
ways (e.g., Caffarelli's contraction theorem \cite{Caffarelli00monotonicity}
or the Bakry--\'Emery condition \cite{BGL14analysis}). Below, we
demonstrate another way to see this through connections to a logarithmic-Cheeger
inequality.

Since $\cpi(\pi)\leq\clsi(\pi)\leq1/t$ for $t$-strongly logconcave
distributions $\pi$, the bound $\cpi(\pi)\leq1/t$ becomes useless
as $t\to0$. We have already seen that $\norm{\cov\pi}\leq\cpi(\pi)$
for any probability measures $\pi$, and thus $\norm{\cov\pi}\leq\cpi(\pi)\leq1/t$
for $t$-strongly logconcave $\pi$. Hence, the KLS conjecture essentially
asks for a ``refined'' bound for $0$-strongly (i.e., \emph{all})
logconcave distributions, asserting an $\O(\norm{\cov\pi})$ bound
instead of the $1/t$-bound.

\paragraph{LSI for logconcave distributions with bounded support.}

Although any logconcave distribution has a finite Poincar\'e constant
(e.g., $\cpi(\pi)\lesssim\norm{\cov\pi}\log n$ due to the KLS bound),
this is not the case for \eqref{eq:lsi}. It is well-known that having
a finite LSI constant implies Gaussian-like tail decay, which is clearly
false, for example, for exponential distributions. Nonetheless, some
results are  known for logconcave distributions with finite support.
For instance, Frieze and Kannan established that $\clsi(\pi)\lesssim D^{2}$
\cite{FK99lsi} for any logconcave distributions with support of diameter
$D$, and this was improved to $\O(D)$ for\emph{ }isotropic logconcave
distributions \cite{LV24eldan}. Later, these two bounds were unified
as $\clsi(\pi)\lesssim D\cpi^{1/2}(\pi)$ \cite[Theorem 1.5]{KV25faster}
(up to logarithmic factors), since $\cpi(\pi)\lesssim\norm{\cov\pi}\log n\leq D^{2}$. 

\paragraph{Log-Cheeger and LSI.}

Recall that the Cheeger inequality is an equivalent ``geometric''
formulation of \eqref{eq:PI} for logconcave distributions. It turns
out that there exists a similar equivalence for \eqref{eq:lsi}. To
see this, we recall the \emph{logarithmic-Cheeger inequality}.
\begin{defn}
The \emph{log-Cheeger constant} $\clch(\pi)$ of a probability measure
$\pi$ on $\R^{n}$ is defined as 
\[
\clch(\pi)=\inf_{S\subset\R^{n},\pi(S)\le\frac{1}{2}}\frac{\pi(\partial S)}{\pi(S)\sqrt{\log(1/\pi(S))}}\,,
\]
where the infimum runs over every measurable set $S$ of measure at
most $1/2$.
\end{defn}

It is obvious that $\clch\leq\cch$ due to $\sqrt{\log(1/\pi(S))}$,
and Ledoux showed that the log-Cheeger inequality is equivalent to
\eqref{eq:lsi} through $\clsi(\pi)\asymp\clch^{-2}(\pi)$ \cite{ledoux1994simple},
just as $\cpi(\pi)\asymp\cch^{-2}(\pi)$ for any logconcave distribution
$\pi$. Ledoux also showed that $\clch(\pi)\gtrsim t^{1/2}$ for any
$t$-strongly logconcave $\pi$ \cite{Ledoux1999}, which implies
that $\clsi(\pi)\asymp\clch^{-2}(\pi)\lesssim1/t$ as desired.
\begin{thm}
\label{thm:strongly_logconcave} For any $t$-strongly logconcave
distribution $\pi$, and any measurable set $S$ with $\pi(S)\le1/2$,
we have 
\[
\frac{\pi(\partial S)}{\pi(S)\sqrt{\log(1/\pi(S))}}\gtrsim\sqrt{t}\,.
\]
\end{thm}

We note that this result can be proven via the classical localization
technique in \S\ref{sec:The-Classical-Method}, as shown in \cite[Theorem 47]{LV24eldan}.

\subsection{Carbery--Wright anti-concentration}

The following is an illustrative special case of the Carbery--Wright
theorem \cite{carbery2001distributional}. Their general theorem holds
for complex vector-valued polynomials and allows for weaker and more
general moment conditions. 
\begin{thm}
For a polynomial $p:\K\rightarrow\R$ of degree $d$ defined over
a logconcave distribution $\pi$ in $\R^{n}$ with $\var_{\pi}(p(X))\ge1$,
for any $t\in\R$, 
\[
\P_{\pi}(|p(X)-t|\le\veps)\lesssim d\veps^{1/d}\,.
\]
\end{thm}

\subsection{Certifiable hypercontractivity}
\begin{conjecture}
Let $f:\R^{n}\rightarrow\R$ be a degree-$k$ polynomial defined as
$f(x)=\langle(1,x)^{\otimes k},A\rangle$ where $A$ is a $k$-th
order tensor of real coefficients. Then, for an isotropic logconcave
distribution $\pi$, there exists a universal constant $C>0$ such
that 
\[
\var_{\pi}f\le(Ck)^{2k}\,\|A\|_{\hs}^{2}\,.
\]
\end{conjecture}

It was shown by Kothari and Steinhardt \cite{kothari2017betteragnosticclusteringrelaxed}
that this conjecture is implied by the KLS conjecture. They also developed
several algorithmic consequences.

\subsection{Spectral gap and entropy }

The mixing rate of a Markov chain can be bounded using the \emph{spectral
gap} of its transition operator. Doing so is typically easier for
distributions with additional structure. Chen and Eldan \cite{chen2022localization}
proposed localization as an approach to reduce bounding the spectral
gap for general families of target distributions for sampling by Markov
chains to specialized families. Using this, they improved the analysis
of Markov chains for several well-known problems, notably in both
discrete and continuous state spaces. Their approach has since been
employed by other authors to simplify and refine such analyses. Since
this (active) topic deserves a separate and detailed discussion, we
do not include it in the survey.

\section{The Classical Method\label{sec:The-Classical-Method}}

\inputencoding{latin9}Here we present the classical or standard localization
method, introduced by Lov\'asz and Simonovits \cite{LS93random}
and refined by Kannan, Lov\'asz, and Simonovits \cite{KLS95isop}.
The method can be captured as a concise general lemma. 

\subsection{The main idea: bisection}
\begin{lem}[Localization Lemma \cite{KLS95isop}]
\label{lem:LOCALIZATION}Let $g,h:\R^{n}\rightarrow\R$ be lower
semi-continuous integrable functions such that 
\[
\int g(x)\,\D x>0\quad\text{and}\quad\int h(x)\,\D x>0\,.
\]
 Then, there exist two points $a,b\in\R^{n}$ and an affine function
$\ell:[0,1]\rightarrow\R_{+}$ such that 
\[
\int_{0}^{1}\ell(t)^{n-1}\,g\bpar{(1-t)a+tb}\,\D t>0\quad\mbox{and}\quad\int_{0}^{1}\ell(t)^{n-1}\,h\bpar{(1-t)a+tb}\,\D t>0\,.
\]
\end{lem}

The points $a,b$ represent the endpoints of an interval and one may
think of $\ell(t)^{n-1}$ as proportional to the cross-sectional area
of an infinitesimal cone with $t$ going along its axis. The lemma
says if $g,h$ have positive integrals in $\R^{n},$ then there is
a cone truncated at $a$ and $b$, so that the integrals of $g$ and
$h$ over it are positive. Without loss of generality, we can assume
that $a,b$ are in the union of the supports of $g$ and $h$.

\begin{wrapfigure}{r}{0.4\textwidth}
\centering
\begin{tikzpicture}[y=0.80pt, x=0.80pt, yscale=-0.250000, xscale=0.250000, inner sep=0pt, outer sep=0pt]   \path[draw=black,fill=cff6666,line cap=butt,miter limit=4.00,fill     opacity=0.761,nonzero rule,line width=1.000pt] (120.7860,427.2994) ellipse     (0.4222cm and 1.0081cm);   \path[draw=black,fill=cff6666,line cap=butt,miter limit=4.00,fill     opacity=0.758,nonzero rule,line width=1.0pt] (564.1912,414.3622) ellipse     (1.3706cm and 2.5744cm);   \path[draw=black,line join=miter,line cap=butt,miter limit=4.00,even odd     rule,line width=1pt] (119.7851,425.6119) .. controls (564.1911,416.6121)     and (562.1893,416.6121) .. (562.1893,416.6121);   \path[draw=black,line join=miter,line cap=butt,miter limit=4.00,even odd     rule,line width=1pt] (119.5685,391.8627) .. controls (563.9745,324.3643)     and (560.1875,323.2393) .. (560.1875,323.2393);   \path[draw=black,line join=miter,line cap=butt,miter limit=4.00,even odd     rule,line width=1pt] (120.0068,462.8607) -- (561.7147,505.1406);
\end{tikzpicture}
\caption*{Truncated cone}
\end{wrapfigure}
\begin{proof}
Choose a continuous integrable function $\gamma:\Rn\to(0,\infty)$
such that 
\[
\int\gamma(x)\,\D x<\frac{1}{2}\min\Bbrace{\int g(x)\,\D x,\int h(x)\,\D x}\,.
\]
It is enough to prove the lemma with $\ge0$ in place of $>0$. Indeed,
if the weak form holds for $g-\gamma$ and $h-\gamma$, then for some
$a,b\in\Rn$ and affine $\ell:[0,1]\to\R_{+}$, 
\[
\int_{0}^{1}\ell(t)^{n-1}(g-\gamma)\bpar{(1-t)\,a+tb}\,\D t\ge0,\qquad\int_{0}^{1}\ell(t)^{n-1}(h-\gamma)\bpar{(1-t)\,a+tb}\,\D t\ge0\,.
\]
Since $\gamma>0$ everywhere, this implies 
\[
\int_{0}^{1}\ell(t)^{n-1}g\bpar{(1-t)\,a+tb}\,\D t\ge\int_{0}^{1}\ell(t)^{n-1}\gamma\bpar{(1-t)\,a+tb}\,\D t>0\,,
\]
and similarly for $h$. Also, every lower semi-continuous integrable
function on $\Rn$ is the pointwise increasing limit of continuous
functions, so it suffices to focus continuous $g$ and $h$.

Fix a large closed ball $\K_{0}\subset\Rn$ such that 
\[
\int_{\K_{0}}g(x)\,\D x>0\quad\text{and}\quad\int_{\K_{0}}h(x)\,\D x>0\,.
\]
Let $A_{ijrs}:=\{x\in\Rn:x_{i}=r,x_{j}=s\}$ for $1\leq i<j\leq n$
and $r,s\in\mbb Q$, and enumerate its countable collection as $\{A_{k}\}_{k\geq1}$.
We now define a decreasing sequence of convex bodies $\K_{0}\supset\K_{1}\supset\cdots$
such that for every $m$, 
\[
\int_{\K_{m}}g(x)\,\D x\ge0\quad\text{and}\quad\int_{\K_{m}}h(x)\,\D x\ge0\,.
\]
When $\K_{m}$ has been defined, if $A_{m}\cap\inter\K_{m}\neq\emptyset$,
set $\K_{m+1}:=\K_{m}$. Otherwise, rotate a halfspace $H$ through
$A_{m}$; the function 
\[
H\mapsto\int_{\K_{m}\cap H}g(x)\,\D x-\int_{\K_{m}\cap H^{c}}g(x)\,\D x
\]
changes continuously during the rotation and changes sign after a
half-turn. Thus, there is a \emph{bisecting halfspace} $H_{m}$ such
that 
\[
\int_{\K_{m}\cap H_{m}}g(x)\,\D x=\int_{\K_{m}\cap H_{m}^{c}}g(x)\,\D x=\frac{1}{2}\int_{\K_{m}}g(x)\,\D x\,.
\]
Clearly, either $\int_{\K_{m}\cap H_{m}}h\geq0$ or $\int_{\K_{m}\cap H_{m}^{c}}h\geq0$
holds, and we define $\K_{m+1}$ to be the domain that has non-negative
$h$-integral.

Let $\K:=\cap_{m=0}^{\infty}\K_{m}$. We claim that $\K$ is of the
form $\overline{ab}$ (i.e., a line segment joining $a,b\in\Rn$).
If not (i.e., $\dim\K\geq2$), for some pair $1\le i<j\le n$, the
projection of $\K$ onto the $(x_{i},x_{j})$-plane has nonempty interior.
Choose a rational point $(r,s)$ in that interior. Then the affine
subspace $A:=\{x\in\Rn:x_{i}=r,x_{j}=s\}$ meets (the relative) $\inter\K$.
Let $m$ be such that $A_{m}=A$. Then, the hyperplane boundary of
$H_{m}$ intersects $\inter\K$, and properly separates points of
$\K$. However, $\K\subset\K_{m+1}\subset\K_{m}\cap H_{m}$ (or $\K_{m}\cap H_{m}^{c}$),
so $\K$ would have to lie entirely in one of the two closed halfspaces.
This contradiction leads to $\dim\K\le1$.

If $a=b$ (so $\K$ is a point), then by continuity, 
\[
\frac{1}{\vol\K_{m}}\int_{\K_{m}}g(x)\,\D x\to g(u)\,,\qquad\frac{1}{\vol\K_{m}}\int_{\K_{m}}h(x)\,\D x\to h(u)\,.
\]
The LHS are nonnegative by construction, hence $g(u),h(u)\geq0$.
Taking $\ell\equiv1$, we are done.

If $a\neq b$, then we may assume $a=0$ and $b=e_{1}$ (by scaling
and translation). For a slice $Z_{t}:=\{x\in\Rn:x_{1}=t\}$ (here
$t\in\R$) and $m\in\Z_{\geq0}$, consider the following sequence
of functions over $[0,1]$: 
\[
\psi_{m}(t):=\bpar{\frac{\vol_{n-1}(\K_{m}\cap Z_{t})}{\vol_{n}\K_{m}}}^{\frac{1}{n-1}}\,.
\]
Note that $\int\psi_{m}(t)^{n-1}\,\D t=1$. For $\alpha_{m}:=\min\{x_{1}:x\in\K_{m}\}$
and $\beta_{m}:=\max\{x_{1}:x\in\K_{m}\}$, we have $\alpha_{m}\le0,\beta_{m}\ge0$
and $\alpha_{m}\rightarrow0,\beta_{m}\rightarrow1$. Moreover, by
the Brunn--Minkowski inequality $\psi_{i}$ is non-negative concave
on its support interval. Then we can take a subsequence of indices
$m$ for which $\psi_{m}$ converges to a limit. To see this, note
that $\psi_{m}$ is uniformly equicontinuous in every interval $[s,t]$
with $0<s<t<1$, and by the Arzela--Ascoli theorem, the sequence
has a pointwise convergent subsequence. The limit $\psi$ is also
non-negative concave with the same integral, i.e., $\int\psi(t)^{n-1}\,\D t=1$.

Note that for $x=(y,t)\in\R^{n-1}\times\R$,
\[
\frac{1}{\vol\K_{m}}\int_{\K_{m}}g(x)\,\D x=\int_{\alpha_{i}}^{\beta_{i}}\Bpar{\frac{1}{\vol_{n-1}(\K_{m}\cap Z_{t})}\int_{\K_{m}\cap Z_{t}}g(y,t)\,\D y}\,\psi_{m}(t)^{n-1}\,\D t\,.
\]
Since the LHS is nonnegative, and the RHS approaches $\int_{0}^{1}g(0,t)\,\psi(t)^{n-1}\,\D t$,
we obtain that

\[
\int_{0}^{1}\psi(t)^{n-1}\,g(te_{1})\,\D t\ge0\quad\text{and}\quad\int_{0}^{1}\psi(t)^{n-1}\,h(te_{1})\,\D t\ge0\,.
\]

The last part of the proof is to argue that we can take $\psi$ to
be a nonnegative linear function, reducing the extremal functions
to a small parameter family (namely, $a,b,$ and the values of the
function at $a$ and $b$). To do this consider a $\psi$ that has
(i) minimal support $\norm{a-b}$; WLOG assume $a=0,b=e_{1}$, and
(ii) is linear in the segment $[\alpha,\beta]$ for $0\le\alpha\le\beta\le1$
with $\beta-\alpha$ is maximal. It will be useful to view $\psi$
as defining a convex body $\K'$ whose slices are simplices:
\[
0\le t\le1,\quad x_{1},\ldots,x_{n}\ge0,\quad x_{2}+\ldots+x_{n}\le\psi(t)\,.
\]
In words, the slice of $\K'$ at $x_{1}=t$ is the standard simplex
given by the convex hull of 
\[
\{te_{1},te_{1}+\psi(t)\,e_{2},\ldots,te_{1}+\psi(t)\,e_{i},\ldots,te_{1}+\psi(t)\,e_{n}\}\,.
\]
The slice at $x_{1}=t$ has volume $\psi(t)^{n-1}/(n-1)!$. Hence,
if we define $\hat{g}(x)=g(x_{1}e_{1})$ and $\hat{h}(x)=h(x_{1}e_{1})$,
then 
\[
\int_{\K'}\hat{g}=\frac{1}{(n-1)!}\int_{0}^{1}\psi(t)^{n-1}\,g(te_{1})\,\D t\ge0
\]
and similarly for $\int_{\K'}\hat{h}\ge0$. WLOG we can assume that
$\int_{\K'}\hat{g}=0$.

\begin{figure}
\centering{}\includegraphics[width=0.5\textwidth]{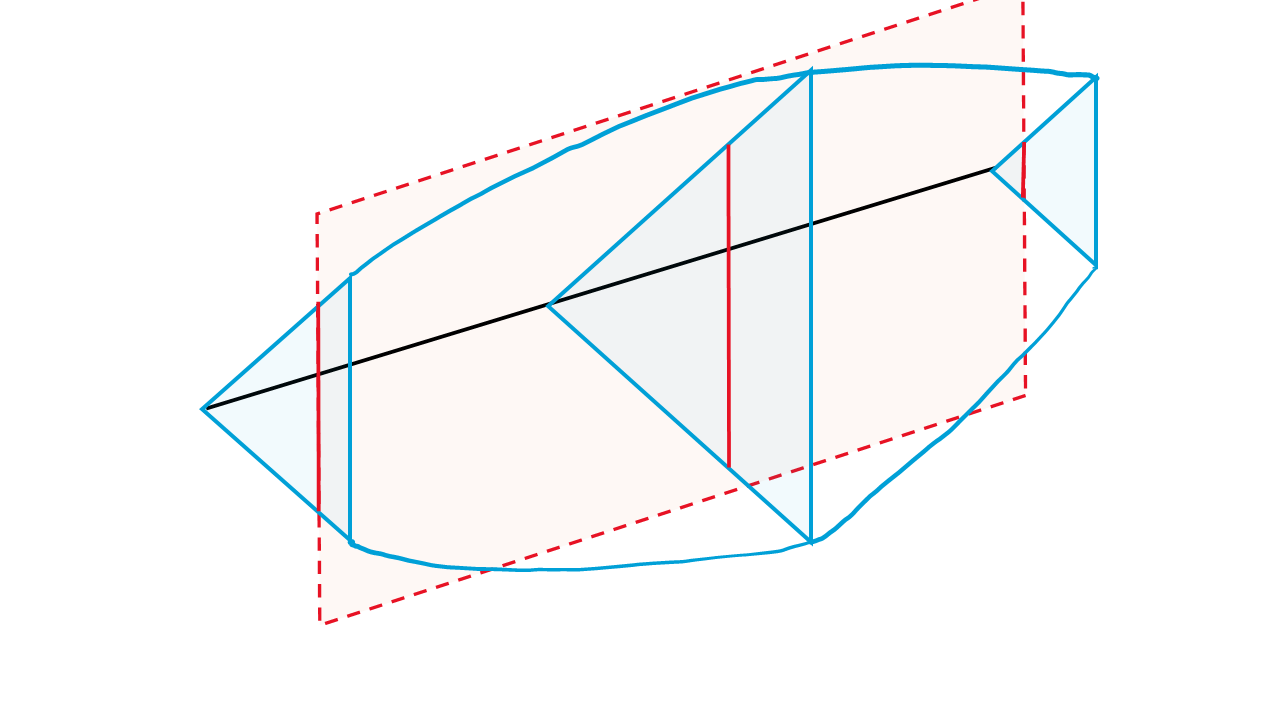}\caption{Going to an affine needle}

\end{figure}

The overall proof idea is to subdivide $\K'$ with a hyperplane in
such a way that one of the resulting halfspaces satisfies the same
two inequalities but violates either (i) or (ii). To choose this hyperplane,
we first fix an $(n-2)$-dimensional subspace $A=\left\{ x:x_{1}=\sigma,\,x_{2}+\ldots+x_{n}=\tau\right\} $
where $0<\sigma<1$ and $\tau>0$. The hyperplane will be chosen to
contain $A$ and bisect the function $\hat{g}$ so that $\int_{L}\hat{g}=\int_{L'}\hat{g}=0$
in both resulting convex bodies $L$ and $L'$. Suppose $\int_{L}h\ge0$.
Then we can infer that $L$ must intersect both $Z_{0}$ and $Z_{1}$,
otherwise the function $\psi$ restricted to $L$ contradicts the
(i), the minimality of the support of $\psi$. 

We now consider 3 cases:
\begin{enumerate}
\item $\psi(0)=\psi(1)=0$. For this case, let $\sigma=1/2$. Then $L$
must contain the segment $[0,1]\,e_{1}$, for any $\tau>0$, else
it contradicts (i). Then there is a linear function $\ell(t)$ with
$\ell(0),\ell(1)\ge0,\ell(1/2)=\tau$ such that the function $\psi$
restricted to $L$ is $\psi_{L}(t)=\min\{\psi(t),\ell(t)\}$. But
then as $\tau\rightarrow0$, $\ell(t)$ tends to zero in the interval
$[0,1]$ and so $\psi_{L}$ is linear in a larger sub-interval tending
to $[0,1]$ contradicting (ii).
\item $\psi(0)=0,\psi(1)>0$. For this case, consider $\tau=\psi(1)\,\sigma$.
Now if $L$ does not contain $[0,1]\,e_{1}$, then $\psi_{L}(0)=\psi_{L}(1)=0$,
taking us back to the first case. Otherwise, $\psi_{L}$ is linear
in the interval $[\sigma,1]$ and hence for small enough $\sigma$
this again contradicts (ii).
\item $\psi(0),\psi(1)>0$. Consider the family of continuous, nonnegative,
convex functions $\eta:[0,1]\to\R_{+}$ satisfying $\eta(t)\le\psi(t)$
for all $t\in[0,1]$. Let 
\[
\K_{\eta}=\K'\cap\{x:x_{2}+\ldots+x_{n}\ge\eta(x_{1})\}=\{x\in\R_{+}:x_{1}\in[0,1],\,\eta(x_{1})\le x_{2}+\ldots+x_{n}\le\psi(x_{1})\}\,.
\]
Then we choose $\eta$ so that (iii) $\int_{0}^{1}\eta(t)\,\D t$
is maximal while still satisfying 
\[
\int_{\K_{\eta}}\hat{g}=0\,,\quad\int_{\K_{\eta}}\hat{h}\ge0\,.
\]
Then, we can assume that $\eta(0)<\psi(0),\eta(1)<\psi(1)$. If not,
we consider the function $\psi-\eta$ and this takes us back to one
of the previous cases. Now let $(\sigma,\tau)$ be the intersection
of the segments $[(0,\eta(0)),(1,\psi(1))]$ and $[(0,\psi(0)),(1,\eta(1))]$.
Let $M,M'$ be the partition of $\K_{\eta}$ such that $M$ does not
contain $[0,1]\,e_{1}$ and
\[
\int_{M}\hat{g}=\int_{M'}\hat{g}=0\,.
\]
Then, if $\int_{M}\hat{h}\ge0$, this contradicts (iii), the maximality
of $\eta$. Hence, $\int_{M}\hat{h}<0$ and the truncated cone $\K'\setminus M$
satisfies 
\[
\int_{\K'\setminus M}\hat{g}=\int_{\K'}\hat{g}-\int_{M}\hat{g}=0\qquad\int_{\K'\setminus M}\hat{h}=\int_{\K'}\hat{h}-\int_{M}\hat{h}\ge0\,,
\]
which contradicts (ii), the maximality of the linear subinterval of
$\psi$. 
\end{enumerate}
This completes the proof.
\end{proof}
The origin of this method, and in particular, the bisection idea,
can be traced back to a paper by Payne and Weinberger \cite{payne1960optimal}.
Localization has been extended to the Riemannian setting \cite{klartag2017needle}. 

Next we give alternative versions of Lemma~\ref{lem:LOCALIZATION}
which are sometimes more convenient in applications.

The first is in terms of single inequality on products of integrals.
Note that an $n$-dimensional needle $N$ is specified by a pair of
points $a,b\in\R^{n}$ and a linear function $\ell:[0,1]\rightarrow\R_{+}$.
The integral over a needle is defined as 
\[
\int_{N}f=\int_{0}^{1}f\bpar{(1-t)a+tb}\,\ell(t)^{n-1}\,\D t\,.
\]

\begin{lem}[Localization: product form]
\cite{KLS95isop} Let $f_{1},f_{2},f_{3},f_{4}:\R^{n}\rightarrow\R_{+}$
be nonnegative, continuous functions and $\alpha,\beta>0$. Then,
the following are equivalent:
\end{lem}

\begin{enumerate}
\item For every convex body $\K\subset\R^{n},$ we have
\[
\left(\int_{\K}f_{1}\right)^{\alpha}\left(\int_{\K}f_{2}\right)^{\beta}\le\left(\int_{\K}f_{3}\right)^{\alpha}\left(\int_{\K}f_{4}\right)^{\beta}\,.
\]
\item For every needle $N=([a,b],\ell)$ in $\R^{n}$, we have
\[
\left(\int_{N}f_{1}\right)^{\alpha}\left(\int_{N}f_{2}\right)^{\beta}\le\left(\int_{N}f_{3}\right)^{\alpha}\left(\int_{N}f_{4}\right)^{\beta}\,.
\]
\end{enumerate}
The next version is in terms of exponential needles, i.e., we replace
the linear functions obtained in the one-dimensional case in Lemma~\ref{lem:LOCALIZATION}
with truncated exponentials. 
\begin{lem}[Localization: exponential needles]
\label{lem:localization_exp_needles} Let $f_{1},f_{2},f_{3},f_{4}:\R^{n}\rightarrow\R_{+}$
be nonnegative, continuous functions and $\alpha,\beta>0$. Then,
the following are equivalent:
\end{lem}

\begin{enumerate}
\item For every logconcave function $F:\R^{n}\rightarrow\R_{+}$, we have
\[
\left(\int_{\R^{n}}F(x)\,f_{1}(x)\,\D x\right)^{\alpha}\left(\int_{\R^{n}}F(x)\,f_{2}(x)\,\D x\right)^{\beta}\le\left(\int_{\R^{n}}F(x)\,f_{3}(x)\,\D x\right)^{\alpha}\left(\int_{\R^{n}}F(x)\,f_{4}(x)\,\D x\right)^{\beta}\,.
\]
\item For every interval $a,b\in\R^{n}$ and every $\gamma\in\R$, we have
\[
\left(\int_{0}^{1}e^{\gamma t}f_{1}(t)\,\D t\right)^{\alpha}\left(\int_{0}^{1}e^{\gamma t}f_{2}(t)\,\D t\right)^{\beta}\le\left(\int_{0}^{1}e^{\gamma t}f_{3}(t)\,\D t\right)^{\alpha}\left(\int_{0}^{1}e^{\gamma t}f_{4}(t)\,\D t\right)^{\beta}\,
\]
where $f_{i}(t)=f_{i}((1-t)a+tb)$. 
\end{enumerate}
\begin{rem}
\label{rem:one-eq-ineq-localization}We can also write the exponential
needle version of localization in terms of two inequalities or one
equality and one inequality: For every logconcave $\ensuremath{F:\R\rightarrow\R_{+}}$
\[
\int_{a}^{b}Ff_{1}=0\,,\qquad\int_{a}^{b}Ff_{2}>0
\]
if and only if for every subinterval $\ensuremath{[a',b']\subseteq[a,b]}$,
and every $\gamma\in\R$, 
\[
\int_{a'}^{b'}e^{\gamma t}f_{1}=0\,,\qquad\int_{a'}^{b'}e^{\gamma t}f_{2}>0\,.
\]

We conclude this section with a nice interpretation of the localization
lemma by Fradelizi and Guedon. They also give a version that extends
localization to multiple inequalities.
\end{rem}

\begin{thm}[Reformulated Localization Lemma \cite{fradelizi2004extreme}]
Let $\K$ be a compact convex set in $\Rn$ and $f$ be an upper
semi-continuous function. Let $P_{f}$ be the set of logconcave distributions
$\mu$ supported by $\K$ satisfying $\int f\,\D\mu\geq0$. The set
of extreme points of $\text{conv}\,P_{f}$ is exactly:
\end{thm}

\begin{enumerate}
\item The Dirac measure at points $x$ such that $f(x)\geq0$, or
\item The distributions $v$ that have
\begin{enumerate}
\item density function of the form $e^{\ell}$ with linear $\ell$,
\begin{enumerate}
\item support equal to a segment $[a,b]\subset\K$,
\item $\int f\,\D v=0$,
\item $\int_{a}^{x}f\,\D v>0$ for $x\in(a,b)$ or $\int_{x}^{b}f\,\D v>0$
for $x\in(a,b)$.
\end{enumerate}
\end{enumerate}
\end{enumerate}
This version provides a useful tool: Since maximizers of convex functions
contain extreme points, one can solve $\max_{\mu\in P_{f}}\Phi(\mu)$
for any convex $\Phi$ by only checking Dirac measures and log-affine
functions.

\subsection{Application: logconcave isoperimetry}

We will sketch a proof of the following theorem to illustrate the
use of localization. This theorem was also proved by Karzanov and
Khachiyan \cite{KarzanovK91} using a different, more direct approach.
\begin{thm}[\cite{DyerF90,LS90,KarzanovK91}]
\label{ISO2}Let $f$ be a logconcave function whose support has
diameter $D$ and let $\pi_{f}$ be the induced measure. Then for
any partition of $\R^{n}$ into measurable sets $S_{1},S_{2},S_{3}$, 

\[
\pi_{f}(S_{3})\geq\frac{2d(S_{1},S_{2})}{D}\,\min\{\pi_{f}(S_{1}),\pi_{f}(S_{2})\}\,,
\]
where $d(S_{1},S_{2}):=\inf_{x\in S_{1},y\in S_{2}}d(x,y)$.
\end{thm}

\begin{wrapfigure}{r}{0.4\textwidth}
\centering
\begin{tikzpicture}[y=0.80pt, x=0.80pt, yscale=-0.200000, xscale=0.200000, inner sep=0pt, outer sep=0pt]   \path[draw=cff0000,line join=miter,line cap=butt,miter limit=4.00,even odd     rule,line width=1.000pt] (84.8522,415.4839) .. controls (377.3462,302.3555)     and (380.0296,302.3555) .. (380.0296,302.3555) -- (490.3539,367.2152) --     (693.9910,486.9334) -- (683.2572,668.5343);   \path[draw=cff0000,line join=miter,line cap=butt,miter limit=4.00,even odd     rule,line width=1.000pt] (84.8522,415.4839) .. controls (50.3252,819.2758) and     (284.0089,759.4034) .. (683.2572,668.5343);   \path[draw=c063cb8,line join=miter,line cap=butt,miter limit=4.00,draw     opacity=0.761,even odd rule,line width=1.000pt] (174.0000,382.3622) ..     controls (223.1713,796.0781) and (466.8943,432.1223) .. (462.0000,716.3622);   \path[draw=c063cb8,line join=miter,line cap=butt,miter limit=4.00,draw     opacity=0.761,even odd rule,line width=1.000pt] (296.0000,389.3622) ..     controls (625.9415,359.4824) and (483.4065,426.7886) .. (442.0000,496.3622) ..     controls (402.3261,614.9930) and (611.7898,588.9575) .. (548.0000,698.1069);   \path[draw=c063cb8,line join=miter,line cap=butt,miter limit=4.00,draw     opacity=0.761,even odd rule,line width=1.000pt] (232.9787,359.8090) ..     controls (238.3701,382.0707) and (243.5154,393.1383) .. (296.0000,389.3622);   \path[draw=c063cb8,fill=c063cb8,line join=miter,line cap=butt,miter     limit=4.00,draw opacity=0.761,fill opacity=0.761,line width=1.000pt]     (160.5745,624.7026) node[above right] (text6789) {$S_1$};   \path[draw=c063cb8,fill=c063cb8,line join=miter,line cap=butt,miter     limit=4.00,draw opacity=0.761,fill opacity=0.761,line width=1.000pt]     (343.6170,506.6175) node[above right] (text6793) {$S_3$};   \path[draw=c063cb8,fill=c063cb8,line join=miter,line cap=butt,miter     limit=4.00,draw opacity=0.761,fill opacity=0.761,line width=1.000pt]     (579.7872,593.8516) node[above right] (text6797) {$S_2$};
\end{tikzpicture}
\caption*{Euclidean isoperimetry}
\end{wrapfigure}

A classical variant of this result is in the Riemannian setting.
\begin{thm}[\cite{li1980estimates}]
If $K\subset(M,g)$ is a locally convex bounded domain with smooth
boundary, diameter $D$ and $\text{Ric}_{g}\geq0$, then the Poincar\'e
constant is at most $4D^{2}/\pi^{2}$, i.e., for any $g$ with $\int g=0$,
we have that
\[
\int\left|\nabla g(x)\right|^{2}\,\D x\geq\frac{\pi^{2}}{4D^{2}}\int g(x)^{2}\,\D x\,.
\]
\end{thm}

For the case of convex bodies in $\Rn$, this result is equivalent
to Theorem~\ref{ISO2} up to a constant. 
\begin{proof}[Proof of Theorem~\ref{ISO2}.]
 For a proof by contradiction, let us assume the converse of its
conclusion, i.e., for some partition $S_{1},S_{2},S_{3}$ of $\R^{n}$
and logconcave density $f$, assume that
\[
\int_{S_{3}}f(x)\,\D x<C\int_{S_{1}}f(x)\,\D x\quad\mbox{ and }\quad\int_{S_{3}}f(x)\,\D x<C\int_{S_{2}}f(x)\,\D x\,,
\]
where $C=2d(S_{1},S_{2})/D$. This can be reformulated as 
\begin{equation}
\int g(x)\,\D x>0\quad\mbox{ and }\quad\int h(x)\,\D x>0\,,\label{eq:1}
\end{equation}
where 
\[
g(x)=\begin{cases}
Cf(x) & \mbox{ if }x\in S_{1}\,,\\
0 & \mbox{ if }x\in S_{2}\,,\\
-f(x) & \mbox{ if }x\in S_{3}\,.
\end{cases}\quad\mbox{ and }\quad h(x)=\begin{cases}
0 & \mbox{ if }x\in S_{1}\,,\\
Cf(x) & \mbox{ if }x\in S_{2}\,,\\
-f(x) & \mbox{ if }x\in S_{3}\,.
\end{cases}
\]
We can apply the localization lemma to get an interval $[a,b]$ and
an affine function $\ell$ such that 
\begin{equation}
\int_{0}^{1}\ell(t)^{n-1}\,g\bpar{(1-t)a+tb}\,\D t>0\quad\mbox{ and }\quad\int_{0}^{1}\ell(t)^{n-1}\,h\bpar{(1-t)a+tb}\,\D t>0\,.\label{AFTERLEMMA}
\end{equation}
The functions $g,h$ as we have defined them are not lower semi-continuous.
However, this can be addressed by expanding $S_{1}$ and $S_{2}$
slightly so as to make them open sets, and making the support of $f$
an open set. Since we are proving strict inequalities, these modifications
do not affect the conclusion.

Let us partition $[0,1]$ into $Z_{1},Z_{2},Z_{3}$ as follows: 
\[
Z_{i}=\{t\in[0,1]\,:\,(1-t)a+tb\in S_{i}\}\,.
\]
Note that $\abs{u-v}\ge d(S_{1},S_{2})/D$ for any pair of points
$u\in Z_{1},v\in Z_{2}$. We can rewrite (\ref{AFTERLEMMA}) as 
\[
\int_{Z_{3}}\ell(t)^{n-1}\,f\bpar{(1-t)a+tb}\,\D t<C\int_{Z_{1}}\ell(t)^{n-1}\,f\bpar{(1-t)a+tb}\,\D t
\]
and 
\[
\int_{Z_{3}}\ell(t)^{n-1}\,f\bpar{(1-t)a+tb}\,\D t<C\int_{Z_{2}}\ell(t)^{n-1}\,f\bpar{(1-t)a+tb}\,\D t\,.
\]
The functions $f$ and $\ell^{n-1}$ are both logconcave, so $F(t)=\ell(t)^{n-1}\,f((1-t)a+tb)$
is also logconcave. We get, 
\begin{equation}
\int_{Z_{3}}F(t)\,\D t<C\min\Bbrace{\int_{Z_{1}}F(t)\,\D t,\int_{Z_{2}}F(t)\,\D t}\,.\label{FEQ}
\end{equation}
Now consider what Theorem~\ref{ISO2} asserts for the function $F(t)$
over the interval $[0,1]$ and the partition $Z_{1},Z_{2},Z_{3}$:
\begin{equation}
\int_{Z_{3}}F(t)\,\D t\geq2d(Z_{1},Z_{2})\min\Bbrace{\int_{Z_{1}}F(t)\,\D t,\int_{Z_{2}}F(t)\,\D t}\,.\label{FEQ2}
\end{equation}
We have substituted $1$ for the diameter of the interval $[0,1]$.
Also, $2d(Z_{1},Z_{2})\geq2d(S_{1},S_{2})/D=C$. Thus, Theorem~\ref{ISO2}
applied to the function $F(t)$ contradicts (\ref{FEQ}) and to prove
the theorem in general, it suffices to prove it in the one-dimensional
case.

A combinatorial argument reduces this to the case when each $Z_{i}$
is a single interval. In detail, denoting $\pi(I):=\int_{I}F(t)\,\D t$
and $t:=d(Z_{1},Z_{2})$, we can show that if $\pi(I)\geq t\min\{\pi([0,a]),\pi([b,1])\}$
for every interval $I=(a,b)\subset[0,1]$ with $b-a\geq t$, then
$\pi(Z_{3})\geq t\min\{\pi(Z_{1}),\pi(Z_{2})\}$. For every maximal
open interval $I=(a,b)\subset Z_{3}$ with $b-a\geq t$, we color
one of $[0,a]$ and $[b,1]$ whose integral is smaller; namely, color
$[0,a]$ \emph{red} if $\pi([0,a])\le\pi([b,1])$. Consider the union
$R$ of all red sets. We can write $R=[0,r]\cup[s,1]$ for some $0\leq r\leq s\leq1$
(possibly $R=[0,1]$). If one of $Z_{1}$ or $Z_{2}$ is fully contained
in $R$, then we have
\[
\pi(Z_{3})\ge t\pi(R)\ge t\min\{\pi(Z_{1}),\pi(Z_{2})\}\,,
\]
and we are done. Otherwise, we can pick $x\in U\cap Z_{1}$ and $y\in U\cap Z_{2}$
where $U\subset[0,1]$ is uncolored so that we can take a maximal
interval $I=(a,b)\subset U\cap Z_{3}$ (with $b-a\geq t$) lying between
$x$ and $y$. Hence, either $[0,a]$ or $[b,1]$ was colored red,
so one of $Z_{1}$ and $Z_{2}$ must indeed be contained in $R$. 

Lastly, proving the assumption (i.e., the case of three intervals)
up to a factor of $2$ is a simple exercise and uses only the unimodality
of $F$. Precisely, for every $0\leq a<b\leq1$, we can check $\pi([a,b])\geq(b-a)\,\min\{\pi([0,a]),\pi([b,1])\}$.
Let $m\in[0,1]$ be a mode of $F$, so $F$ is non-decreasing on $[0,m]$
and non-increasing on $[m,1]$. When $b\leq m$, as $F$ is non-decreasing
on $[0,b]$, we have $\pi([0,a])\leq aF(a)\leq F(a)$, so 
\[
\pi([a,b])\geq(b-a)\,F(a)\geq(b-a)\,\pi([0,a])\geq(b-a)\,\min\{\pi([0,a]),\pi([b,1])\}.
\]
The case of $m\leq a$ can be proven in a similar fashion. When $a\leq m\leq b$,
using $F(t)\ge\min\{F(a),F(b)\}$ for all $t\in[a,b]$, we have $\pi([a,b])\ge(b-a)\,\min\{F(a),F(b)\}$.
Since $\pi([0,a])\leq F(a)$ and $\pi([b,1])\leq F(b)$, we are done.
The improvement by a factor of $2$ to get the tight bound uses one-dimensional
logconcavity, specifically observing that the inequality is extremal
for a truncated exponential distribution. 
\end{proof}
\begin{figure}
\centering{}\vspace*{-3cm}
\begin{tikzpicture}[y=0.80pt, x=0.80pt, yscale=-0.500000, xscale=0.500000, inner sep=0pt, outer sep=0pt]   \path[draw=black,line join=miter,line cap=butt,miter limit=4.00,even odd     rule,line width=1.000pt] (49.4368,481.0238) -- (638.9274,473.7382);   \path[draw=cff0000,line join=miter,line cap=butt,miter limit=4.00,even odd     rule,line width=1.000pt] (49.9408,477.0956) .. controls (211.8739,472.0423)     and (225.1589,42.0372) .. (408.0000,363.3622);   \path[draw=cff0000,line join=miter,line cap=butt,miter limit=4.00,even odd     rule,line width=1.000pt] (408.0000,363.3622) .. controls (465.1394,474.3529)     and (497.3013,470.6621) .. (637.3252,470.6447);   \path[draw=black,line join=miter,line cap=butt,miter limit=4.00,even odd     rule,line width=1.000pt] (163.1089,392.7482) .. controls (162.6487,480.1715)     and (162.6487,480.1715) .. (162.6487,480.1715);   \path[draw=black,line join=miter,line cap=butt,miter limit=4.00,even odd     rule,line width=1.000pt] (245.6003,268.6231) .. controls (245.3088,479.5246)     and (245.3088,479.5246) .. (245.3088,479.5246);   \path[xscale=0.943,yscale=1.060,fill=c008600,line join=miter,line cap=butt,line     width=1.000pt] (196.7728,481.5306) node[above right] (text7961) {$Z_3$};   \path[xscale=0.943,yscale=1.060,fill=black,line join=miter,line cap=butt,line     width=1.000pt] (355.1962,481.5861) node[above right] (text7961-8) {$Z_2$};   \path[xscale=0.943,yscale=1.060,fill=black,line join=miter,line cap=butt,line     width=1.000pt] (78.3270,482.0243) node[above right] (text7961-7) {$Z_1$};   \path[fill=c008600] (163.8199,435.1567) -- (164.1599,391.7936) --     (168.7722,384.6251) .. controls (174.7102,375.3961) and (178.2118,369.6886) ..     (193.1964,344.8145) .. controls (219.2131,301.6275) and (227.9729,288.7495) ..     (243.0990,271.4510) -- (244.3579,270.0115) -- (244.3768,373.6738) --     (244.3957,477.3362) -- (222.9072,477.6342) .. controls (211.0885,477.7981) and     (192.8825,478.0644) .. (182.4493,478.2260) -- (163.4800,478.5199) --     (163.8200,435.1567) -- cycle;
\end{tikzpicture}\caption{One-dimensional isoperimetry}
\end{figure}

\subsection{Further applications}

\paragraph{Isoperimetry. }

The localization lemma has been used to prove a variety of isoperimetric
inequalities. Theorem~\ref{thm:strongly_logconcave} shows that the
KLS conjecture is true for an important family of distributions. We
state it below in a slightly more general form. The proof is again
by localization \cite{KannanLM06,Bobkov2007,CV2014}. We note that
the same theorem can be obtained by other means \cite{Ledoux1999,eldan13thin}.
\begin{thm}
\label{thm:Gaussian-iso}Let $h(x)=f(x)e^{-\frac{1}{2}x^{\T}Bx}/\int f(y)e^{-\frac{1}{2}y^{\T}By}\,\D y$
where $f:\R^{n}\rightarrow\R_{+}$ is an integrable logconcave function
and $B$ is positive definite. Then $h$ is logconcave and for any
measurable subset $S$ of $\R^{n}$ with $h(S)\le1/2$,
\[
\frac{h(\partial S)}{h(S)\sqrt{\log(1/h(S))}}\gtrsim\frac{1}{\norm{B^{-1}}^{\frac{1}{2}}}\,.
\]
\end{thm}

In other words, the expansion of $h$ is higher for subsets of smaller
measure; when $B=tI$, then the RHS is $\sqrt{t}$. 

\paragraph{Annealing. }

The analysis of the Gaussian Cooling algorithm for volume computation
\cite{CV2015} uses localization. The algorithm is based on sampling
from a sequence of distributions with each distribution giving a warm
start for the next and ensuring that an estimator for the ratio of
consecutive integrals has small variance. The proof of the inequality
for bounding the variance uses the localization lemma. 

\paragraph*{Anti-concentration.}

The proof of the Carbery--Wright anti-concentration inequality uses
the localization lemma to reduce to one-dimensional distributions,
and then verifies the resulting one-dimensional inequality. The proof
proceeds as follows. We will in fact prove the following slightly
more general result.
\begin{thm}
For a polynomial $p:\Rn\rightarrow\R$ of degree $d$ and a logconcave
distribution $\pi$ in $\R^{n}$, it holds that for any $q\in[1,\infty)$,
\[
\P_{\pi}(\abs{p(X)}\leq\veps)\lesssim\frac{q\veps^{1/d}}{(\E_{\pi}[\abs{p(X)}^{q/d}])^{1/q}}\,.
\]
\end{thm}

\begin{proof}
[Proof sketch] We apply the localization lemma, specifically the
version with exponential needles (Remark~\ref{rem:one-eq-ineq-localization}).
Suppose that the desired inequality is false: there exists a degree-$d$
polynomial $p:\Rn\to\R$, $T>0$, and $q\in[1,\infty)$ such that
for any constant $C>0$,
\[
\int\ind[\abs{p(X)}\leq\veps]\,\pi(x)\,\D x>Cq\veps^{1/d}\Bpar{\int\abs{p(x)}^{q/d}\,\pi(x)\,\D x}^{-1/q}\,.
\]
Let $c_{1}:=\int\abs{p(x)}^{q/d}\,\pi(x)\,\D x$, and define
\[
f_{1}(x):=\abs{p(x)}^{q/d}-c_{1}\,,\qquad f_{2}(x):=\ind[\abs{p(x)}\leq\veps]-Cq\veps^{1/d}c_{1}^{-1/q}\,,
\]
which satisfies $\int f_{1}(x)\,\pi(x)\,\D x=0$ and $\int f_{2}(x)\,\pi(x)\,\D x>0$.
Due to Remark~\ref{rem:one-eq-ineq-localization}, 
\[
\int_{0}^{T}f_{1}(t)e^{-t}\,\D t=0\,,\qquad\int_{0}^{T}f_{2}(t)e^{-t}\,\D t>0\,.
\]
Equivalently, this means that
\begin{align*}
c_{1} & =\frac{\int_{0}^{T}\abs{p(t)}^{q/d}e^{-t}\,\D t}{\int_{0}^{T}e^{-t}\,\D t}\,,\\
\int_{0}^{T}f_{2}(t)e^{-t}\,\D t & =\int_{0}^{T}\ind[\abs{p(t)}\leq\veps]\,e^{-t}\,\D t>Cq\veps^{1/d}c_{1}^{-1/q}\int_{0}^{T}e^{-t}\,\D t\,.
\end{align*}
Hence,
\[
\Bpar{\frac{\int_{0}^{T}\abs{p(t)}^{q/d}e^{-t}\,\D t}{\int_{0}^{T}e^{-t}\,\D t}}^{1/q}\int_{0}^{T}\ind[\abs{p(t)}\leq\veps]\,e^{-t}\,\D t>Cq\veps^{1/d}\int_{0}^{T}e^{-t}\,\D t\,.
\]
Thus, it suffices to disprove this by showing that for any degree-$d$
polynomial $p:\Rn\to\R$, $T>0$, and $q\in(0,\infty)$: 
\[
\Bpar{\frac{\int_{0}^{T}\left|p(t)\right|^{q/d}e^{-t}\,\D t}{\int_{0}^{T}e^{-t}\,\D t}}^{1/q}\,\frac{\int_{0}^{T}\ind[\abs{p(t)}\le\varepsilon]\,e^{-t}\,\D t}{\int_{0}^{T}e^{-t}\,\D t}\lesssim\varepsilon^{1/d}\max\{q,1\}\,.
\]
For the proof of this one-dimensional inequality, which relies on
classical facts about the analysis of polynomials, we refer the reader
to \cite{carbery2001distributional}. For the reader's convenience
we include one important illustrative case, which is a classical result
\cite{dudley1935determining}. 
\begin{claim}
For any polynomial $p:\R\rightarrow\mathbb{C}$ of degree at most
$d$ and any interval $I\subset\R$, we have 
\[
(\sup_{x\in I}\abs{p(x)})^{1/d}\,\abs{\{x\in I:\abs{p(x)}\le\varepsilon\}}\lesssim\abs I\,.
\]
\end{claim}

\end{proof}

\paragraph{Optimization.}

The proof that the universal barrier in the interior-point method
for convex optimization in $\R^{n}$ has the optimal value of $n$
for its self-concordance parameter is based on the localization lemma
\cite{lee2021universal}.

\selectlanguage{american}%

\section{The Stochastic Method}

\inputencoding{latin9}In this section, we present the stochastic version
of localization, introduced by Eldan \cite{eldan13thin} and developed
further by several researchers. While classical localization can be
viewed as a deterministic reduction from an arbitrary distribution
to one with one-dimensional support, stochastic localization instead
decomposes a given distribution into a convex combination (distribution)
over simpler distributions, which satisfy nice properties with high
probability. This decomposition itself is achieved by a simple Martingale
process applied to the original density. Roughly speaking, it replaces
the hyperplane decomposition of standard localization with a more
general weighted decomposition. 

\paragraph{It\^o calculus.}

Before we proceed, we review the basics of It\^o calculus for the
reader's convenience. Let $(X_{t})_{t\ge0}$ be a vector-valued It\^o
process in $\Rn$ of the form
\[
X_{t}=X_{0}+\int_{0}^{t}b_{s}\,\D s+\int_{0}^{t}\sigma_{s}\,\D W_{t}\quad\text{ for }t\geq0\,,
\]
where $b_{t}\in\Rn$ is the \emph{drift} coefficient, $\sigma_{t}\in\R^{n\times m}$
is the \emph{diffusion} coefficient, and $W_{t}$ is an $m$-dimensional
Brownian motion. This is often written in differential form as 
\[
\D X_{t}=b_{t}\,\D t+\sigma_{t}\,\D W_{t}\,.
\]

When $X_{t}$ is a real-valued It\^o process, the \emph{quadratic
variation} $[X]_{t}$ is defined as
\[
[X]_{t}=\lim_{n\to\infty}\sum_{i=1}^{n}(X_{t_{i}}-X_{t_{i-1}})^{2}=\int_{0}^{t}\sigma_{s}^{2}\,\D s,
\]
where $\{t_{i}\}_{0\leq i\leq n}$ is a partition of $[0,t]$ whose
mesh size $\max_{i\in[n]}\abs{t_{i}-t_{i-1}}$ goes to zero as $n\to\infty$.
When $X_{t}$ is vector-valued, the quadratic variation generalizes
to the matrix-valued process 
\[
[X]_{t}=\int_{0}^{t}\sigma_{s}\sigma_{s}^{\T}\,\D s\,
\]
and $\D[X^{i},X^{j}]_{t}=(\sigma_{t}\sigma_{t}^{\top})_{ij}dt$. For
instance, we can readily check that $[W]_{t}=tI$.

\emph{It\^o's formula} describes how the process $X_{t}$ transforms
under a smooth function $f$. For a twice-differentiable function
$f\in C^{2}(\Rn)$, the new It\^o process $(f(X_{t}))_{t\geq0}$
is given by
\begin{align}
\D f(X_{t}) & =\sum_{i}\frac{\D f(X_{t})}{\D X^{i}}\,\D X^{i}+\frac{1}{2}\sum_{i,j}\frac{\D^{2}f(X_{t})}{\D X^{i}\D X^{j}}\,\D[X^{i},X^{j}]_{t}\label{eq:Ito}\\
 & =\inner{\nabla f(X_{t}),\D X_{t}}+\half\,\inner{\hess f(X_{t}),\D[X]_{t}}\nonumber \\
 & =\bpar{\inner{\nabla f(X_{t}),b_{t}}+\half\,\inner{\hess f(X_{t}),\sigma_{t}\sigma_{t}^{\T}}}\,\D t+\inner{\sigma_{t}^{\T}\nabla f(X_{t}),\D W_{t}}\,,\nonumber 
\end{align}
where $\inner{A,B}:=\tr(A^{\T}B)$ for two matrices $A,B$. When $f$
is a vector-valued function, the formula applies component-wise.

\subsection{The main idea: infinitesimal re-weighting }

Stochastic Localization (SL) is a continuous transformation applied
to a probability density. It is defined as follows for any distribution
$p$ in $\Rn$:
\begin{align}
p_{0} & =p\nonumber \\
\D p_{t}(x) & =p_{t}(x)\,\langle x-b_{t},\D W_{t}\rangle\ \text{for all }x\in\Rn\,,\tag{\ensuremath{\msf{SL}}-\ensuremath{\msf{SDE}}}\label{eq:SL-SDE}
\end{align}
where $b_{t}=\int x\,\D p_{t}(x)$ is the barycenter of $p_{t}$.
Its properties are summarized below:
\begin{enumerate}
\item The density $p_{t}$ at time $t$ is a probability measure over $\Rn$
(i.e., $\int\D p_{t}=1$):
\[
\D\int p_{t}=\int\D p_{t}=\Bigl<\int(x-b_{t})\,\D p_{t}(x),\D W_{t}\Bigr>=\langle0,\D W_{t}\rangle=0\,.
\]
\item It is a martingale with respect to the ``filtration induced by the
Brownian motion'', i.e., $\E p_{t}(x)=p(x)$ for all $x\in\Rn$, where
the expectation is taken over the randomness of $\D W_{t}$ (not the
measure $p$ or $p_{t}$).
\item For a function $F:\R^{n}\rightarrow\R$, the It\^o derivative of
the martingale $M_{t}=\int F(x)\,p_{t}(\D x)$ is:
\begin{equation}
\D M_{t}=\int F(x)\,\langle x-b_{t},\D W_{t}\rangle\,p_{t}(\D x)\,.\tag{{\ensuremath{\msf{MG}}}}\label{eq:moment}
\end{equation}
\item The solution at time $t$ can be explicitly stated as follows: 
\begin{equation}
p_{t}(x)\propto p(x)\,\exp\bpar{c_{t}^{\T}x-\frac{t}{2}\,\|x\|^{2}}\,.\tag{\ensuremath{\msf{SL}}-\ensuremath{\msf{pdf}}}\label{eq:SL-pdf}
\end{equation}
Hence, for any $t>0$, we have that $p_{t}$ is $t$-strongly logconcave,
since $p$ is logconcave. 
\end{enumerate}
An alternative equivalent definition is often useful.
\begin{defn}
\label{def:A} Given a logconcave density $p$, we define the following
stochastic differential equation:
\begin{equation}
c_{0}=0\,,\quad\D c_{t}=\mu_{t}\,\D t+\D W_{t}\,,\label{eq:dBt}
\end{equation}
where $\mu_{t}$ is defined as the barycenter of a time-evolving density,
\[
p_{t}(x)=\frac{e^{c_{t}^{\T}x-\frac{t}{2}\,\norm x^{2}}p(x)}{\int e^{c_{t}^{\T}y-\frac{t}{2}\,\norm y^{2}}p(y)\,\D y}\,,\quad\mu_{t}=\E_{p_{t}}X\,.
\]
\end{defn}

We will presently explain why $p_{t}$ takes this form with a Gaussian
component. Before we do that, we note that the process can be generalized
using a ``control'' matrix $C_{t}$ at time $t$. This is a positive
definite matrix that could, for example, be used to adapt the process
to the covariance of the current distribution. At time $t$, the covariance
matrix is
\[
A_{t}\defeq\E_{p_{t}}[(X-\mu_{t})\,(X-\mu_{t})^{\T}]\,.
\]

The control matrix is incorporated in the following more general version:
\begin{defn}
\label{def:SDE} Given a logconcave density $p$, we define the following
stochastic differential equation:
\begin{equation}
c_{0}=0\,,\quad\D c_{t}=C_{t}\mu_{t}\,\D t+C_{t}^{1/2}\,\D W_{t}\,,\label{eq:dBtC}
\end{equation}
\[
B_{0}=0,\quad\D B_{t}=C_{t}\,\D t,
\]
where the probability distribution $p_{t}$, the mean $\mu_{t}$,
and the covariance $A_{t}$ are defined by 
\[
p_{t}(x)=\frac{e^{c_{t}^{\T}x-\frac{1}{2}\,x^{\T}B_{t}x}p(x)}{\int e^{c_{t}^{\T}y-\frac{1}{2}\,y^{\T}B_{t}y}p(y)\,\D y}\,,\quad\mu_{t}=\E_{p_{t}}X,\quad A_{t}=\E_{p_{t}}[(X-\mu_{t})^{\otimes2}]\,,
\]
and the control matrices $C_{t}$ are symmetric matrices to be specified
later.
\end{defn}

When $C_{t}$ is a Lipschitz function with respect to $c_{t},\mu_{t},A_{t}$,
and $t$, standard theorems (e.g., \cite[\S5.2]{oksendal2013stochastic})
show the existence and uniqueness of the solution in time $[0,T]$
for any $T>0$.

In most applications to date, it suffices to focus on the case $C_{t}=I$
and hence $B_{t}=tI$, a setting sometimes called the LV process \cite{LV24eldan}.
The lemma below says that the stochastic process is the same as continuously
multiplying $p_{t}(x)$ by a random infinitesimally small linear function. 
\begin{lem}
\label{lem:def-pt} Equation~(\ref{eq:SL-SDE}) and Definition~\ref{def:A}
describe the same density $p_{t}$ with barycenter $b_{t}=\mu_{t}$
and covariance $\Sigma_{t}=A_{t}$.
\end{lem}

\paragraph{Important properties of stochastic localization.}

Let $(p_{t})_{t\geq0}$ with isotropic logconcave $p_{0}=p$ be a
stochastic process obtained by stochastic localization (either Equation~(\ref{eq:SL-SDE})
or Definition~\ref{def:A}). Let $\D\pi_{t}(x)\propto p_{t}(x)\,\D x$
and $\Sigma_{t}=\cov\pi_{t}$. Intuitively, $\Sigma_{t}$ will not
deviate too much from $\Sigma_{0}=I_{n}$, so quantities of interest
including $\norm{\Sigma_{t}}$ and $\tr\Sigma_{t}$ will remain close
to $\norm{\Sigma_{0}}=1$ and $\tr\Sigma_{0}=n$. We collect important
properties of $(\pi_{t})_{t\geq0}$ that will be proven in later sections.
\begin{itemize}
\item \textbf{The largest eigenvalue} of $\Sigma_{t}$:
\[
\E\norm{\Sigma_{t}}=\O(1)\quad\text{for }t\lesssim\log^{-2}n\,.
\]
\item \textbf{Trace of $\Sigma_{t}^{2}$}:
\[
\E\tr(\Sigma_{t}^{2})=\Theta(n)\quad\text{for }t\lesssim1\,.
\]
In fact, $\E\tr(\Sigma_{t}^{2})=O(n)$ for \emph{all} $t>0$.
\item \textbf{Trace of $\Sigma_{t}$}:
\[
\E\tr\Sigma_{t}=\Omega(n)\quad\text{for }t\lesssim1\,.
\]
\end{itemize}
\selectlanguage{english}%

\selectlanguage{american}%

\subsection{Warm-up I: classical localization with short needles}

Given a distribution with logconcave density $p$, we start at time
$t=0$ and for every time $t>0$, we apply an infinitesimal change
to the density. This is done by picking a random direction from a
Gaussian with a certain covariance matrix $C_{t}$, called the \emph{control
matrix}, as in Definition~\ref{def:SDE}. In order to construct the
stochastic process, we assume that the support of $p$ is contained
in a ball of radius $R>n$. Note that there is only exponentially
small probability outside this ball, at most $e^{-O(R)}$ (see Theorem~\ref{thm:Paouris}).

We now proceed to analyzing the process. Roughly speaking, the first
lemma below says that the stochastic process is the same as continuously
multiplying $p_{t}(x)$ by a random infinitesimal linear function.
\begin{lem}
\label{lem:def-ptC} For the process given by (\ref{eq:dBtC}), we
have 
\[
\D p_{t}(x)=\inner{x-\mu_{t},C_{t}^{1/2}\,\D W_{t}}\,p_{t}(x)\quad\text{for any }x\in\Rn\,.
\]
\end{lem}

By computing $\D\log p_{t}(x)$ with the It\^o's lemma, we see that
applying $\D p_{t}(x)$ as in the lemma above results in the density
$p_{t}(x)$, with a Gaussian term:
\begin{align*}
\D\log p_{t}(x) & =\frac{\D p_{t}(x)}{p_{t}(x)}-\frac{1}{2}\,\frac{\D[p_{t}(x)]_{t}}{p_{t}(x)^{2}}\\
 & =\inner{x-\mu_{t},C_{t}^{1/2}\,\D W_{t}}-\frac{1}{2}(x-\mu_{t})^{\T}C_{t}\,(x-\mu_{t})\,\D t\\
 & =x^{\T}\,(C_{t}\mu_{t}\,\D t+C_{t}^{1/2}\,\D W_{t})-\frac{1}{2}\,x^{\T}C_{t}x\,\D t-\underbrace{(\mu_{t}^{\T}C_{t}^{1/2}\,\D W_{t}+\frac{1}{2}\,\mu_{t}^{\T}C_{t}\mu_{t}\,\D t)}_{=:g(t)}\\
 & =x^{\T}\,\D c_{t}-\frac{1}{2}\,x^{\T}(\D B_{t})x+g(t)\,,
\end{align*}
where the last term is independent of $x$, and the first two terms
explain the form of $p_{t}(x)$ and the appearance of the Gaussian. 

The next lemma gives the time derivative of the covariance matrix.
\begin{lem}
\label{lem:dA} We have
\begin{align*}
\D A_{t} & =\int(x-\mu_{t})^{\otimes2}\,\inner{x-\mu_{t},C_{t}^{1/2}\,\D W_{t}}\,p_{t}(\D x)-A_{t}C_{t}A_{t}\,\D t\,.
\end{align*}
\end{lem}

With this process, we can show the following ``short-needle'' localization
theorem.
\begin{thm}[Short Needle Decomposition]
\label{thm:short-needle}  Let $\nu$ be an isotropic logconcave
probability measure in $\R^{n}$, and $f:\R^{n}\rightarrow\R$ be
a bounded measurable function such that $\int f\,\D\nu=0$. Then there
exists a random probability measure $q$ such that
\begin{enumerate}
\item \emph{{[}Decomposition of $\nu${]}} $\E q=\nu$. 
\item The support of $q$ is either a Dirac measure at points $x$ with
$f(x)=0$, or a logconcave measure supported by a segment $[a,b]\in\R^{n}$
such that $\int f\,\mathrm{d}q=0$, such that with probability at
least $1-O(\log^{-2}n)$ over the randomness of $q$, we have $\var q=O(\log^{2}n)$, 
\item The support of $q$ is either a Dirac measure at points $x$ with
$f(x)=0$, or a log-affine measure supported by a segment $[a,b]\in\R^{n}$
such that $\int f\,\D q=0$, and $\E\var q=O(\log^{2}n)$.
\end{enumerate}
\end{thm}

We will need the following lemmas. 
\begin{lem}
\label{lem:op_norm} We have $A_{t_{0}}\preceq2I$ for $t_{0}\lesssim\log^{-2}n$
with probability at least $1-\exp(-c/t_{0})$, for some constant $c>0$.
\end{lem}

This will be proven in full details in \S\ref{sec:KLS-bound} (see
Theorem~\ref{thm:opnorm-control}). The next lemma is a weaker bound
that holds for all time. 
\begin{lem}
\label{lem:all_t} Conditioning on $p_{s}$ for $s\geq0$, we have
$\E[A_{t}\,|\,p_{s}]\preceq A_{s}$ for all $t\geq s$.
\end{lem}

\begin{proof}
By Lemma \ref{lem:dA}, we have
\[
\D A_{t}=\int(x-\mu_{t})^{\otimes2}\,\inner{x-\mu_{t},C_{t}^{1/2}\,\D W_{t}}\,p_{t}(\D x)-A_{t}C_{t}A_{t}\,\D t\,.
\]
The claim follows by observing that the first term is a martingale
term and the second term $-A_{t}C_{t}A_{t}\preceq0$.
\end{proof}
\begin{lem}
\label{lem:strongly_cov} For $c\in\Rn$ and a PD matrix $B\in\Rnn$,
let $A$ be the covariance matrix of the probability distribution
with density
\[
h(x)=\frac{\exp\bpar{c^{\T}x-\frac{1}{2}\,x^{\T}Bx}\,p(x)}{\int\exp\bpar{c^{\T}y-\half\,y^{\T}By}\,p(y)\,\D y}\,,
\]
where $p(x)$ is a logconcave function. Then, we have $A\preceq B^{-1}.$
\end{lem}

\begin{proof}
Using the Brascamp--Lieb inequality (\ref{eq:Brascamp-Lieb}) for
$\pi=h$ and $f(x)=v^{\T}(x-\E_{\pi}X)$ for any unit vector $v\in\R^{n}$,
we obtain that
\[
v^{\T}Av\leq v^{\T}B^{-1}v\,,
\]
which completes the proof.
\end{proof}
We are now ready for the main proof. 
\begin{proof}
[Proof of Theorem~\ref{thm:short-needle}] We apply the SL starting
from $\nu$ with density $p$ and set $q=\lim_{t\rightarrow\infty}p_{t}$.
The control matrix $C_{t}$ is specified as follows so that $\int f(x)\,p_{t}(\D x)=0$
holds for all $t\geq0$. By Lemma~\ref{lem:def-ptC}, we have
\[
\D\int f(x)\,p_{t}(\D x)=\int f(x)\,\inner{x-\mu_{t},C_{t}^{1/2}\,\D W_{t}}\,p_{t}(x)=\Bigl\langle\int(x-\mu_{t})\,f(x)\,p_{t}(\D x),C_{t}^{1/2}\,\D W_{t}\Bigr\rangle\,.
\]
Define the unit vector 
\[
u_{t}=\frac{\int(x-\mu_{t})\,f(x)\,p_{t}(\D x)}{\bnorm{\int(x-\mu_{t})\,f(x)\,p_{t}(\D x)}}\,.
\]
In the case where the denominator of $u_{t}$ is $0$, define $u_{t}$
to be an arbitrary unit vector. Let the control matrix be
\[
C_{t}=I_{n}-u_{t}u_{t}^{\T}\,.
\]
One can check $\D\int f(x)\,p_{t}(\D x)=0$, so $\int f(x)\,p_{t}(\D x)=0$
holds for all $t\in\R_{+}$ due to $\int f\,\D\nu=0$.

Let $v_{t}$ be the unit eigenvector that corresponds to the largest
eigenvalue of $B_{t}$. We prove the following lower bound on $B_{t}$
in terms of $v_{t}$:
\[
B_{t}\succeq\frac{t}{2}\,(I_{n}-v_{t}v_{t}^{\T})\quad\text{for all }t\geq0\,.
\]
This lower bound implies that as $t\rightarrow\infty$, the support
of $p_{t}$ tends to an at most one-dimensional affine subspace. 

Since $\mathrm{d}B_{t}=C_{t}\,\D t=(I-u_{t}u_{t}^{\T})\,\D t$, we
have
\[
B_{t}=\int_{0}^{t}(I_{n}-u_{s}u_{s}^{\T})\,\D s=tI_{n}-\int_{0}^{t}u_{s}u_{s}^{\T}\,\D s\,.
\]
Let $\lambda_{1}\geq\lambda_{2}\geq\cdots\geq\lambda_{n}\geq0$ be
the eigenvalues of the PSD matrix $\int_{0}^{t}u_{s}u_{s}^{\T}\,\D s$,
and $\eta_{1},\cdots,\eta_{n}$ the corresponding orthonormal eigenvectors.
Notice that $v_{t}=\eta_{1}$, and since $\tr(\int_{0}^{t}u_{s}u_{s}^{\T}\,\D s)=\int_{0}^{t}\norm{u_{s}}^{2}\,\D t=t$,
we have $\lambda_{2}\leq t/2$ and that $\lambda_{1}\leq t$. It follows
that 
\[
tI_{n}-\int_{0}^{t}u_{s}u_{s}^{\T}\,\mathrm{d}s=\frac{t}{2}\,I_{n}+\sum_{i=1}^{n}\frac{t}{2}\,\eta_{i}\eta_{i}^{\T}-\sum_{i=1}^{n}\lambda_{i}\eta_{i}\eta_{i}^{\T}\succeq\frac{t}{2}\,(I_{n}-\eta_{1}\eta_{1}^{\T})\,.
\]
This finishes the proof of the claim.

Lemma~\ref{lem:op_norm} states that with high probability, $A_{t_{0}}\preceq2I$
for $t_{0}\asymp\log^{-2}n$. Conditioning on $p_{t_{0}}$, we have
from Lemma~\ref{lem:all_t} that 
\[
\E[v_{t_{0}}^{\T}A_{\infty}v_{t_{0}}\,|\,A_{t_{0}}]\leq v_{t_{0}}^{\T}A_{t_{0}}v_{t_{0}}\leq2\,.
\]
Taking expectation again, we obtain $\E[v_{t_{0}}^{\T}A_{\infty}v_{t_{0}}]\leq2$,
which implies that $v_{t_{0}}^{\T}A_{\infty}v_{t_{0}}=O(\log^{2}n)$
with probability at least $1-O(1/\log^{2}n)$ due to Markov's inequality.
Under this event, for any unit vector $v\in\R^{n}$, we write $v=u+w$
where $u$ is parallel to $v_{t_{0}}$ and $w\bot v_{t_{0}}$. It
follows that 
\begin{align*}
v^{\T}A_{\infty}v & \leq2u^{\T}A_{\infty}u+2w^{\T}A_{\infty}w\\
 & \underset{(i)}{\leq}2C\log^{2}n+2w^{\T}B_{t_{0}}^{-1}w\\
 & \underset{(ii)}{\leq}2C\log^{2}n+2C\log^{2}n\,w^{\T}(I-v_{t_{0}}v_{t_{0}}^{\T})^{-1}\,w=O(\log^{2}n),
\end{align*}
where $(i)$ follows from Lemma~\ref{lem:strongly_cov} and that
$B_{t}$ is increasing with $t$, and $(ii)$ uses the claim above
with the inverse interpreted as being taken in the $(n-1)$-dimensional
subspace orthogonal to $v_{t_{0}}$. 

The proof of the last part of the theorem, going from logconcave to
log-affine is not based on stochastic localization. We take the logconcave
measure $q$ from the above construction with one-dimensional support,
and can either apply the last part of the proof of Lemma \ref{lem:LOCALIZATION}
or apply the argument of Fradelizi and Guedon \cite{fradelizi2004extreme}
(Step 3 in the proof of their Theorem 1). The latter decomposes $q$
into log-affine distributions while maintaining that zero integral
condition. Any decomposition of a $1$-dimensional measure cannot
increase its variance. The conclusion follows.
\end{proof}

\subsection{Warm-up II: polylog KLS with operator norm control}

Though Klartag showed $\cpi(n)\asymp\psi_{\kls}^{2}(n)=O(\log n)$
in \cite{Klartag23log}, we can give a crisp proof of a logarithmic
bound on the KLS constant using the control on the operator norm from
Lemma~\ref{lem:op_norm}. 
\begin{thm}
$\cpi(n)\asymp\psi_{\kls}^{2}(n)=O(\log^{2}n)$.
\end{thm}

\begin{proof}
Let $\nu$ be an isotropic logconcave distribution with density $p$
in $\R^{n}$, and let $S$ be a measurable subset of $\R^{n}$ with
$\nu(S)=1/2$. Our goal is to lower bound the isoperimetric ratio
of $S$, i.e., the fraction 
\[
\frac{p(\partial S)}{p(S)}\,.
\]
Now consider SL starting at $p_{0}=p$ reaching $p_{t}$ at time $t$,
and denote by $\nu_{t}$ the probability measure with density $p_{t}$.
Using the $t$-strong logconcavity of $p_{t}$ and the Brascamp-Lieb
inequality (\ref{eq:Brascamp-Lieb}), we know that 
\[
\frac{p_{t}(\partial S)}{p_{t}(S)}\apprge t^{1/2}\,.
\]

Since $p_{t}$ is a martingale, $\E p_{t}(S)=p_{0}(S)$ and $\E p_{t}(\partial S)=p_{0}(\partial S)$.
Then,
\begin{align}
p_{0}(\partial S) & =\E p_{t}(\partial S)\apprge t^{1/2}\,\E p_{t}(S)\nonumber \\
 & \apprge t^{1/2}\,\P\bpar{\frac{1}{4}\le p_{t}(S)\le\frac{3}{4}}\,.\label{eq:ptS_bound}
\end{align}
It would suffice to show that $p_{t}(S)$, whose expectation is $p_{0}(S)=1/2$,
remains bounded away from $0$ and $1$. To this end, consider how
$p_{t}(S)$ changes:
\[
\D p_{t}(S)=\D\int_{S}p_{t}(x)\,dx=\int_{S}\D p_{t}(x)=\int_{S}\langle x-\mu_{t},\D W_{t}\rangle p_{t}(x)\,\D x=\Bpar{\int_{S}(x-\mu_{t})\,p_{t}(x)\,\D x}\,\D W_{t}\,,
\]
and also compute its quadratic variation: using the CS inequality
\begin{align*}
\D[p_{t}(S)]_{t} & =\Bnorm{\int_{S}(x-\mu_{t})\,p_{t}(x)\,\D x}^{2}\,\D t\le\sup_{v:\norm v_{2}\le1}\left(\int_{S}v^{\T}(x-\mu_{t})\,p_{t}(x)\,\D x\right)^{2}\,\D t\\
 & \le\sup_{v:\norm v\le1}\Bpar{\int_{S}\inner{v,x-\mu_{t}}^{2}\,p_{t}(x)\,\D x\cdot\int_{S}p_{t}(x)\,\D x}\,\D t\\
 & \le\sup_{v:\norm v\le1}v^{\T}\Bpar{\int(x-\mu_{t})^{\otimes2}\,p_{t}(x)\,\D x}v\,\D t\\
 & =\norm{A_{t}}\,\D t\,.
\end{align*}
Hence,
\[
[p_{t}(S)]_{t}\le t\,\sup_{0\le s\le t}\norm{A_{t}}\,.
\]
By Theorem~\ref{thm:opnorm-control}, we have that for $t_{0}\asymp\log^{-2}n$,
with probability at least $1-e^{-1/(ct_{0})}$, we have $\sup_{t\leq t_{0}}\norm{A_{t}}\le2$,
and thus $[p_{t_{0}}(S)]_{t}\le2t_{0}$. Finally, by Freedman's inequality
(Lemma~\ref{lem:freedman}), 
\[
\abs{p_{t}(S)-p_{0}(S)}\le2\sqrt{2t_{0}}\le0.01\ (\mbox{say)}
\]
with probability at least $1-e^{-2}$. Therefore, $p_{t}(S)\in[1/4,3/4]$
with probability at least (say) $1/2$. Using our earlier inequality
(\ref{eq:ptS_bound}), 
\[
\frac{p_{0}(\partial S)}{p_{0}(S)}\apprge\sqrt{t_{0}}\asymp\frac{1}{\log n}\,,
\]
which completes the proof.
\end{proof}
\selectlanguage{english}%

\section{The KLS Conjecture \label{sec:KLS-bound}}

One of the main purposes of this survey is to provide a unified and
gentle introduction --- together with essentially self-contained
proofs --- of how SL has led to recent breakthroughs on three longstanding
conjectures in asymptotic convex geometry: the KLS conjecture, the
thin-shell conjecture, and the slicing conjecture. We attempt to seamlessly
interleave arguments from Klartag and Lehec's lecture notes \cite{KL24isop},
Klartag's bound on the KLS constant \cite{Klartag23log}, Guan's breakthrough
on the uniform control of the trace along SL \cite{guan2024note},
and Bizeul's simplified proof of the slicing conjecture \cite{bizeul2025slicing}.
To make the exposition reader-friendly, we include preliminaries on
semigroup theory in \S\ref{sec:semigroup} and a useful perspective
on connections between SL and the $\ps$ ($\PS$) in \S\ref{sec:PS-SL}.

We start with Klartag's $\O(\log n)$ bound on $\cpi(n)$ (equivalently,
$\cpi(\pi)\lesssim\norm{\Sigma}\log n$) \cite{Klartag23log}. Although
not strictly necessary, readers may first consult the introduction
to the $\PS$ in \S\ref{sec:PS-SL}, since many of the technical
complications (when working with the SL) become simpler by working
with the $\PS$ language --- we use the notation 
\[
\pi=\pi^{X},\quad\pi^{Y}=\pi^{X}*\gamma_{h},\quad\bs{\pi}(x,y)\propto\pi(x)\,\gamma_{h}(y-x)
\]
and the $X$-conditional for given $Y=y$ is
\[
\pi^{X|Y=y}:=\bm{\pi}^{X|Y=y}=\bm{\pi}^{X,Y}/\pi^{Y}(y).
\]
The proof outline is to (1) relate $\cpi(\pi)$ to some \emph{average}
of $\cpi(\pi^{X|Y=y})$, (2) apply the improved Lichnerowicz inequality
\eqref{eq:improved-LI} --- which states $\cpi(\mu)\leq(\norm{\cov\mu}/t)^{1/2}$
for $t$-strongly logconcave distributions $\mu$ (see \S\ref{sec:semigroup})
--- and (3) use SL to bound $\norm{\Sigma_{\pi^{X|Y}}}$. Going forward,
we use $\cov\pi$ and $\Sigma_{\pi}$ to denote the covariance matrix
of a distribution $\pi$.

\subsection{Gaussian localization}

The following approach, called ``Gaussian localization'' in \cite{Klartag23log},
is essentially just the backward step of the $\PS$. The key idea
is to bound $\cpi(\pi^{X|Y})$ rather than $\cpi(\pi^{X})=\cpi(\pi)$,
because $\pi^{X|Y}$ admits a good bound on its $\cpi$ via \eqref{eq:improved-LI}.

\paragraph{Law of total variance.}

Viewing $\pi^{X}=\E_{y\sim\pi^{Y}}[\pi^{X|Y=y}]$ as a mixture of
the distribution $\pi^{X|Y=y}$ over $y$, we decompose the variance
using the law of total variance as follows:
\begin{equation}
\var_{\pi^{X}}f=\E_{y\sim\pi^{Y}}\var_{\pi^{X|Y=y}}f+\var_{y\sim\pi^{Y}}(\E_{\pi^{X|Y=y}}f)=\E_{\pi^{Y}}\var_{\pi^{X|Y=y}}f+\var_{\pi^{Y}}Q_{h}^{*}f\,,\label{eq:total-variance}
\end{equation}
where $Q_{h}^{*}$ is the heat adjoint (see \S\ref{sec:PS-SL}).

\paragraph{Relating $\protect\cpi(\pi^{X})$ and $\protect\cpi(\pi^{X|Y=y})$.}

We then bound the second term by a multiple of the first term, sandwiching
$\var_{\pi^{X}}f$ in terms of the first term. To do so, we first
compute the gradient of $\pi^{X|Y=y}$ (i.e., $\PS$'s backward step).
\begin{prop}
[Gradient of $\pi^{X|Y=y}$] \label{prop:back-grad} For $\pi^{X|Y=y}(x)=\bs{\pi}(x,y)/\pi^{Y}(y)$,
\[
\nabla_{y}\pi^{X|Y=y}=\frac{x-b_{h}}{h}\,\pi^{X|Y=y}\,,
\]
where $b_{h}$ is the barycenter of $\pi^{X|Y=y}$ (i.e., $b_{h}=b_{h}(y)=\E_{\pi^{X|Y=y}}X$).
\end{prop}

\begin{proof}
We compute the gradient of Gaussian convolution: for any function
$f$,
\begin{align*}
\nabla_{y}\int f(x)\,\gamma_{h}(y-x)\,\D x & =\frac{1}{(2\pi h)^{n/2}}\,\nabla_{y}\int\exp\bpar{-\frac{1}{2h}\,\norm{y-x}^{2}}\,f(x)\,\D x\\
 & =-\frac{1}{(2\pi h)^{n/2}}\,\int\frac{y-x}{h}\,\exp\bpar{-\frac{1}{2h}\,\norm{y-x}^{2}}\,f(x)\,\D x\\
 & =-\int\frac{y-x}{h}\,f(x)\,\gamma_{h}(y-x)\,\D x\,,
\end{align*}
so $\grad_{y}\pi^{Y}=\grad_{y}(\pi^{X}*\gamma_{h})=-\int\frac{y-x}{h}\,\pi^{X}(x)\,\gamma_{h}(y-x)\,\D x$.

Taking the gradient of $\pi^{X|Y=y}(x)=\pi^{X}(x)\,\gamma_{h}(y-x)/\pi^{Y}(y)$,
\begin{align*}
\nabla_{y}\pi^{X|Y=y}(x) & =-\frac{y-x}{h}\,\frac{\pi^{X}(x)\,\gamma_{h}(y-x)}{\pi^{Y}(y)}-\frac{\pi^{X}(x)\,\gamma_{h}(y-x)}{\pi^{Y}(y)}\,\frac{\nabla_{y}\pi^{Y}}{\pi^{Y}(y)}\\
 & =-\frac{y-x}{h}\,\frac{\pi^{X}(x)\,\gamma_{h}(y-x)}{\pi^{Y}(y)}+\frac{\pi^{X}(x)\,\gamma_{h}(y-x)}{\pi^{Y}(y)}\int\frac{y-x}{h}\,\frac{\pi^{X}(x)\,\gamma_{h}(y-x)}{\pi^{Y}(y)}\,\D x\\
 & =-\frac{y-x}{h}\,\pi^{X|Y=y}+\pi^{X|Y=y}\int\frac{y-x}{h}\,\pi^{X|Y=y}\,\D x\\
 & \underset{(i)}{=}\pi^{X|Y=y}\,\bpar{-\frac{y-x}{h}+\frac{y-b_{h}}{h}}\\
 & =\frac{x-b_{h}}{h}\,\pi^{X|Y=y}\,,
\end{align*}
where $b_{h}$ denotes the barycenter of $\pi^{X|Y=y}$.
\end{proof}
\begin{lem}
For $h>0$, function $f:\R^{n}\rightarrow\R$ with finite second moment
and logconcave $\pi^{X}$ over $\Rn$,
\[
\E_{\pi^{Y}}\var_{\pi^{X|Y=y}}f\le\var_{\pi^{X}}f\leq\bpar{2+\frac{\cpi(\pi^{X})}{h}}\,\E_{\pi^{Y}}\var_{\pi^{X|Y=y}}f\,.
\]
\end{lem}

\begin{proof}
Applying \eqref{eq:PI} to $Q_{h}^{*}f(y)$ and using $\cpi(\pi^{Y})=\cpi(\pi^{X}*\gamma_{h})\leq\cpi(\pi^{X})+h$
\cite{chafai04entropies},
\begin{equation}
\var_{\pi^{Y}}Q_{h}^{*}f\leq\cpi(\pi^{Y})\,\E_{\pi^{Y}}[\norm{\nabla_{y}Q_{h}^{*}f}^{2}]\leq\bpar{\cpi(\pi^{X})+h}\,\E_{\pi^{Y}}[\norm{\nabla_{y}Q_{h}^{*}f}^{2}]\,.\label{eq:bound-second-term}
\end{equation}
Using the formula of $\nabla_{y}\pi^{X|Y=y}$ above,
\[
\nabla_{y}Q_{h}^{*}f=\int f(x)\,\nabla_{y}\pi^{X|Y=y}(\D x)=\int\frac{x-b_{h}}{h}\,f(x)\,\pi^{X|Y=y}(\D x)\,.
\]
By Cauchy--Schwarz,
\[
\norm{\nabla_{y}Q_{h}^{*}f}^{2}=\Bnorm{\int\frac{x-b_{h}}{h}\,f(x)\,\pi^{X|Y=y}(\D x)}^{2}\leq\frac{1}{h^{2}}\,\norm{\Sigma_{\pi^{X|Y=y}}}\,\var_{\pi^{X|Y=y}}f\leq\frac{1}{h}\,\var_{\pi^{X|Y=y}}f\,,
\]
where the last inequality follows from that $\pi^{X|Y=y}$ is $h^{-1}$-strongly
logconcave (so $\norm{\Sigma_{\pi^{X|Y=y}}}\leq\cpi(\pi^{X|Y=y})\le h$
(say) by Brascamp--Lieb). Putting this bound back to \eqref{eq:bound-second-term},
\[
\var_{\pi^{Y}}Q_{h}^{*}f\leq\frac{\cpi(\pi^{X})+h}{h}\,\E_{\pi^{Y}}\var_{\pi^{X|Y=y}}f\,.
\]
Putting this back to the bound \eqref{eq:total-variance} from the
total law of variance,
\[
\var_{\pi^{X}}f=\E_{\pi^{Y}}\var_{\pi^{X|Y=y}}f+\var_{\pi^{Y}}Q_{h}^{*}f\leq\bpar{2+\frac{\cpi(\pi^{X})}{h}}\,\E_{\pi^{Y}}\var_{\pi^{X|Y=y}}f\,,
\]
which completes the proof.
\end{proof}
We now convert this ``variance coupling'' to a ``$\cpi$ coupling''
between $\pi^{X}$ and $\pi^{X|Y}$.
\begin{cor}
For $h>0$,
\[
\cpi(\pi^{X})\lesssim\bpar{1+\frac{\cpi(\pi^{X})}{h}}\,\E_{\pi^{Y}}\cpi(\pi^{X|Y=y})\,.
\]
Therefore,
\[
\cpi(\pi^{X})\lesssim\bpar{\sqrt{h}+\frac{\cpi(\pi^{X})}{\sqrt{h}}}\,\E_{\pi^{Y}}\sqrt{\norm{\Sigma_{\pi^{X|Y=y}}}}\,.
\]
\end{cor}

\begin{proof}
The variance-coupling lemma tells us that
\[
\var_{\pi^{X}}f\leq\bpar{1+\frac{\cpi(\pi^{X})}{h}}\,\E_{\pi^{Y}}\var_{\pi^{X|Y=y}}f\,.
\]
For the LHS, using Milman \cite{Milman2009}, for a logconcave probability
measure $\pi^{X}$ over $\Rn$, there exists a $1$-Lipschitz function
$f:\Rn\to\R$ such that (for some universal constant hidden in $\asymp$
below)
\[
\var_{\pi^{X}}f\asymp\cpi(\pi^{X})\,.
\]
Taking this $1$-Lipschitz function $f$, we can lower bound the LHS
by $\cpi(\pi^{X})$ (up to some constant).

As for the RHS, we apply \eqref{eq:PI} to $f$ and $\pi^{X|Y=y}$
to obtain that
\[
\var_{\pi^{X|Y=y}}f\le\cpi(\pi^{X|Y=y})\,\E_{\pi^{X|Y=y}}[\norm{\nabla f}^{2}]\leq\cpi(\pi^{X|Y=y})\,.
\]
Combining these two bounds on the LHS and RHS, 
\[
\cpi(\pi^{X})\lesssim\bpar{1+\frac{\cpi(\pi^{X})}{h}}\,\E_{\pi^{Y}}\cpi(\pi^{X|Y})\,.
\]

For the second claim, we apply \eqref{eq:improved-LI} to $\cpi(\pi^{X|Y})$
in the inequality above:
\begin{align*}
\cpi(\pi^{X}) & \lesssim\bpar{1+\frac{\cpi(\pi^{X})}{h}}\,\E_{\pi^{Y}}\cpi(\pi^{X|Y})\leq\bpar{1+\frac{\cpi(\pi^{X})}{h}}\,\E_{\pi^{Y}}\sqrt{h\,\norm{\Sigma_{\pi^{X|Y=y}}}}\\
 & =\bpar{\sqrt{h}+\frac{\cpi(\pi^{X})}{\sqrt{h}}}\,\E_{\pi^{Y}}\sqrt{\norm{\Sigma_{\pi^{X|Y=y}}}}\,,
\end{align*}
which completes the proof.
\end{proof}
Now it suffices to bound $\E_{\pi^{Y}}\norm{\Sigma_{\pi^{X|Y=y}}}$
since 
\[
\E_{\pi^{Y}}\sqrt{\norm{\Sigma_{\pi^{X|Y=y}}}}\leq\sqrt{\E_{\pi^{Y}}\norm{\Sigma_{\pi^{X|Y=y}}}}\,.
\]

\paragraph{SL for operator-norm control.}

We will show the following operator-norm control presently. 
\begin{thm}
\label{thm:cov-bound} Assume $\pi^{X}$ is isotropic. For $\pi_{h}^{X|Y=y}=\pi^{X|Y=y}$for
some $h>0$, let 
\[
s_{0}=\min\{s>0:\E_{\pi^{Y}}\norm{\Sigma_{\pi_{h}^{X|Y=y}}}\leq5\ \text{for all }h\geq s\}\,.
\]
Then, $s_{0}\lesssim\log^{2}n$. Moreover, there is an example in
which $s_{0}\gtrsim\log n$.
\end{thm}

We \textbf{have not used} SL at all thus far (and it may be possible
to bound $\E_{\pi^{Y}}\norm{\Sigma_{\pi^{X|Y=y}}}$ directly without
SL). However, as we discuss in \S\ref{sec:PS-SL}, SL provides a
continuous-time perspective on the backward step $\pi^{X|Y=y}$. This
allows us to study the evolution of $\norm{\Sigma_{\pi^{X|Y}}}$ through
the SL framework, reducing the problem to an application of It\^o
calculus.

Using this theorem, we establish $\cpi(n)\lesssim\log n$ as follows. 
\begin{thm}
[\cite{Klartag23log}] $\cpi(\pi)\lesssim\norm{\Sigma}\log n$ for
any logconcave $\pi$ over $\Rn$.
\end{thm}

\begin{proof}
For an isotropic logconcave $\pi^{X}$, take $h\asymp\log^{2}n$ to
ensure $\E_{\pi^{Y}}\norm{\Sigma_{\pi^{X|Y=y}}}=\O(1)$. Then, 
\[
\cpi(\pi)\lesssim\sqrt{h}+\frac{\cpi(\pi)}{\sqrt{h}}\,.
\]
Rearranging terms, we obtain $\cpi(n)\lesssim\log n$, which implies
$\cpi(\pi)\lesssim\norm{\Sigma}\log n$.
\end{proof}

\subsection{Operator-norm control via SL\label{subsec:Operator-norm-control}}

We are now ready to do the ``heavy-lifting'' required to bound $\norm{\Sigma_{\pi^{X|Y=y}}}$,
using SL .

\subsubsection{Proof outline for covariance bound}

The main problem addressed in this section is the following:
\begin{center}
\emph{For isotropic $\pi^{X}$, bound $\E_{\pi^{Y}}\norm{\Sigma_{\pi^{X|Y=y}}}=\O(1)$
for $h>0$, as small as possible.}
\par\end{center}

It may not be immediately clear how SL enters the picture when tackling
this problem. To see this, let us compare the following two approaches.
For $h=1/t$,
\begin{itemize}
\item \textbf{Backward step} \textbf{of the $\PS$}:
\[
\pi^{X|Y=y}(x)\propto\pi^{X}(x)\,\exp\bpar{-\frac{1}{2h}\,\norm{x-y}^{2}}\,,
\]
where $y\sim\pi^{Y}=\pi^{X}*\gamma_{h}\overset{d}{=}X+W_{h}$.
\item \textbf{SL }\eqref{eq:SL-pdf}: for $\pi=\pi^{X}\propto p$ and $\pi_{t}\propto p_{t}$,
\[
\pi_{t}(x)\propto\pi(x)\,\exp\bpar{-\frac{t}{2}\,\bnorm{x-\frac{c_{t}}{t}}^{2}}\,,
\]
where $c_{t}$ satisfies $c_{0}=0$ and $\D c_{t}=b_{t}\,\D t+\D W_{t}$.
\end{itemize}
As the reader may notice, these two approaches look similar when $h=1/t$.
In fact, $\E_{c_{t}}\norm{\cov\pi_{t}}=\E_{\pi^{Y}}\norm{\Sigma_{\pi^{X|Y=y}}}$
for $h=1/t$, and this can be justified by showing that $c_{t}/t\overset{d}{=}Y\sim\pi^{Y}$.
Indeed, for another Brownian motion $W'_{t}$, 
\[
Y\overset{d}{=}X+W_{h}=X+W_{1/t}\overset{d}{=}X+\frac{1}{t}\,W'_{t}\overset{d}{\underset{(i)}{=}}\frac{c_{t}}{t}\,,
\]
where $(i)$ follows from \cite{KP23spectral}. Namely, the ``randomness''
in SL is captured by $c_{t}/t$ (which corresponds to $Y\sim\pi^{Y}$
obtained after the forward step of the $\PS$ starting from $\pi^{X}$).
This relationship also implies that the backward step $\pi^{X|Y=y}$
with step size $h$ can be viewed as a ``snapshot'' of SL at time
$t=1/h$. Hence, \emph{we may equivalently attempt to bound }$\E_{c_{t}}\norm{\cov\pi_{t}}$\emph{
instead of }$\E_{\pi^{Y}}\norm{\Sigma_{\pi^{X|Y=y}}}$. This connection
turns the problem above into the following:
\begin{problem}
[Operator-norm control via SL] When $(\pi_{t})_{t\geq0}$ is SL with
isotropic $\pi_{0}=\pi$, establish $\lVert\cov\pi_{t}\rVert=\O(1)$
for $t$ less than some threshold $\tau$.
\end{problem}

When $t$ is close to $0$, we can expect $\pi_{t}$ to remain nearly
isotropic. Thus, we should not run the SL process for too long (i.e.,
we restrict to $t\leq T$ for some threshold $T$), because the longer
SL runs, the more likely $\norm{\cov\pi_{t}}$ is to deviate from
its initial value $\norm{\cov\pi}=1$. Therefore, Theorem~\ref{thm:cov-bound}
can be expressed in the SL language as follows:
\begin{thm}
[SL-version of Theorem~\ref{thm:cov-bound}] Let $\pi=\pi_{0}=\pi^{X}$
be isotropic logconcave probability measure over $\Rn$, and $\Sigma_{t}$
be the covariance of $\pi_{t}$. Then, $\E\norm{\Sigma_{t}}=\O(1)$
for $t\lesssim\log^{-2}n$.
\end{thm}

To this end, observe that
\[
\E\norm{\Sigma_{t}}=\E[\norm{\Sigma_{t}}\cdot\ind\{\norm{\Sigma_{t}}\leq2\}]+\E[\underbrace{\norm{\Sigma_{t}}}_{\leq1/t}\cdot\ind\{\norm{\Sigma_{t}}\geq2\}]\leq2+\frac{1}{t}\,\P(\norm{\Sigma_{t}}\geq2)\,.
\]
Hence, it suffices to establish the following operator-norm control.
\begin{thm}
[Operator-norm control]\label{thm:opnorm-control} Let $\pi$ be
an isotropic logconcave probability measure in $\Rn$. For $T\asymp\log^{-2}n$,
\begin{equation}
\P(\norm{\Sigma_{t}}\leq2\ \text{for }t\in[0,T])\geq1-\exp\bpar{-\frac{1}{CT}}\,,\label{eq:final-target}
\end{equation}
where $C>0$ is a universal constant.
\end{thm}

In words, $\norm{\Sigma_{t}}$ stays close to $\norm{\cov\pi}=1$
with high probability up to the time threshold $T$.

\subsubsection{It\^o calculus for covariance $\Sigma_{t}$}

\paragraph{Preliminary computations.}

While handling the stochastic quantity $\norm{\Sigma_{t}}$, we will
encounter several important quantities involving the barycenter $b_{t}$
and covariance $\Sigma_{t}$ of $\pi_{t}$. We therefore compute $\D b_{t}$
and $\D\Sigma_{t}$. Using It\^o's lemma and applying \eqref{eq:moment}
with $F(x)=x$ and $x^{\otimes2}$,
\begin{align*}
\D b_{t} & =\D\Bpar{\int x\,\pi_{t}(\D x)}=\int x\otimes(x-b_{t})\,\pi_{t}(\D x)\,\D W_{t}=\int(x-b_{t})^{\otimes2}\,\pi_{t}(\D x)\,\D W_{t}\\
 & =\Sigma_{t}\,\D W_{t}\,,\\
\D\Sigma_{t} & =\D\Bpar{\int x^{\otimes2}\,\pi_{t}(\D x)-b_{t}^{\otimes2}}\\
 & =\int x^{\otimes2}\inner{x-b_{t},\D W_{t}}\,\pi_{t}(\D x)-\D(b_{t}^{\otimes2})\\
 & \underset{(i)}{=}\int x^{\otimes2}\inner{x-b_{t},\D W_{t}}\,\pi_{t}(\D x)-(\D b_{t}\otimes b_{t}+b_{t}\otimes\D b_{t}+\D b_{t}\otimes\D b_{t})\\
 & =\int(x-b_{t})^{\otimes2}\inner{x-b_{t},\D W_{t}}\,\pi_{t}(\D x)+\int(x\otimes b_{t}+b_{t}\otimes x-b_{t}\otimes b_{t})\inner{x-b_{t},\D W_{t}}\,\pi_{t}(\D x)\\
 & \qquad-\Bpar{\int(x-b_{t})^{\otimes2}\,\pi_{t}(\D x)\,\D W_{t}\otimes b_{t}+b_{t}\otimes\int(x-b_{t})^{\otimes2}\,\pi_{t}(\D x)\,\D W_{t}+\Sigma_{t}^{2}\,\D t}\\
 & \underset{(ii)}{=}\underbrace{\int(x-b_{t})^{\otimes2}\inner{x-b_{t},\D W_{t}}\,\pi_{t}(\D x)}_{=:\D H_{t}}-\Sigma_{t}^{2}\,\D t\,,
\end{align*}
where $(i)$ follows from the It\^o product rule, and $(ii)$ follows
from $\int x\otimes b_{t}\,\inner{x-b_{t},\D W_{t}}\,\pi_{t}(\D x)=(\int(x-b_{t})^{\otimes2}\,\pi_{t}(\D x)\,\D W_{t})\otimes b_{t}$.
As for $(i)$, one can actually check it by It\^o's formula as follows:
for $f(x)=x\otimes x$,
\begin{align*}
\D f(X_{t}) & =\text{D}(X_{t}\otimes X_{t})[\D X_{t}]+\half\,\text{D}^{2}(X_{t}\otimes X_{t})[\D X_{t},\D X_{t}]\\
 & =(\id\otimes X_{t}+X_{t}\otimes\id)[\D X_{t}]+(\id\otimes\id)[\D X_{t},\D X_{t}]\\
 & =\D X_{t}\otimes X_{t}+X_{t}\otimes\D X_{t}+\D X_{t}\otimes\D X_{t}\,.
\end{align*}

The $3$-tensor term merits its own notation $H_{t}$, as it is the
source of many analytical complications. For the $3$-tensor analysis,
we also define, for each $i\in[n]$
\begin{equation}
H_{t,i}=H_{i}:=\int(x-b_{t})^{\otimes2}(x-b_{t})_{i}\,\D\pi_{t}\,,\label{eq:3-tensor}
\end{equation}
so $\D H_{t}=\sum_{i=1}^{n}H_{i}\,\D W_{t,i}$.

As we will see later, the main quantities for downstream arguments
are $\E\norm{\Sigma_{t}}$, $\E\tr(\Sigma_{t}^{2})$, and $\E\tr\Sigma_{t}$.
\begin{itemize}
\item \textbf{KLS conjecture}: We have already seen that we need
\[
\E\norm{\Sigma_{t}}=\O(1)\quad\text{for }t\lesssim\log^{-2}n\,.
\]
\item \textbf{Thin-shell conjecture}: When working towards the thin-shell
conjecture in \S\ref{sec:thinshell}, we will need 
\[
\E\tr(\Sigma_{t}^{2})=\Theta(n)\quad\text{for }t\lesssim1\,.
\]
\item \textbf{Slicing conjecture}: When resolving the slicing conjecture
in \S\ref{sec:slicing}, we will need 
\[
\E\tr\Sigma_{t}=\Omega(n)\quad\text{for }t\lesssim1\,.
\]
In fact, this is immediate from the trace of $\Sigma_{t}^{2}$ above.
Integrating both sides of 
\[
\D\Sigma_{t}=\int(x-b_{t})^{\otimes2}\inner{x-b_{t},\D W_{t}}\,\pi_{t}(\D x)-\Sigma_{t}^{2}\,\D t\,,
\]
and taking trace and expectation, noting that $\E\tr\Sigma_{0}=n$
and $\E dW_{t}=0$, we obtain 
\[
\E\tr\Sigma_{t}-n=-\int_{0}^{t}\E\tr(\Sigma_{s}^{2})\,\D s\,.
\]
Therefore,
\[
\de_{t}\E\tr\Sigma_{t}=-\E\tr(\Sigma_{t}^{2})\,.
\]
Since $\E\tr\Sigma_{t}^{2}=\Theta(n)$ for $t\lesssim1$, we can conclude
that for $t\lesssim1$,
\[
\E\tr\Sigma_{t}\gtrsim n\,.
\]
\end{itemize}
Lastly, we will also need $\de_{t}\E[\norm{b_{t}}^{2}]$. By It\^o's
formula,
\[
\D(\norm{b_{t}}^{2})=2\inner{\D b_{t},b_{t}}+\tr(I_{n}\Sigma_{t}^{2})\,\D t=\tr(\Sigma_{t}^{2})\,\D t+2\inner{b_{t},\Sigma_{t}}\,\D W_{t}\,,
\]
and thus
\begin{equation}
\frac{\D}{\D t}\,\E[\norm{b_{t}}^{2}]=\E\tr(\Sigma_{t}^{2})\,.\label{eq:derivative-bt}
\end{equation}

\paragraph{Control of $\protect\norm{\Sigma_{t}}$ via a proxy function.}

We work with a proxy function defined by 
\[
h(M):=\frac{1}{\beta}\log(\tr e^{\beta M})\quad\text{for a symmetric matrix \ensuremath{M}\ensuremath{\in\Rnn\ }and constant }\beta>0.
\]
Clearly, for any $t>0$,
\[
\norm{\Sigma_{t}}\leq h(\Sigma_{t})\leq\norm{\Sigma_{t}}+\frac{\log n}{\beta}\,.
\]
To control $h(\Sigma_{t})$, we compute the It\^o derivative $\D h(\Sigma_{t})$
to see the magnitude of its drift. By It\^o's lemma and $\D\Sigma_{t}=\D H_{t}-\Sigma_{t}^{2}\,\D t$,
\begin{align*}
\D h(\Sigma_{t}) & =\nabla h(\Sigma_{t})[\D\Sigma_{t}]+\frac{1}{2}\,\hess h(\Sigma_{t})[\D\Sigma_{t},\D\Sigma_{t}]\\
 & =\underbrace{\nabla h(\Sigma_{t})[\D H_{t}]}_{=\tr(\nabla h(\Sigma_{t})\,\D H_{t})}-\underbrace{\nabla h(\Sigma_{t})[\Sigma_{t}^{2}]}_{=\tr(\nabla h(\Sigma_{t})\,\Sigma_{t}^{2})}\D t+\frac{1}{2}\,\hess h(\Sigma_{t})[\D H_{t},\D H_{t}]\,.
\end{align*}
From direct computation,
\[
G_{t}:=\nabla h(\Sigma_{t})=\frac{e^{\beta\Sigma_{t}}}{\tr e^{\beta\Sigma_{t}}}\succeq0\,,
\]
and note that $\tr G_{t}=1$. Then,
\begin{align*}
\D h(\Sigma_{t}) & \leq\sum_{i}\tr(G_{t}H_{i})\,\D W_{t,i}+\half\,\hess h(\Sigma_{t})\Bbrack{\sum_{i}H_{i}\,\D W_{t,i},\sum_{i}H_{i}\,\D W_{t,i}}\\
 & =\sum_{i}\tr(G_{t}H_{i})\,\D W_{t,i}+\half\,\sum_{i}\hess h(\Sigma_{t})[H_{i},H_{i}]\,\D t\\
 & \underset{(i)}{\leq}\sum_{i}\tr(G_{t}H_{i})\,\D W_{t,i}+\frac{\beta}{2}\,\sum_{i}\tr(G_{t}H_{i}^{2})\,\D t\\
 & \leq\underbrace{\sum_{i}\tr(G_{t}H_{i})\,\D W_{t,i}}_{\eqqcolon\D Z_{t}}+\frac{\beta}{2}\,\Bnorm{\sum_{i}H_{i}^{2}}\,\D t\,,
\end{align*}
where in $(i)$, the bound on $\hess h(\Sigma_{t})[H_{i},H_{i}]$
follows from Lemma~\ref{lem:second-order-bound}, and the last line
follows from the $(1,\infty)$-H\"older's inequality with $\tr G_{t}=\norm{G_{t}}_{1}\leq1$.

\paragraph{Bounding a bad event via Freedman's inequality.}

Integrating both sides, we have that for some constant $C>0$,
\[
h(\Sigma_{t})\leq h(\Sigma_{0})+Z_{t}+\frac{\beta}{2}\int_{0}^{t}\Bnorm{\sum_{i}H_{s,i}^{2}}\,\D s\leq1+\frac{\log n}{\beta}+Z_{t}+\frac{\beta C}{2}\int_{0}^{t}\frac{\norm{\Sigma_{s}}^{5/2}}{s^{1/2}}\,\D s\,,
\]
where the last line follows from Lemma~\ref{lem:3tensor-2}, as well
as the improved Lichnerowicz inequality \eqref{eq:improved-LI}.

Let us take the smallest $\tau$ such that $\tau\leq t$ and $\norm{\Sigma_{\tau}}\geq2$.
Then, for some constant $C'>0$,
\[
2=\norm{\Sigma_{\tau}}\leq h(\Sigma_{\tau})\leq1+\frac{\log n}{\beta}+Z_{\tau}+\Theta(\beta C\tau^{1/2})\,.
\]
For $\beta=2\log n$ and $t\asymp(\beta C)^{-2}\asymp(C\log n)^{-2}$,
we can ensure $Z_{\tau}\geq\frac{1}{4}$. To apply a deviation inequality
to $Z_{t}$, let us compute the quadratic variation of $Z_{t}$:
\[
\D[Z]_{t}=\sum_{i}\tr^{2}(G_{t}H_{i})\,\D t=\norm v^{2}\,\D t\,,
\]
where the vector $v\in\Rn$ satisfies $v_{i}=\tr(G_{t}H_{i})$. For
any unit vector $\theta\in\mbb S^{n-1}$, using the $(1,\infty)$-H\"older's
inequality and $\tr G_{t}\leq1$,
\[
v\cdot\theta=\tr\Bpar{G_{t}\sum_{i}H_{i}\theta_{i}}\leq\Bnorm{\sum_{i}H_{i}\theta_{i}}=\Bnorm{\int(x-b_{t})^{\otimes2}\inner{x-b_{t},\theta}\,\D\pi_{t}}\lesssim\norm{\Sigma_{t}}^{3/2}\,,
\]
where the last inequality follows from Lemma~\ref{lem:3tensor-1}.
Thus, $\norm v^{2}\leq\norm{\Sigma_{t}}^{3}$ and $[Z]_{\tau}\lesssim\int_{0}^{\tau}\norm{\Sigma_{s}}^{3}\,\D s\lesssim\tau\leq t$.
Therefore, for some constant $C'>0$,
\begin{equation}
\P(\exists\tau\leq t:\norm{\Sigma_{\tau}}\geq2)\leq\P\bpar{\exists\tau>0:Z_{\tau}\geq\frac{1}{4}\ \text{and}\ [Z]_{\tau}\lesssim t}\leq\exp\bpar{-\frac{1}{C't}}\,,\label{eq:deviation-bound}
\end{equation}
where the last inequality follows from the classical deviation inequality
below, and this completes the proof of Theorem~\ref{thm:opnorm-control}.
\begin{lem}
[Freedman's inequality]\label{lem:freedman} Let $(M_{t})_{t\geq0}$
be a continuous local martingale with $M_{0}=0$. Then for $u,\sigma^{2}>0$,
we have
\[
\P(\exists t>0:M_{t}>u\ \text{and}\ [M]_{t}\leq\sigma^{2})\leq\exp\bpar{-\frac{u^{2}}{2\sigma^{2}}}\,.
\]
\end{lem}

\paragraph{Matrix inequalities.}

We compute the second-order directional derivative of $h(M)=\frac{1}{\beta}\log(\tr e^{\beta M})$.
\begin{lem}
[{\cite[Corollary 56]{KL24isop}}]\label{lem:second-order-bound}
For symmetric matrices $M,H\in\Rnn$, 
\[
\hess h(M)[H,H]\leq\beta\,\tr\bpar{\frac{e^{\beta M}}{\tr e^{\beta M}}H^{2}}\,.
\]
\end{lem}

\begin{proof}
Note that
\begin{align*}
\nabla h(M)[H] & =\frac{\nabla(\tr e^{\beta M})[H]}{\beta\,\tr e^{\beta M}}\,,\\
\nabla^{2}h(M)[H,H] & =\frac{\nabla^{2}(\tr e^{\beta M})[H,H]}{\beta\,\tr e^{\beta M}}-\frac{1}{\beta}\,\bpar{\frac{\nabla(\tr e^{\beta M})[H]}{\tr e^{\beta M}}}^{2}\leq\frac{\nabla^{2}(\tr e^{\beta M})[H,H]}{\beta\,\tr e^{\beta M}}\,.
\end{align*}

Let us bound the second-order directional derivative of $f(M):=\tr e^{\beta M}$.
From $f(M)=\sum_{i=0}^{\infty}\frac{\beta^{i}}{i!}\tr(M^{i})$, we
can readily check that
\begin{align*}
\nabla f(M)[H] & =\sum_{i=1}^{\infty}\frac{\beta^{i}}{(i-1)!}\,\tr(M^{i-1}H)=\beta\sum_{i\geq0}\frac{\beta^{i}}{i!}\tr(M^{i}H)\,,\\
\nabla^{2}f(M)[H,H] & =\beta\sum_{i=1}^{\infty}\frac{\beta^{i}}{i!}\sum_{j=0}^{i-1}\tr(M^{j}HM^{i-1-j}H)\,.
\end{align*}
Assuming $M\succeq0$ for a while, we show below in Lemma~\ref{lem:swap-MHMH}
that 
\[
\tr(M^{j}HM^{i-1-j}H)\leq\tr(M^{i-1}H^{2})\,.
\]
Using this,
\begin{equation}
\hess f(M)[H,H]\leq\beta\sum_{i=1}^{\infty}\frac{\beta^{i}}{(i-1)!}\,\tr(M^{i-1}H^{2})=\beta^{2}\sum_{i\geq0}\frac{\beta^{i}}{i!}\,\tr(M^{i}H^{2})=\beta^{2}\tr(e^{\beta M}H^{2})\,.\label{eq:second-diff-bound}
\end{equation}

Now for a general symmetric $M$, pick $\veps>0$ such that $M+\veps I\succeq0$.
Note that 
\[
f(M+\veps I)=\tr e^{\beta(M+\veps I)}=\tr(e^{\beta M}e^{\beta\veps I})=e^{\beta\veps}f(M)\,.
\]
Therefore, \eqref{eq:second-diff-bound} is still valid for any symmetric
$M$. Substituting \eqref{eq:second-diff-bound} back to the bound
on $\hess h(M)[H,H]$, we obtain
\[
\hess h(M)[H,H]\leq\frac{\beta^{2}\,\tr(e^{\beta M}H^{2})}{\beta\,\tr e^{\beta M}}=\beta\,\tr\bpar{\frac{e^{\beta M}}{\tr e^{\beta M}}H^{2}}\,,
\]
which completes the proof.
\end{proof}
We now show that $\tr(M^{j}HM^{i-1-j}H)\leq\tr(M^{i-1}H^{2})$.
\begin{lem}
[{\cite[Lemma 55]{KL24isop}}]\label{lem:swap-MHMH} Let $M\succeq0$
and $H$ be a symmetric matrix in $\Rdd$. For $n,m\in\mbb N$,
\[
\tr(M^{n}HM^{m}H)\le\tr(M^{n+m}H^{2})\,.
\]
\end{lem}

\begin{proof}
By the spectral decomposition of $M$, we can write that for eigenvalues
$\lda_{i}$ and orthonormal basis $\{u_{i}\}$, 
\[
M=\sum_{i=1}^{d}\lda_{i}u_{i}u_{i}^{\T}\,.
\]
Then,
\[
\tr(M^{n}HM^{m}H)=\sum_{i,j}\lda_{i}^{n}\lda_{j}^{m}\,\tr(u_{i}u_{i}^{\T}Hu_{j}u_{j}^{\T}H)=\sum_{i,j}\lda_{i}^{n}\lda_{j}^{m}\,(u_{i}^{\T}Hu_{j})^{2}\,.
\]
By Young's inequality,
\[
\lda_{i}^{n}\lda_{j}^{m}\leq\frac{n}{n+m}\,\lda_{i}^{n+m}+\frac{m}{n+m}\,\lda_{j}^{n+m}\,.
\]
Thus,
\begin{align*}
\sum_{i,j}\lda_{i}^{n}\lda_{j}^{m}\,(u_{i}^{\T}Hu_{j})^{2} & \leq\sum_{i,j}\frac{n}{n+m}\,\lda_{i}^{n+m}(u_{i}^{\T}Hu_{j})^{2}+\sum_{i,j}\frac{m}{n+m}\,\lda_{j}^{n+m}(u_{i}^{\T}Hu_{j})^{2}\\
 & =\sum_{i,j}\lda_{i}^{n+m}\,u_{i}^{\T}Hu_{j}u_{j}^{\T}Hu_{i}=\sum_{i}\lda_{i}^{n+m}\,u_{i}^{\T}H\Bpar{\sum_{j}u_{j}u_{j}^{\T}}Hu_{i}\\
 & =\sum_{i}\lda_{i}^{n+m}\,u_{i}^{\T}H^{2}u_{i}\\
 & =\sum_{i}\tr(\lda_{i}^{n+m}u_{i}u_{i}^{\T}H^{2})\\
 & =\tr(M^{n+m}H^{2})\,,
\end{align*}
which completes the proof.
\end{proof}

\paragraph{3-tensor bounds.}

In the analysis below, we bound $\norm{\int(x-b_{t})^{\otimes2}\inner{x-b_{t},\theta}\,\D\pi_{t}}$
for a unit vector $\theta$. 
\begin{lem}
[{\cite[Lemma 57]{eldan13thin,KL24isop}}]\label{lem:3tensor-1}
Let $X$ be a centered logconcave random variable. Then, for the covariance
$\Sigma$ of $\law X$,
\[
\sup_{\theta\in\mbb S^{n-1}}\norm{\E[(X\cdot\theta)\,X^{\otimes2}]}\lesssim\norm{\Sigma}^{3/2}\,.
\]
\end{lem}

\begin{proof}
For $v\in\mbb S^{n-1}$,
\[
\E[(X\cdot\theta)(X\cdot v)^{2}]\leq\sqrt{\E[(X\cdot\theta)^{2}]}\sqrt{\E[(X\cdot v)^{4}]}\,.
\]
Since $X\cdot v$ is an $1$-dimensional logconcave random variable,
the reverse H\"older inequality tells us that: 
\[
\bpar{\E[(X\cdot v)^{4}]}^{1/4}\lesssim\bpar{\E[(X\cdot v)^{2}]}^{1/2}\,.
\]
Hence,
\[
\E[(X\cdot\theta)(X\cdot v)^{2}]\leq\norm{\Sigma}^{1/2}\,\E[(X\cdot v)^{2}]\leq\norm{\Sigma}^{3/2}\,,
\]
which completes the proof.
\end{proof}
Next, we bound $\norm{\sum_{i}H_{i}^{2}}$ for $H_{i}=\int(x-b_{t})^{\otimes2}(x-b_{t})_{i}\,\pi_{t}(\D x)$.
\begin{lem}
[{\cite[Lemma 58]{KL24isop}}]\label{lem:3tensor-2} Let $X$ be
a centered $t$-strongly logconcave random variable. Then, for the
covariance $\Sigma$ of $\law X$, 
\[
\Bnorm{\sum_{i}(\E[X_{i}X^{\otimes2}])^{2}}\leq\frac{4\norm{\Sigma}^{5/2}}{t^{1/2}}\,.
\]
\end{lem}

\begin{proof}
For $v\in\mbb S^{n-1}$,
\begin{align*}
v^{\T}\sum_{i}(\E[X_{i}X^{\otimes2}])^{2}v & =\sum_{i}\E[(X\cdot v)\,X_{i}X]\,\bpar{\E[(X\cdot v)\,X_{i}X]}^{\T}=\tr\Bpar{\bpar{\underbrace{\E[(X\cdot v)\,X^{\otimes2}]}_{\eqqcolon H_{v}}}^{2}}\,.
\end{align*}
Now let us bound $\tr(H_{v}^{2})$ as follows:
\begin{align*}
\tr(H_{v}^{2}) & =\E\bbrack{\tr\bpar{H_{v}\,(X\cdot v)\,X^{\otimes2}}}=\E[(X^{\T}H_{v}X)\,(X\cdot v)]\\
 & =\E\bbrack{(X^{\T}H_{v}X-\E[X^{\T}H_{v}X])\,(X\cdot v)}\\
 & \leq\sqrt{\E[(X\cdot v)^{2}]}\sqrt{\var(X^{\T}H_{v}X)}\\
 & \leq\sqrt{\norm{\Sigma}}\sqrt{4\cpi(X)\,\E[\norm{H_{v}X}^{2}]}\\
 & =\sqrt{\norm{\Sigma}}\sqrt{4\cpi(X)\,\E[\tr(H_{v}^{2}X^{\otimes2})]}\\
 & =\sqrt{\norm{\Sigma}}\sqrt{4\cpi(X)\,\tr(H_{v}^{2}\,\E[X^{\otimes2}])}\\
 & \leq\norm{\Sigma}\sqrt{4\cpi(X)\,\tr(H_{v}^{2})}\,.
\end{align*}
Rearranging terms, 
\[
\tr(H_{v}^{2})\leq4\cpi(X)\norm{\Sigma}^{2}\leq\frac{4\norm{\Sigma}^{5/2}}{t^{1/2}}\,,
\]
where the last one follows from \eqref{eq:improved-LI}.
\end{proof}

\section{The Slicing Conjecture \label{sec:slicing}}

We now move to the slicing conjecture, which was resolved in December
2024 by Klartag and Lehec \cite{klartag2024affirmativeresolutionbourgainsslicing}.
We follow the much shorter proof by Bizeul \cite{bizeul2025slicing}
combining Guan's trace bound with an approach from \cite{LV24eldan}.

\subsection{Small-ball probability}

Just as the original KLS conjecture is tackled through its Poincar\'e
version \eqref{eq:KLS-PI}, we handle the slicing conjecture through
an equivalent formulation. Dafnis and Paouris \cite{dafnis2010small}
showed that the following is equivalent to $\bar{L}_{n}=\O(1)$ (i.e.,
the slicing conjecture in terms of isotropic convex bodies; see Theorem~\ref{thm:slicing-equiv}):
\begin{thm}
[Small-ball conjecture] \label{thm:small-ball-conjecture} Let $\pi$
be an isotropic logconcave distribution over $\Rn$. Then, there exist
universal constants $c_{1},c_{2}>0$ such that for any $\veps<c_{1}$
and any $y\in\Rn$,
\begin{equation}
\P_{\pi}(\norm{X-y}^{2}\leq\veps n)=\pi\bpar{B_{\sqrt{\veps n}}(y)}\leq\veps^{c_{2}n}\,,\tag{\ensuremath{\msf{sBall}}}\label{eq:small-ball-conjecture}
\end{equation}
where $B_{r}(y)$ is the ball of radius $r$ centered at $y$ in $\Rn$.
\end{thm}

We prove this theorem in this section.
\begin{rem}
[Interpretation] Paouris showed that there exists a universal constant
$c>0$ such that for any $t\geq1$ and any isotropic logconcave $\pi$
over $\Rn$,
\[
\P_{\pi}(\norm X^{2}\leq t^{2}n)\geq1-e^{-ct\sqrt{n}}\,.
\]
In words, the centered $t\sqrt{n}$-ball with $t\geq1$ is \emph{large
enough} to take up most of the measure. On the other hand, the small-ball
conjecture posits how small the measure of a $t\sqrt{n}$-radius ball
should be when $t$ is less than $1$. More precisely, it claims that
its measure decreases polynomially fast in $t$ (i.e., $\sim t^{n}$).
\end{rem}

\subsubsection{Proof outline}

Recall our approach to \eqref{eq:KLS-PI}: we crucially relied on
a \emph{model inequality} such as \eqref{eq:improved-LI}, $\cpi(\pi)\leq(\norm{\Sigma}/t)^{1/2}$,
which holds for $t$-strongly logconcave distributions. Using SL,
we then leveraged this inequality to establish a similar result for
any logconcave distributions.

We take a similar approach toward \eqref{eq:small-ball-conjecture},
which leads the following question:
\begin{center}
\textbf{Q.} What is a suitable \emph{model result} that we can make
use of?
\par\end{center}

In \cite{paouris2012small}, Paouris established a small-ball estimate
for \emph{sub-Gaussian} logconcave distributions, where sub-Gaussianity
is implied by strong log-concavity. 
\begin{lem}
[Small-ball for sub-Gaussian] \label{lem:subGaussian-small-ball}
Let $\pi$ be a $t$-strongly logconcave distribution over $\Rn$
with covariance $\Sigma$. Then, there exists universal constants
$c_{1},c_{2}>0$ such that for any $\veps<c_{1}$ and any $y\in\Rn$,
\[
\P_{\pi}(\norm{X-y}^{2}\leq\veps\tr\Sigma)\leq\veps^{\frac{c_{2}t\tr\Sigma}{\norm{\Sigma}\norm{\Sigma^{-1}}}}\,.
\]
\end{lem}

Namely, this lemma proves \eqref{eq:small-ball-conjecture} for isotropic
and $1$-strongly logconcave $\pi$ (e.g., $\msf N(0,I_{n})$), so
we will treat it as the model result in our approach. To start off,
we consider the SL process $(\pi_{t})_{\geq0}$ with $\pi_{0}=\pi$
isotropic logconcave, and SL  ensures $\pi(S)=\E_{\msf{SL}}\pi_{t}(S)$
for $S:=B_{\sqrt{\veps n}}(y)$. We will carry out a simple stochastic
analysis to show that with some probability
\[
\pi(S)\lesssim\pi_{t}(S)\leq\veps^{\frac{c_{2}t\tr\Sigma_{t}}{\norm{\Sigma_{t}}\norm{\Sigma_{t}^{-1}}}}\,,
\]
where the last inequality follows from Lemma~\ref{lem:subGaussian-small-ball}.
This prompts us to
\begin{enumerate}
\item Upper bound $\norm{\Sigma_{t}}$
\item Upper bound $\norm{\Sigma_{t}^{-1}}$
\item Lower bound $\tr\Sigma_{t}$
\end{enumerate}
The first is very easy due to \eqref{eq:improved-LI} (i.e., $\norm{\Sigma_{t}}\leq t^{-1}$).
The second part is addressed by working with a version of Lemma~\ref{lem:subGaussian-small-ball},
and the last part relies on a consequence of Guan's trace bound (explained
in the next section).

\subsubsection{A version of Paouris' small-ball estimate}

We establish a simple consequence of Paouris' small-ball estimate
by cutting off small eigenvalues of $\Sigma$ suitably.
\begin{lem}
[Small-ball with cutoff, {\cite[Lemma 4]{bizeul2025slicing}}] \label{lem:small-ball-cutoff}
Let $\pi$ be a $t$-strongly logconcave distribution over $\Rn$
with covariance $\Sigma$ satisfying $4n\,\norm{\Sigma}/\tr\Sigma\leq C$
for $C>0$. Then, there exist universal constants $c_{1},c_{2}>0$
such that for any $\veps<c_{1}/C$ and any $y\in\Rn$,
\[
\P_{\pi}(\norm{X-y}^{2}\leq\veps\tr\Sigma)\leq\bpar{\frac{4n\veps\norm{\Sigma}}{\tr\Sigma}}^{\frac{c_{2}t^{3}\tr^{3}\Sigma}{8n^{2}}}\,.
\]
\end{lem}

\begin{proof}
Let $\{(v_{i},\lda_{i})\}_{i\in[n]}$ be the pairs of eigenvectors
and eigenvalues of $\Sigma$ with $\lda_{1}\geq\lda_{2}\geq\dots\geq\lda_{n}\geq0$.
For an effective rank $r:=\tr\Sigma/\norm{\Sigma}$, let us pick 
\[
k=\bigl\lfloor\frac{r}{2}\bigr\rfloor\vee1\,.
\]
Under this choice of $k$, since 
\[
\tr\Sigma=(\lda_{1}+\cdots+\lda_{k})+(\lda_{k+1}+\cdots+\lda_{n})\leq k\,\norm{\Sigma}+n\lda_{k}\,,
\]
we have
\[
\frac{\tr\Sigma}{2n}\leq\lda_{k}\,.
\]
Namely, $\lda_{k}$ is almost as large as the average of the eigenvalues
of $\Sigma$.

Now consider the subspace $\mc S$ spanned by $\{v_{1},\dots,v_{k}\}$,
and the projection $P_{\mc S}:\Rn\to\R^{k}$ onto $\mc S$. Then,
the marginal distribution $\pi_{\mc S}:=(P_{\mc S})_{\#}\pi$ is still
$t$-strongly logconcave (by Pr\'ekopa--Leindler) and satisfies
\begin{equation}
\norm{\Sigma_{\mc S}}\leq\norm{\Sigma}\,,\quad\norm{\Sigma_{\mc S}^{-1}}\leq\frac{2n}{\tr\Sigma}\,.\label{eq:ineq-lem1}
\end{equation}
Furthermore, since the average of eigenvalues goes up after cutting
off small ones,
\[
\frac{1}{k}\tr\Sigma_{\mc S}\geq\frac{1}{n}\tr\Sigma\,,
\]
and thus
\begin{equation}
\frac{\tr\Sigma}{\tr\Sigma_{\mc S}}\leq\frac{n}{k}\leq\frac{4n\,\norm{\Sigma}}{\tr\Sigma}\leq C\,,\quad\text{so}\ \frac{\tr^{2}\Sigma}{4n\,\norm{\Sigma}}\leq\tr\Sigma_{\mc S}\,.\label{eq:ineq-lem2}
\end{equation}
It then follows that for $\veps\leq c_{1}/C$,
\begin{align*}
\P(\norm{X-y}\leq\veps\tr\Sigma) & \leq\P(\norm{P_{\mc S}X-P_{\mc S}y}\leq\veps\tr\Sigma)=\P\bpar{\norm{P_{\mc S}X-P_{\mc S}y}\leq\frac{\veps\tr\Sigma}{\tr\Sigma_{\mc S}}\,\tr\Sigma_{\mc S}}\\
 & \underset{(i)}{\leq}\bpar{\frac{\veps\tr\Sigma}{\tr\Sigma_{\mc S}}}^{\frac{c_{2}t\tr\Sigma_{\mc S}}{\norm{\Sigma_{\mc S}}\norm{\Sigma_{\mc S}^{-1}}}}\underset{(ii)}{\leq}\bpar{\frac{4n\veps\,\norm{\Sigma}}{\tr\Sigma}}^{\frac{c_{2}t\tr^{3}\Sigma}{8n^{2}\norm{\Sigma}^{2}}}\underset{(iii)}{\leq}\bpar{\frac{4n\veps\,\norm{\Sigma}}{\tr\Sigma}}^{\frac{c_{2}t^{3}\tr^{3}\Sigma}{8n^{2}}}\,,
\end{align*}
where in $(i)$ we used Lemma~\ref{lem:subGaussian-small-ball} to
$\pi_{\mc S}$, used \eqref{eq:ineq-lem1} and \eqref{eq:ineq-lem2}
in $(ii)$, and used $\norm{\Sigma}\leq1/t$ in $(iii)$.
\end{proof}

\subsubsection{Stochastic analysis}

\paragraph{Reduction to the centered $\sqrt{n}$-ball.}

We note that when studying the small-ball conjecture, the general
isotropic case can be reduced to truncated isotropic distributions
with support of diameter $\Theta(n^{1/2})$ and $y=0$. To see this,
consider 
\[
Y:=\frac{X_{1}-X_{2}}{\sqrt{2}}\,,
\]
where $X_{1}$ and $X_{2}$ are independent copies from an isotropic
logconcave distribution $\pi$. Then, $\law Y$ is symmetric, isotropic,
and logconcave (since convolution preserved logconcavity). Note that
for any $r>0$,
\[
\P_{\pi}(\norm{X-y}\leq r)=\{\P_{(X_{1},X_{2})\sim\pi^{\otimes2}}(\norm{X_{1}-y}\leq r,\norm{X_{2}-y}\leq r)\}^{1/2}\leq\{\P(\norm Y\leq\sqrt{2}r)\}^{1/2}\,,
\]
so it suffices to focus on the case of $y=0$.

Next, we claim that one can truncate $\mu:=\law Y$ to a ball of radius
$\Theta(n^{1/2})$ and then focus on $\tilde{\mu}\propto\mu\cdot\ind_{\mc B}$
for $\mc B:=B_{10n^{1/2}}(0)$. Note that 
\[
\sup_{\mc B}\frac{\D\tilde{\mu}}{\D\mu}\leq\bpar{\mu(\mc B)}^{-1}\leq\frac{100}{99}\,,
\]
where the last inequality follows from Markov's inequality:
\[
\mu(\mc B)=\P(\norm Y^{2}\leq100n)\geq1-\frac{\E[\norm Y^{2}]}{100n}\geq\frac{99}{100}\,.
\]
Hence, $\norm{\Sigma_{\tilde{\mu}}}\leq\frac{100}{99}$ follows from
\[
\E_{\tilde{\mu}}[(Z\cdot\theta)^{2}]\leq\frac{100}{99}\,\E_{\mu}[(Z\cdot\theta)^{2}]=\frac{100}{99}\,.
\]
For the lower bound, using the reverse H\"older in $(i)$,
\begin{align*}
\E_{\tilde{\mu}}[(Z\cdot\theta)^{2}] & =\frac{\int_{\mc B}(z\cdot\theta)^{2}\,\D\mu}{\mu(\mc B)}\geq1-\int_{\mc B^{c}}(z\cdot\theta)^{2}\,\D\mu\\
 & \geq1-\bpar{\E_{\mu}[(Z\cdot\theta)^{4}]}^{1/2}\bpar{\mu(\mc B^{c})}^{1/2}\\
 & \underset{(i)}{\geq}1-\frac{5\,\E_{\mu}[(Z\cdot\theta)^{2}]}{10}=\frac{1}{2}\,.
\end{align*}
Hence, $\frac{1}{2}\,I_{n}\preceq\Sigma_{\tilde{\mu}}\preceq2I_{n}$.

Lastly, consider $\nu:=[(\Sigma_{\tilde{\mu}})^{-1/2}]_{\#}\tilde{\mu}$
which is isotropic logconcave. Then,
\begin{align*}
\P_{\pi}(\norm{X-y}\leq r) & \le\{\P_{\mu}(\norm Y\leq\sqrt{2}r)\}^{1/2}\leq\{\P_{\tilde{\mu}}(\norm Y\leq\sqrt{2}r)\}^{1/2}\\
 & \leq\{\P_{\tilde{\mu}}(\norm{\Sigma^{-1/2}Y}\leq\norm{\Sigma^{-1/2}}\sqrt{2}r)\}^{1/2}\leq\{\P_{\nu}(\norm Z\leq2r)\}^{1/2}\,.
\end{align*}
Therefore, redefining the universal constants $c_{1},c_{2}$ in the
small-ball estimate (through suitable rescaling), we can focus on
truncated isotropic distributions with support of diameter $\Theta(n^{1/2})$
and $y=0$.

\paragraph{Stochastic analysis vis SL.}

We now use the SL process $(\pi_{t})$ with isotropic logconcave $\pi_{0}=\pi$
supported over $\mc B=B_{10n^{1/2}}(0)$. Let $S=B_{\sqrt{\veps n}}(0)$
and define $Z_{t}:=\pi_{t}(S)$. Since
\[
\D Z_{t}=\D\Bpar{\int_{S}\pi_{t}(\D x)}=\int_{S}\inner{x-b_{t},\D W_{t}}\,\pi_{t}(\D x)\,.
\]
Hence, the quadratic variation is bounded as follows:
\[
\D[Z]_{t}=\Bnorm{\int_{S}(x-b_{t})\,\pi_{t}(\D x)}^{2}\,\D t\le(10n^{1/2}Z_{t})^{2}\,\D t=100nZ_{t}^{2}\,\D t\,.
\]
Now using It\^o's formula to $\log Z_{t}$,
\[
\D\log Z_{t}=\frac{\D Z_{t}}{Z_{t}}-\frac{\D[Z]_{t}}{2Z_{t}^{2}}\geq\frac{1}{Z_{t}}\int_{S}\inner{x-b_{t},\D W_{t}}\,\pi_{t}(\D x)-50n\,\D t\,,
\]
and thus
\[
\log\pi(S)\leq\E\log\pi_{t}(S)+50nt\,.
\]

\begin{lem}
\label{lem:measure-control-Markov} For $\lda>0$, with probability
at least $1-\lda^{-1}$, it holds that
\[
\pi(S)\leq e^{50nt}\pi_{t}^{1/\lda}(S)\,.
\]
\end{lem}

\begin{proof}
Using Markov's inequality to $\log\frac{1}{\pi_{t}(S)}>0$, we have
that for any $\lda>0$,
\[
\P\bpar{\log\frac{1}{\pi_{t}(S)}\geq\lda\,\E\log\frac{1}{\pi_{t}(S)}}\leq\frac{1}{\lda}\,.
\]
Hence, with probability at least $1-\lda^{-1}$,
\[
-\log\pi_{t}(S)=\log\frac{1}{\pi_{t}(S)}\leq\lda\,\E\log\frac{1}{\pi_{t}(S)}=-\lda\,\E\log\pi_{t}(S)\,,
\]
so $\frac{1}{\lda}\log\pi_{t}(S)\geq\E\log\pi_{t}(S)$. Since $\log\pi(S)\leq\E\log\pi_{t}(S)+50nt$
from above, it follows that
\[
\log\pi(S)\leq\frac{1}{\lda}\log\pi_{t}(S)+50nt\,,
\]
which completes the proof.
\end{proof}

\paragraph{Putting all together.}

We introduce a simple consequence of Guan's important bound \cite{guan2024note},
addressed in \S\ref{subsec:Guan}.
\begin{lem}
\label{lem:trace-lower-bound} There exists a universal constant $c\in(0,1)$
such that 
\[
\E\tr\Sigma_{c}\geq cn\,.
\]
\end{lem}

We now prove the small-ball conjecture \eqref{eq:small-ball-conjecture}.
\begin{proof}
[Proof of Theorem~\ref{thm:small-ball-conjecture}] Using $\norm{\Sigma_{c}}=\norm{\cov\pi_{c}}\leq1/c$,
we obtain that 
\begin{align*}
cn & \leq\E\tr\Sigma_{c}=\E\bbrack{\tr\Sigma_{c}\cdot\ind[\tr\Sigma_{c}\geq\frac{cn}{2}]}+\E\bbrack{\tr\Sigma_{c}\cdot\ind[\tr\Sigma_{c}\leq\frac{cn}{2}]}\\
 & \leq\frac{n}{c}\,\P(\tr\Sigma_{c}\geq\frac{cn}{2})+\frac{cn}{2}\,,
\end{align*}
so
\[
\P\bpar{\tr\Sigma_{c}\geq\frac{cn}{2}}\geq\frac{c^{2}}{2}\,.
\]

Using Lemma~\ref{lem:measure-control-Markov} with $\lda=4/c^{2}$,
\[
\P\bpar{\pi(S)\leq e^{50cn}\pi_{c}^{c^{2}/4}(S)}\geq1-\frac{c^{2}}{4}\,.
\]
Therefore,
\[
\P\bpar{\tr\Sigma_{c}\geq\frac{cn}{2}\,,\pi(S)\leq e^{50cn}\pi_{c}^{c^{2}/4}(S)}>0\,.
\]
In this event,
\[
\frac{4n\,\norm{\Sigma_{c}}}{\tr\Sigma_{c}}\leq\frac{8n}{c^{2}n}=\frac{8}{c^{2}}=:C\,.
\]
Below in $(i)$, we use the small-ball estimate (Lemma~\ref{lem:small-ball-cutoff})
for $c$-strongly logconcave $\pi_{c}$: for $S=B_{\sqrt{\veps n}}(0)$
with $\veps\leq\frac{c_{1}c}{2C}$,
\begin{align*}
\pi(S) & \leq e^{50cn}\pi_{c}^{c^{2}/4}(S)=e^{50cn}\bpar{\P(\norm X^{2}\leq\frac{\veps n}{\tr\Sigma_{c}}\tr\Sigma_{c})}^{c^{2}/4}\\
 & \leq e^{50cn}\bpar{\P(\norm X^{2}\leq\frac{2\veps}{c}\tr\Sigma_{c})}^{c^{2}/4}\\
 & \underset{(i)}{\leq}e^{50cn}\bpar{\frac{8n\veps\norm{\Sigma_{c}}}{c\tr\Sigma_{c}}}^{\frac{c_{2}c^{5}\tr^{3}\Sigma_{c}}{32n^{2}}}\\
 & \leq e^{50cn}\bpar{\frac{16\veps}{c^{3}}}^{\frac{c_{2}c^{8}n}{256}}\leq\veps^{C'n}\,,
\end{align*}
where one can take a proper universal constant $C'$ in the last line.
\end{proof}

\subsection{Trace control via SL\label{subsec:Guan}}

We now present Guan's important uniform bound on the trace of $\Sigma_{t}$:
\[
\E\tr(\Sigma_{t}^{2})=O(n)\,.
\]
This part is technically more complicated than the operator-norm control
in \S\ref{subsec:Operator-norm-control}.

Two types of bounds had been established before Guan's bound.
\begin{itemize}
\item \textbf{Type I}: $\E\tr(\Sigma_{t}^{2})=\O(n)$ for $t\lesssim\log^{-2}n$
(not a universal constant).
\item \textbf{Type II}: For any $t_{0}>0$, $\E\tr(\Sigma_{t}^{2})\leq\bpar{\frac{t}{t_{0}}}^{\Theta(1)}\E\tr(\Sigma_{t_{0}}^{2})$
for all $t\geq t_{0}$. (The best exponent thus far is $2\sqrt{2}$)
\end{itemize}
Type I's valid regime for $t$ ends at a dimension-dependent bound,
but Guan managed to remove this limit.

In this section, we prove the two bounds on $\E\tr(\Sigma_{t}^{2})$
and then move to Guan's bound.

\subsubsection{Two types of control}

\paragraph{Type I. }

This follows from the operator-norm control established in \eqref{eq:final-target}:
for $t\asymp\log^{-2}n$ and constant $C>0$,
\[
\P(\exists\tau\leq t:\norm{\Sigma_{\tau}}\geq2)\leq\exp\bpar{-\frac{1}{Ct}}\,.
\]
Using this and $\norm{\Sigma_{t}}\leq1/t$,
\begin{align*}
\E\tr(\Sigma_{t}^{2}) & =\E\bbrack{\tr(\Sigma_{t}^{2})\,\ind[\norm{\Sigma_{t}}\leq2]}+\E\bbrack{\tr(\Sigma_{t}^{2})\,\ind[\norm{\Sigma_{t}}\geq2]}\\
 & \lesssim n+\frac{n}{t^{2}}\,\P(\norm{\Sigma_{t}}\geq2)\leq n+\frac{n}{t^{2}}\,e^{-\frac{1}{Ct}}\\
 & \lesssim n\,.
\end{align*}

\paragraph{Type II.}

Just as we work with $h(M)=\frac{1}{\beta}\,\log(\tr e^{\beta M})$
when handling the operator norm, we directly work with the following
proxy function:
\[
h(M)=\tr(M^{2})\,.
\]
Recall the $3$-tensor $H$ from \eqref{eq:3-tensor}:
\[
H_{ijk}=\E_{\pi_{t}}[(x-b_{t})_{i}(x-b_{t})_{j}(x-b_{t})_{k}]\,,\qquad H_{i}:=\E_{\pi_{t}}[(x-b_{t})^{\otimes2}(x-b_{t})_{i}]\,,
\]
and that 
\[
\D\Sigma_{t}=\sum_{i}H_{i}\,\D W_{t,i}-\Sigma_{t}^{2}\,\D t\,.
\]
Then,
\begin{align*}
\D h(\Sigma_{t}) & =\nabla h(\Sigma_{t})[\D\Sigma_{t}]+\half\,\hess h(\Sigma_{t})[\D\Sigma_{t},\D\Sigma_{t}]=2\tr(\Sigma_{t}\,\D\Sigma_{t})+\tr(\D[\Sigma]_{t})\\
 & =2\sum_{i}\tr(\Sigma_{t}H_{i}\,\D W_{t,i})-2\tr(\Sigma_{t}^{3})\,\D t+\sum_{i}\tr(H_{i}^{2})\,\D t\,,
\end{align*}
so $\de_{t}\E\tr(\Sigma_{t}^{2})=\sum_{i}\tr(H_{i}^{2})-2\tr(\Sigma_{t}^{3})\leq\sum_{i}\tr(H_{i}^{2})$.
Due to $\norm H_{\hs}^{2}:=\sum_{ijk}H_{ijk}^{2}=\sum_{i}\norm{H_{i}}_{\hs}^{2}=\sum_{i}\tr(H_{i}^{2})$,
\begin{equation}
\de_{t}\E\tr(\Sigma_{t}^{2})\le\E\Bbrack{\sum_{ijk}H_{ijk}^{2}}\,.\label{eq:tr-ineq}
\end{equation}

We need one observation and technical lemma:
\begin{prop}
$\norm H_{\hs}^{2}=\sum_{ijk}H_{ijk}^{2}=\E_{X,Y\sim\pi_{t}}[(X\cdot Y)^{3}]$,
where $X$ and $Y$ are independent and centered.
\end{prop}

\begin{proof}
For $X\cdot Y=\sum_{i}x_{i}y_{i}$,
\[
(X\cdot Y)^{3}=\Bpar{\sum_{i}x_{i}y_{i}}^{3}=\sum_{ijk}x_{i}x_{j}x_{k}y_{i}y_{j}y_{k}\,,
\]
so
\[
\E_{\pi_{t}}[(X\cdot Y)^{3}]=\sum_{ijk}\E[x_{i}x_{j}x_{k}]\,\E[y_{i}y_{j}y_{k}]=\sum_{ijk}H_{ijk}^{2}=\norm H_{\hs}^{2}\,.
\]
\end{proof}
Thus, we now focus on the third-order moment $\E_{\pi_{t}}[(X\cdot Y)^{3}]$.
We could just use a Poincar\'e inequality to bound this, but we can
bound it more tightly through It\^o calculus. To this end, we state
a preliminary lemma without proof.
\begin{prop}
[{\cite[Lemma 4.1]{KL22Bourgain}}] Let $\pi$ be centered and $1$-strongly
logconcave over $\Rn$, and $(W_{t})$ be a standard Brownian motion
with $W_{0}=0$. Then, there exists an adapted process $(M_{t})$
of symmetric matrices such that $\int_{0}^{1}M_{t}\,\D W_{t}$ has
the same law with $\pi$ and that almost surely $0\preceq M_{t}\preceq I_{n}$
for all $t\in[0,1]$.
\end{prop}

Thus, interpolating $X\sim\pi_{t}$ via another stochastic process,
we can show the following:
\begin{lem}
\label{lem:third-moment-bound} For $t>0$, let $\pi$ be centered
and $t$-strongly logconcave over $\Rn$. Let $X,Y\sim\pi$ be independent.
Then, for $\Sigma=\cov\pi$,
\[
\E[(X\cdot Y)^{3}]\leq\frac{5}{t}\,\E[(X\cdot Y)^{2}]=\frac{5}{t}\,\tr(\Sigma^{2})\,.
\]
\end{lem}

\begin{proof}
We may assume that $t=1$ due to homogeneity. Using the interpolation
of $X$, define 
\[
\D X_{t}:=M_{t}\,\D W_{t}\,.
\]
Note that It\^o's formula leads to $\D\E[X_{t}^{\otimes2}]=\E[M_{t}^{2}]\,\D t=:Q_{t}\,\D t$,
so 
\[
\Sigma_{t}:=\E[X_{t}^{\otimes2}]=\int_{0}^{t}Q_{s}\,\D s\,,
\]
and in particular, $\Sigma=\int_{0}^{1}Q_{s}\,\D s$.

Let $Y$ be independent of $W_{t}$ above. By It\^o's formula,
\begin{align*}
\D(X_{t}\cdot Y)^{3} & =3\,(X_{t}\cdot Y)^{2}\,(Y\cdot\D X_{t})+3\,(X_{t}\cdot Y)\,(Y^{\T}\,\D[X]_{t}\,Y)\\
 & =3\,(X_{t}\cdot Y)^{2}\,(Y^{\T}M_{t}\,\D W_{t})+3\,(X_{t}\cdot Y)\,\norm{M_{t}Y}^{2}\,\D t\,.
\end{align*}
Integrating both sides over $[0,1]$ and taking expectation, due to
$X_{0}=0$,
\[
\E[(X\cdot Y)^{3}]=3\int_{0}^{1}\E[(X_{t}\cdot Y)\,\norm{M_{t}Y}^{2}]\,\D t\,.
\]
Then, using \eqref{eq:PI} and $\cpi(\pi)\leq1$,
\begin{align*}
\E_{Y}[(X_{t}\cdot Y)\,\norm{M_{t}Y}^{2}] & =\E\bbrack{(X_{t}\cdot Y)\,(\norm{M_{t}Y}^{2}-\E[\norm{M_{t}Y}^{2}])}\\
 & \leq\sqrt{\E[(X_{t}\cdot Y)^{2}]}\sqrt{\var_{Y}(\norm{M_{t}Y}^{2})}\leq\sqrt{X_{t}^{\T}\Sigma X_{t}}\sqrt{4\cpi(\pi)\,\E[\norm{M_{t}^{2}Y}^{2}]}\\
 & \leq2\sqrt{X_{t}^{\T}\Sigma X_{t}}\sqrt{\tr(M_{t}^{4}\Sigma)}\leq2\sqrt{X_{t}^{\T}\Sigma X_{t}}\sqrt{\tr(M_{t}^{2}\Sigma)}\,,
\end{align*}
where the last line follows from $M_{t}\preceq I$ almost surely for
all $t\in[0,1]$. Putting back to $\E[(X\cdot Y)^{3}]$,
\begin{align*}
\E[(X\cdot Y)^{3}] & \leq6\int_{0}^{1}\E\bbrack{\sqrt{X_{t}^{\T}\Sigma X_{t}}\sqrt{\tr(M_{t}^{2}\Sigma)}}\,\D t\leq6\int_{0}^{1}\sqrt{\E[X_{t}^{\T}\Sigma X_{t}]\,\E[\tr(M_{t}^{2}\Sigma)]}\,\D t\\
 & \leq6\,\Bpar{\int_{0}^{1}\E[X_{t}^{\T}\Sigma X_{t}]\,\E[\tr(M_{t}^{2}\Sigma)]\,\D t}^{1/2}=6\,\Bpar{\int_{0}^{1}\int_{0}^{t}\tr(Q_{s}\Sigma)\,\D s\cdot\tr(Q_{t}\Sigma)\,\D t}^{1/2}\\
 & =\frac{6}{\sqrt{2}}\int_{0}^{1}\tr(Q_{t}\Sigma)\,\D t=\frac{6}{\sqrt{2}}\tr(\Sigma^{2})\leq5\tr(\Sigma^{2})\,,
\end{align*}
which completes the proof.
\end{proof}
With this result in hand, we continue on \eqref{eq:tr-ineq}: 
\[
\de_{t}\E\tr(\Sigma_{t}^{2})\leq\E\Bbrack{\sum_{ijk}H_{ijk}^{2}}=\E\E_{X,Y\sim\pi_{t}}[(X\cdot Y)^{3}]\leq\frac{5}{t}\,\E\tr(\Sigma_{t}^{2})\,.
\]
Solving this,
\[
\E\tr(\Sigma_{t}^{2})\leq\bpar{\frac{t}{t_{0}}}^{5}\,\E\tr(\Sigma_{t_{0}}^{2})\,.
\]

\subsubsection{Guan's approach}

In this section, we sketch the proof of Guan's uniform bound on the
trace.
\begin{thm}
[{\cite{guan2024note}}] \label{thm:trace-control} $\E\tr(\Sigma_{t}^{2})=\Theta(n)$
for $t\lesssim1$, where $\Sigma_{t}$ is the covariance matrix of
$\pi_{t}$ obtained by SL  with isotropic logconcave $\pi_{0}=\pi$.
Moreover, $\E\tr(\Sigma_{t}^{2})=\O(n)$ for any $t>0$.
\end{thm}

This removes the dimension-dependent restriction of $t\lesssim\log^{-2}n$.
We first prove this up to a universal constant time $t\leq c$, and
then bootstrap this to \emph{any} time $t$. 

\paragraph{High-level ideas.}

Guan's idea can be summarized as follows:
\begin{enumerate}
\item Design a proxy function $F_{k,t}=\tr f_{k}(\Sigma_{t})$, which focuses
mainly on \emph{large eigenvalues} of $\Sigma_{t}$ (while suppressing
small ones exponentially). Here, $f_{k}\in C^{2}(\R_{\geq0})$ is
explicitly constructed for two parameters $D_{k}>0$ and $r_{k}\in[7/3,8/3]$,
such that for some constant $b_{k}=\Theta(1)$, $f_{k}(x)$ is defined
as:
\[
f_{k}(x)=\begin{cases}
e^{-D_{k}\,|r_{k}-x|} & x\leq r_{k}-1/D_{k}\,,\\
b_{k}x^{2} & x\geq r_{k}\,,
\end{cases}
\]
and $|f''_{k}(x)|\leq D_{k}^{2}f_{k}(x)$ (see Lemma~\ref{lem:proxy-function}).
By construction, $F_{k,t}$ \emph{almost} bounds $\tr\Sigma_{t}^{2}$
for each $k\in\mbb N$. 
\begin{enumerate}
\item Even though $f_{k}$ is not specified on $[r_{k}-1/D_{k},r_{k}]$,
intuitively $C^{2}(\R_{\geq0})$ is rich enough to $C^{2}$-smoothly
interpolate $e^{-1}(=f_{k}(r_{k}-D_{k}^{-1}))$ and $b_{k}r_{k}^{2}(=f_{k}(r_{k}))$,
while ensuring $|f''_{k}(x)|\leq D_{k}^{2}f_{k}(x)$.
\item Focusing only on the large eigenvalues is beneficial, as eigenvalues
deviating from $\Theta(1)$ are much less likely. Note that the property
of $f_{k}$ ensures that $\E[\tr(\Sigma_{t}^{2})\,\ind\{\norm{\Sigma_{t}}\gtrsim1\}]\approx\E F_{k,t}$
and recall the operator-norm control in \eqref{eq:final-target}.
Thus, since all eigenvalues of $\Sigma_{t}$ are bounded by $1/t$,
we have 
\[
\E F_{k,t}\approx\E[\tr(\Sigma_{t}^{2})\,\ind\{\norm{\Sigma_{t}}\gtrsim1\}]\lesssim\frac{n}{t^{2}}\,e^{-1/t}\approx e^{-\log^{2}t}\,n\quad\text{for }t=t_{0}\,.
\]
\end{enumerate}
\item Establish a Type II bound: for any $t_{0}>0$ and all $t\in[t_{0},t_{0}\vee D_{k}^{-4}]$,
\[
\E F_{k,t}\leq\bpar{\frac{t}{t_{0}}}^{\Theta(1)}\,\E F_{k,t_{0}}\,.
\]
This part leverages the improved Lichnerowicz inequality \eqref{eq:improved-LI}.
\begin{enumerate}
\item Using this bound, the paper implements the following idea: increase
$t$ slightly from $t_{0}$ (which would increase the bound $e^{-\log^{2}t}\,n$
(on $\E F_{k,t_{0}}$) by a polynomial factor $t_{0}^{-\Theta(1)}$),
but \emph{this polynomial increase is balanced by the exponentially
small coefficient of} $e^{-\log^{2}t_{0}}$ in the given bound $e^{-\log^{2}t_{0}}\,n$.
This lemma allows us to keep the coefficient under control until $t_{0}\vee D_{k}^{-4}$.
\end{enumerate}
\item By choosing $r_{k},D_{k}$ suitably, the paper designs a partition
$\bigcup_{k}[t_{k},t_{k+1}]$ covering $[\log^{-2}n,\Theta(1)]$.
By induction on $k$ (i.e., repeating the argument in 2-(a) for each
interval), it shows that 
\[
\E F_{k,t}\leq e^{-\log^{2}t_{k}}\,n\quad\text{for }t\in[t_{k},t_{k+1}]\,.
\]

\begin{enumerate}
\item This allows us to stitch together an $\O(n)$-bound on $\E F_{k,t}$
across the intervals $t\in[t_{k},t_{k+1}]$, up until $\max t_{k}=\Theta(1)$
(i.e., universal-constant time). Then, for any $t\lesssim1$, find
the corresponding interval $[t_{k},t_{k+1}]$ including $t$, and
bound the quantity of interest as follows:
\[
\E\tr(\Sigma_{t}^{2})=\E[\tr(\Sigma_{t}^{2})\,\ind\{\norm{\Sigma_{t}}\lesssim1\}]+\E[\tr(\Sigma_{t}^{2})\,\ind\{\norm{\Sigma_{t}}\gtrsim1\}]\lesssim n+\E F_{k,t}\lesssim n\,.
\]
\end{enumerate}
\end{enumerate}

\paragraph{Detailed sketch.}

We now reflect on Guan's proof in detail and the need for this sophisticated
approach.

\subparagraph{(1) What if we work directly with $\protect\tr(\Sigma_{t}^{2})$
(not $F_{k,t}$)?}

Recall the Type I bound, $\E\tr(\Sigma_{t}^{2})\lesssim n$ up until
$t\lesssim\log^{-2}n$, and the Type II bound given by
\[
\E\tr(\Sigma_{t}^{2})\leq\bpar{\frac{t}{t_{0}}}^{5}\,\E\tr(\Sigma_{t_{0}}^{2})\quad\text{for any }t\geq t_{0}\asymp\log^{-2}n\,.
\]
Let us decompose $\E\tr(\Sigma_{t_{0}}^{2})$ as follows:
\begin{align*}
\E\tr(\Sigma_{t_{0}}^{2}) & =\E\bbrack{\tr(\Sigma_{t_{0}}^{2})\,\ind[\lda_{i}\ \text{small}]}+\E\bbrack{\tr(\Sigma_{t_{0}}^{2})\,\ind[\lda_{i}\ \text{large}]}\\
 & \leq\E\bbrack{\tr(\Sigma_{t_{0}}^{2})\,\ind[\lda_{i}\ \text{small}]}+\frac{n}{t_{0}^{2}}\,\P(\lda_{i}\ \text{large})\\
 & \leq\E\bbrack{\tr(\Sigma_{t_{0}}^{2})\,\ind[\lda_{i}\ \text{small}]}+\frac{n}{t_{0}^{2}}\exp\bpar{-\frac{1}{Ct_{0}}}\,.
\end{align*}
With this in mind, let us slightly increase $t$ in the Type II bound.
As already seen in the high-level sketch, the multiplicative factor
of $(t/t_{0})^{5}\lesssim t_{0}^{-5}$ increases the RHS of the Type
II bound by a $\polylog n$ factor.

As for the second term, for $t_{0}\asymp\log^{-2}n$ and $u:=\log n\asymp t_{0}^{-2}$,
we can write 
\[
\frac{1}{t_{0}^{2}}\exp\bpar{-\frac{1}{Ct_{0}}}\approx\frac{u^{\O(1)}}{e^{u^{\O(1)}}}\,.
\]
For $t=\Theta(1)$, the multiplicative term of $(t/t_{0})^{5}\lesssim t_{0}^{-5}=u^{10}$
doesn't make a meaningful increase in $\frac{u^{\O(1)}}{e^{u^{\O(1)}}}$,
because the exponential term in the denominator overwhelms this additional
$\polylog n$ factor. Hence, when using the Type II bound to relay
$\E\bbrack{\tr(\Sigma_{t_{0}}^{2})\,\ind[\lda_{i}\ \text{large}]}$
up to $t=\Theta(1)$, we might be able to control it \textbf{without
introducing} $\polylog n$ factors.

As for the first term, the best bound one can hope for is
\[
\E\bbrack{\tr(\Sigma_{t_{0}}^{2})\,\ind[\lda_{i}\ \text{small}]}\lesssim n\,.
\]
There is no slack in this bound (since we already know $\E\tr(\Sigma_{t}^{2})\lesssim n$
for $t\geq0$). However, as opposed to the second term, this $n$
bound on the first term has no exponentially small factors (like $e^{-\polylog n}$),
which (if exists) would keep the polynomial blow-up by $t_{0}^{-5}$
under control. Thus, if we na\"ively stitch and relay the bound on
$\E\tr(\Sigma_{t}^{2})$ across the intervals by using the Type II
bound, then the bound on the first term would end up introducing an
extra $\polylog n$ factor at $t=\Theta(1)$.

\subparagraph{(2) Avoid small eigenvalues.}

From the discussion above, we are prompted to try \emph{to}\textbf{\emph{
only relay}}\emph{ a very slowly increasing bound on $\E\bbrack{\tr(\Sigma_{t}^{2})\,\ind[\lda_{i}\ \text{\emph{large}}]}$}.
If we can do so, when bounding $\E\tr(\Sigma_{t}^{2})$ for any $t\lesssim1$,
we could bound the first term by $n$ as we already did, while bounding
the second term by the relayed small bound on \emph{$\E\bbrack{\tr(\Sigma_{\Theta(1)}^{2})\,\ind[\lda_{i}\ \text{\emph{large}}]}\lesssim n$.
}To this end, Guan introduces the following proxy function which almost
kills small eigenvalues by making them exponentially small.
\begin{lem}
[Proxy function with focus on large eigenvalues] \label{lem:proxy-function}
For two parameters $D>0$ and $r\in[7/3,8/3]$, there exists a $C^{2}(\R)$
function $f=f_{D,r}$ such that for some $b\in[1/20,1/5]$,
\[
f(x)=\begin{cases}
e^{-D\,|r-x|} & x\leq r-1/D\,,\\
bx^{2} & x\geq r\,,
\end{cases}
\]
(1) is increasing and (2) satisfies $\Abs{f''(x)}\leq D^{2}f(x)$
on $x\in\R$.
\end{lem}

Since this $f$ becomes quadratic when $x$ is above $r$, it \emph{almost}
bounds $\tr\Sigma_{t}^{2}$. Also, $x=r-1/D$ is where the exponential-cutoff
starts, so any eigenvalues smaller than (say) $r-1/\sqrt{D}$ will
be converted to an exponentially small number like $e^{-\sqrt{D}}$.
Then, by setting $D=\log^{\O(1)}n$, we can make the first term (i.e.,
the sum of small eigenvalues of $\Sigma_{t}^{2}$) exponentially small
like $ne^{-\polylog n}$, which looks very similar to the second term
in the bound of $\E\tr(\Sigma_{t_{0}}^{2})$!

Fortunately, this proxy function also enjoys a Type II bound:
\begin{lem}
[Type II  bound]\label{lem:typeII-guan} Define $F_{t}=F_{t}(\Sigma_{t}):=\sum_{i=1}^{n}f(\lda_{i,t})$
for the eigenvalues $\{\lda_{i,t}\}_{i\in[n]}$ of $\Sigma_{t}$.
Then, for any $t_{0}\in(0,1]$ and $t_{0}\leq t\leq t_{0}\vee D^{-4}$,
\[
\E F_{t}\leq\bpar{\frac{t}{t_{0}}}^{1000}\,\E F_{t_{0}}\,.
\]
\end{lem}

This lemma allows us to carry out the analytical idea discussed above:
(1) partition $[\log^{-2}n,\Theta(1)]$ into $\{[t_{k},t_{k+1}]\}_{k\geq0}$
and (2) for any $t\in[t_{k},t_{k+1}]$, pass a given bound on $\E F_{t_{k}}$
to a bound on $\E F_{t}$.

\subparagraph{(3) Additional indexing by $k$.}

At this point, it is quite tempting to directly transfer the second-term
bound from $t_{0}=\log^{-2}n$ to $t=\Theta(1)$, since the first
term is also exponentially small that the multiplicative term of $t_{0}^{-1000}=\log^{2000}n$
is dominated by $e^{-\polylog n}$. However, one should notice that
the lemma above is valid only on 
\[
t_{0}\leq t\leq t_{0}\vee\frac{1}{D^{4}}\,.
\]
Given that $D$ would be set very large (e.g., $D\geq\polylog n$),
the valid regime of the lemma is actually very short (e.g., the valid
length is at most $D^{-4}\lesssim\log^{-\O(1)}n$). 

A strategy at this point is quite obvious: 
\begin{itemize}
\item Properly design the intervals $[t_{k},t_{k+1}]$ so that $[\log^{-2}n,\Theta(1)]\subset\cup_{k}[t_{k},t_{k+1}]$.
\item For each interval $[t_{k},t_{k+1}]$, define a corresponding proxy
function $F_{k,t}$ (in terms of two parameters $D_{k}$ and $r_{k}$)
such that the Type II bound can be applied to $\E F_{k,t}$ on $[t_{k},t_{k+1}]$
(i.e., $|t_{k+1}-t_{k}|\leq D_{k}^{-4}$).
\item Once we have bounded $\E F_{k,t}$ on $t\in[t_{k},t_{k+1}]$ (in particular,
a bound on $\E F_{k,t_{k+1}}$), we should pass this bound to $\E F_{k+1,t}$
on $t\in[t_{k+1},t_{k+2}]$. By induction, Guan shows the intended
bound on $\E[\tr(\Sigma_{t}^{2})\,\ind\{\lda_{i}\ \text{large}\}]$:
\end{itemize}
\begin{lem}
\label{lem:induction-guan} For any $k\in\mathbb{N}$,
\[
\E F_{k,t}\leq e^{-\log^{2}t_{k}}n\leq n\quad\text{for }t\in[t_{k},t_{k+1}]\,.
\]
\end{lem}

This allows us to invoke this $\O(n)$-bound on the large part of
eigenvalues for any $t\lesssim1$. Then, combining with an obvious
$\O(n)$-bound on the first term (i.e., the sum of square of small
eigenvalues), we can achieve $\O(n)$-bound on $\E\tr(\Sigma_{t}^{2})$
for any $t\lesssim1$.

\subparagraph{(4) Lower bound and extension to any time.}

The lower bound of $\Omega(n)$ readily follows from the upper bound.
As $\D\Sigma_{t}=\D H_{t}-\Sigma_{t}^{2}\,\D t\,,$we have 
\[
\de_{t}\E\tr\Sigma_{t}=-\E\tr(\Sigma_{t}^{2})\,.
\]
Using $\E\tr(\Sigma_{t}^{2})\lesssim n$ for $t=\Theta(1)$, we obtain
that 
\[
\E\tr\Sigma_{t}\gtrsim n\,.
\]
Noting that $\E\tr^{2}\Sigma_{t}\gtrsim n^{2}$ (by Jensen) and
\[
\tr^{2}\Sigma_{t}=(\lda_{1}+\cdots+\lda_{n})^{2}\leq(1+\cdots+1)\,(\lda_{1}^{2}+\cdots+\lda_{n}^{2})=n\tr(\Sigma_{t}^{2})\,,
\]
we can conclude that 
\[
\E\tr(\Sigma_{t}^{2})\gtrsim n\,.
\]

Generalization of $\O(n)$-bound to \emph{any time} $t$ is also obvious:
for any $t\gtrsim1$, 
\[
\E\tr(\Sigma_{t}^{2})\leq\frac{n}{t^{2}}\lesssim n\,.
\]

\paragraph{Type II bound via \eqref{eq:improved-LI}.}

We prove the bound of type II for the proxy function $F_{t}$ (Lemma~\ref{lem:typeII-guan}),
following \cite{guan2024note} and Klartag's note. To streamline computations
later, we need two preliminaries --- (i) tensor shuffling and (ii)
marginalization.

\subparagraph{(1) Tensor shuffling. }

We denote the eigenvalues of $\Sigma_{t}=\cov\pi_{t}$ by $0<\lda_{1}\leq\dots\leq\lda_{n}$,
and their corresponding eigenvectors by $\{u_{i}\}_{i\in[n]}$. Denoting
$x_{i}:=\inner{x,u_{i}}$, we recall the third-order tensor $H_{t}$
defined by 
\[
H_{ijk}:=\int(x-b_{t})_{i}\,(x-b_{t})_{j}\,(x-b_{t})_{k}\,\pi_{t}(\D x)\,,
\]
and denote its fiber by
\[
H_{ij}:=(H_{ij1},\dots,H_{ijn})\in\Rn\,.
\]
We also define a \emph{threshold index} by
\[
\tau(\lda)=\max\{i:\lda_{i}\leq\lda\}\,.
\]

We are going to use the following lemma often:
\begin{lem}
[Tensor bound] For any $r>0$,
\begin{equation}
\sum_{\substack{i,j,k=1\\
i\geq\tau(r)
}
}^{n}H_{ijk}^{2}\leq3\sum_{i\geq\tau(r)}\sum_{j,k=1}^{i}H_{ijk}^{2}\,.\tag{\ensuremath{\msf{ts}}-\ensuremath{\msf{shuffle}}}\label{eq:tensor-shuffle}
\end{equation}
\end{lem}

\begin{proof}
It is almost immediate from the symmetry of $H_{ijk}$ in $ijk$ and
a simple combinatorial argument:
\[
\sum_{\substack{i,j,k=1\\
i\geq\tau(r)
}
}^{n}H_{ijk}^{2}\leq\sum_{i\geq j,k}H_{ijk}^{2}+\sum_{j\geq i,k}H_{ijk}^{2}+\sum_{k\geq i,j}H_{ijk}^{2}\leq3\sum_{i\geq\tau(r)}\sum_{j,k=1}^{i}H_{ijk}^{2}\,,
\]
which completes the proof.
\end{proof}

\subparagraph{(2) Marginalization. }

When bounding the expectation of drift of $\D\E F_{k,t}$, we will
run into $\sum_{i,j}^{\tau(r)}\Abs{H_{ijk}}^{2}$ quite often, which
can be viewed as a (partial) marginalization of $i,j$.
\begin{lem}
[Partial marginalization] For any $r>0$, $k\geq\tau(r)$, and $f(x)=bx^{2}$,
\begin{equation}
\sum_{i,j=1}^{\tau(r)}H_{ijk}^{2}\leq\frac{4r^{3/2}\lda_{k}}{t^{1/2}}\,,\qquad\&\qquad b\sum_{i,j=1}^{k}H_{ijk}^{2}\leq\frac{4f(\lda_{k})}{t}\,.\tag{\ensuremath{\msf p}-\ensuremath{\msf{marginal}}}\label{eq:partial-marginalization}
\end{equation}
\end{lem}

\begin{proof}
Note that
\begin{align*}
\sum_{i,j=1}^{\tau(r)}H_{ijk}^{2} & =\int\underbrace{\sum_{ij}^{\tau(r)}H_{ijk}(x-b_{t})_{i}(x-b_{t})_{j}}_{=:g_{k}(x)}(x-b_{t})_{k}\,\pi_{t}(\D x)\\
 & =\int g_{k}(x)\,(x-b_{t})_{k}\,\pi_{t}(\D x)\leq\sqrt{\var_{\pi_{t}}g_{k}}\sqrt{\lda_{k}}\,.
\end{align*}

Let $\mc S$ be the subspace spanned by $\{u_{1},\dots,u_{\tau(r)}\}$,
and $P_{\mc S}:\Rn\to\R^{\tau(r)}$ the projection function onto $\mc S$.
Clearly, $\pi_{\mc S}:=(P_{\mc S})_{\#}\pi_{t}$ is not only $1/t$-strongly
logconcave (Pr\'ekopa--Leindler) but also $\norm{\cov\pi_{\mc S}}\leq r$
by its construction. Since $g_{k}$ has no components in $\mc S^{\perp}$,
by \eqref{eq:PI} with $\cpi(\pi_{\mc S})\le(\norm{\cov\pi_{\mc S}}/t)^{1/2}\le(r/t)^{1/2}$
(due to \eqref{eq:improved-LI}),
\[
\var_{\pi_{t}}g_{k}=\var_{\pi_{\mc S}}g_{k}\leq\bpar{\frac{r}{t}}^{1/2}\,\E_{\pi_{\mc S}}[\norm{\nabla_{\mc S}g_{k}}^{2}]=4\,\bpar{\frac{r}{t}}^{1/2}\sum_{i,j=1}^{\tau(r)}H_{ijk}^{2}\lda_{j}\leq\frac{4r^{3/2}}{t^{1/2}}\sum_{i,j=1}^{\tau(r)}H_{ijk}^{2}\,.
\]
Putting this back to the display above,
\[
\sum_{i,j=1}^{\tau(r)}H_{ijk}^{2}\le\Bpar{\frac{4r^{3/2}}{t^{1/2}}\sum_{i,j=1}^{\tau(r)}H_{ijk}^{2}}^{1/2}\sqrt{\lda_{k}}\,,
\]
so 
\[
\sum_{i,j=1}^{\tau(r)}H_{ijk}^{2}\leq\frac{4r^{3/2}\lda_{k}}{t^{1/2}}\,,
\]
which proves the first part. 

As for the second claim, using the first claim with $\tau^{-1}(k)=\lda_{k}$
and $\lda_{k}\leq1/t$,
\[
b\sum_{i,j=1}^{k}H_{ijk}^{2}\leq\frac{4b\lda_{k}^{5/2}}{t^{1/2}}\leq\frac{4b\lda_{k}^{2}}{t}=\frac{4f(\lda_{k})}{t}\,,
\]
which completes the proof.
\end{proof}
We are now ready to prove the bound of type II on $\E F_{t}$. The
following It\^o derivative of eigenvalues is well known \cite{Klartag21yuansi}:
\begin{prop}
[It\^o derivative of eigenvalues] For any smooth function $f:\R_{\geq0}\to\R$
with bounded second derivative,
\begin{equation}
\de_{t}\E\tr f(\Sigma_{t})=-\sum_{i=1}^{n}\E[\lda_{i}^{2}f'(\lda_{i})]+\half\sum_{i,j=1}^{n}\E\bbrack{\norm{H_{ij}}^{2}\,\frac{f'(\lda_{i})-f'(\lda_{j})}{\lda_{i}-\lda_{j}}}\,,\label{eq:ito-proxy}
\end{equation}
where the quotient is interpreted as $f''(\lda_{i})$ when $\lda_{i}=\lda_{j}$.
\end{prop}

All we need to do now is bound the magnitude of the drift (i.e., the
RHS) for the proxy function $f$ (Lemma~\ref{lem:proxy-function}):
given two parameters $D>0$ and $r\in[7/3,8/3]$, there exists $b\in[1/20,1/5]$
such that $f$ is positive and increasing, and satisfies
\[
\Abs{f''}\leq D^{2}f\qquad\&\qquad f(x)=bx^{2}\quad\text{for }x\geq r\,.
\]

\begin{proof}
[Proof of Lemma~\ref{lem:typeII-guan}] Since $f$ is increasing,
we bound the first term in \eqref{eq:ito-proxy} by $0$, focusing
on the second term. We partition $\{(i,j)\in[n]\times[n]\}$ into
four pieces:
\begin{align*}
(i): & i,j>\tau(r)\\
(ii): & i\leq\tau(r),\,j\geq\tau\bpar{r+\frac{1}{3}}\\
(iii): & j\leq\tau(r),\,i\geq\tau\bpar{r+\frac{1}{3}}\\
(iv): & i,j\leq\tau\bpar{r+\frac{1}{3}}
\end{align*}
so we slightly overcount $\{(i,j):\tau(r)\leq i,j\leq\tau(r+1/3)\}$.
Also, $(ii)$ and $(iii)$ are essentially the same due to symmetry.

As for $(i)$, using $f'(x)=2bx$ for $x\geq r$,
\begin{align*}
\half\sum_{i,j\geq\tau(r)+1}^{n}\norm{H_{ij}}^{2}\,\frac{f'(\lda_{i})-f'(\lda_{j})}{\lda_{i}-\lda_{j}} & =b\sum_{i,j\geq\tau(r)+1}\sum_{k}H_{ijk}^{2}\\
 & \underset{\eqref{eq:tensor-shuffle}}{\leq}3b\sum_{i\geq\tau(r)+1}\sum_{j,k=1}^{i}H_{ijk}^{2}\\
 & \underset{\eqref{eq:partial-marginalization}}{\leq}\frac{12}{t}\sum_{i\geq\tau(r)+1}f(\lda_{i})\,.
\end{align*}

As for $(ii)$ (and $(iii)$), using $r\leq8/3$,
\begin{align*}
\half\sum_{j\leq\tau(r)}\sum_{i\geq\tau(r+\frac{1}{3})}\norm{H_{ij}}^{2}\,\frac{f'(\lda_{i})-f'(\lda_{j})}{\lda_{i}-\lda_{j}} & \leq b\sum_{j\leq\tau(r)}\sum_{i\geq\tau(r+\frac{1}{3})}\sum_{k}H_{ijk}^{2}\,\frac{\lda_{i}}{\lda_{i}-\lda_{j}}\\
 & \leq3\bpar{r+\frac{1}{3}}\,b\sum_{j\leq\tau(r)}\sum_{i\geq\tau(r+\frac{1}{3})}\sum_{k}H_{ijk}^{2}\\
 & \leq3\bpar{r+\frac{1}{3}}\,b\sum_{i\geq\tau(r+\frac{1}{3})}\sum_{j,k}H_{ijk}^{2}\\
 & \underset{\eqref{eq:tensor-shuffle}}{\leq}27b\sum_{i\geq\tau(r+\frac{1}{3})}\sum_{j,k}^{i}H_{ijk}^{2}\\
 & \underset{\eqref{eq:partial-marginalization}}{\leq}\frac{108}{t}\sum_{i\geq\tau(r+\frac{1}{3})}f(\lda_{i})
\end{align*}

As for $(iv)$, using $\Abs{f''}\leq D^{2}f$,
\[
\half\sum_{i,j\leq\tau(r+\frac{1}{3})}\norm{H_{ij}}^{2}\,\frac{f'(\lda_{i})-f'(\lda_{j})}{\lda_{i}-\lda_{j}}\leq\frac{D^{2}}{2}\sum_{i,j\leq\tau(r+\frac{1}{3})}\norm{H_{ij}}^{2}\,\bpar{f(\lda_{i})\vee f(\lda_{j})}\,.
\]
We then further separate the sum into two cases: $k\leq\tau(r+1/3)$
or $k>\tau(r+1/3)$. For the former,
\begin{align*}
\frac{D^{2}}{2}\sum_{i,j,k\leq\tau(r+\frac{1}{3})}H_{ijk}^{2}\,\bpar{f(\lda_{i})\vee f(\lda_{j})} & \leq\frac{D^{2}}{2}\sum_{i\leq\tau(r+\frac{1}{3})}f(\lda_{i})\sum_{j,k}^{i}H_{ijk}^{2}+\frac{D^{2}}{2}\sum_{j\leq\tau(r+\frac{1}{3})}f(\lda_{j})\sum_{i,k}^{j}H_{ijk}^{2}\\
 & \quad+\frac{D^{2}}{2}\sum_{k\leq\tau(r+\frac{1}{3})}\sum_{i,j}^{k}H_{ijk}^{2}f(\lda_{k})\\
 & =\frac{3D^{2}}{2}\sum_{i\leq\tau(r+\frac{1}{3})}f(\lda_{i})\sum_{j,k}^{i}H_{ijk}^{2}\\
 & \underset{\eqref{eq:partial-marginalization}}{\leq}\frac{6D^{2}}{t^{1/2}}\sum_{i\leq\tau(r+\frac{1}{3})}f(\lda_{i})\,\lda_{i}^{5/2}\\
 & \leq\frac{6D^{2}\,(r+1/3)^{5/2}}{t^{1/2}}\sum_{i\leq\tau(r+\frac{1}{3})}f(\lda_{i})\\
 & \leq\frac{100D^{2}}{t^{1/2}}\sum_{i\leq\tau(r+\frac{1}{3})}f(\lda_{i})\,.
\end{align*}
For the latter,
\begin{align*}
\frac{D^{2}}{2}\sum_{i,j\leq\tau(r+\frac{1}{3})}\sum_{k>\tau(r+\frac{1}{3})}H_{ijk}^{2}\,\bpar{f(\lda_{i})\vee f(\lda_{j})} & \leq\frac{D^{2}b\,(r+1/3)^{2}}{2}\sum_{i,j\leq\tau(r+\frac{1}{3})}\sum_{k>\tau(r+\frac{1}{3})}H_{ijk}^{2}\\
 & \leq\frac{D^{2}b\,(r+1/3)^{2}}{2}\sum_{k>\tau(r+\frac{1}{3})}\sum_{ij}^{\tau(r+\frac{1}{3})}H_{ijk}^{2}\\
 & \underset{\eqref{eq:partial-marginalization}}{\le}\frac{4D^{2}\,(r+1/3)^{7/2}}{2t^{1/2}}\sum_{k>\tau(r+\frac{1}{3})}b\lda_{k}\\
 & \leq\frac{4D^{2}\,(r+1/3)^{9/2}}{2t^{1/2}}\sum_{k>\tau(r+\frac{1}{3})}f(\lda_{k})\\
 & \lesssim\frac{D^{2}}{t^{1/2}}\sum_{i>\tau(r+\frac{1}{3})}f(\lda_{i})\,.
\end{align*}

Adding up these five sums, we can conclude that
\[
\de_{t}\E\tr f(\Sigma_{t})\leq\O\bpar{\frac{1}{t}+\frac{D^{2}}{t^{1/2}}}\sum_{i=1}^{n}\E[f(\lda_{i})]\leq\O\bpar{\frac{1}{t}+\frac{D^{2}}{t^{1/2}}}\,\E\tr f(\Sigma_{t})\,.
\]
Solving this differential inequality leads to the desired bound.
\end{proof}

\paragraph{$n$-bound on the large eigenvalues.}

We prove an $O(n)$-bound on $\E[\tr(\Sigma_{t}^{2})\,\ind\{\lda_{i}\ \text{large}\}]$
for any $t\lesssim1$ through $\E F_{k,t}$ (Lemma~\ref{lem:induction-guan}).
To this end, we carefully design $t_{k},r_{k},D_{k}$.

Recall that we can take $t_{1}\asymp\log^{-2}n$ such that
\[
\P(\exists t\in[0,t_{1}]:\norm{\Sigma_{t}}\geq2)\leq\frac{1}{n}\,.
\]
Set $r_{1}=7/3$. For $k\geq2$, we set
\begin{align*}
t_{k} & =\log^{-16}t_{k-1}\,,\quad r_{k}=r_{1}+\sum_{i=2}^{k}\frac{1}{\sqrt{\Abs{\log t_{k}}}}\,,\quad D_{k}=\log^{4}t_{k}\,,\\
 & \quad f_{k}=f_{D_{k},r_{k}}\,,\quad F_{k,t}=\tr f_{k}(\Sigma_{t})\,,
\end{align*}
and $b_{k}\in[1/20,1/5]$ denotes the given constant for $f_{k}$
defined in Lemma~\ref{lem:typeII-guan}. 

Let $k_{0}=\max\{k:t_{k}\leq e^{-100}\}$. Then, one can show that 
\begin{enumerate}
\item $r_{k}\in[7/3,8/3]$.
\item $0<t_{1}<t_{2}\leq\cdots\leq t_{k_{0}}<e^{-100}$. 
\item For $2\leq k\leq k_{0}$, it clearly holds that 
\[
r_{k-1}=r_{k}-\frac{1}{\sqrt{\Abs{\log t_{k}}}}\leq r_{k}-\frac{1}{\log^{4}t_{k}}\,.
\]
\item For $x\ge r_{k-1}$, $f_{k}(x)\leq5f_{k-1}(x)$.
\end{enumerate}
We are now ready to show Lemma~\ref{lem:induction-guan}. By induction
on $k\in\mathbb{N}$, let us show that for $t\in[t_{k},t_{k+1}]$,
\[
\E F_{k,t}\leq e^{-\log^{2}t_{k}}\,n\,.
\]
When $k=1$,
\[
\E F_{1,t_{1}}\leq nf_{1}(2)+\frac{1}{n}\,nf_{1}(1/t_{1})\leq ne^{-(\log^{4}t_{1})/3}+b_{1}t_{1}^{-2}\,.
\]
Using the Type II bound (Lemma~\ref{lem:typeII-guan}), for any $t\in[t_{1},t_{2}]$,
\[
\E F_{1,t}\leq\bpar{\frac{t_{2}}{t_{1}}}^{1000}\,\E F_{1,t_{1}}\leq t_{1}^{-1000}\,\bpar{ne^{-(\log^{4}t_{1})/3}+b_{1}t_{1}^{-2}}\leq ne^{-\log^{2}t_{1}}\,.
\]

Now we prove the claim for $k\geq2$: since $r_{k-1}\leq r_{k}-\log^{-4}t_{k}=r_{k}-D_{k}^{-1}$
and $f_{k}\leq5f_{k-1}$ when $x\geq r_{k-1}$, 
\begin{align*}
\E F_{k,t_{k}} & =\E\sum_{i\leq\tau(r_{k-1})}f_{k}(\lda_{i,t_{k}})+\E\sum_{i>\tau(r_{k-1})}f_{k}(\lda_{i,t_{k}})\leq ne^{D_{k}\,(r_{k-1}-r_{k})}+5\E\sum_{i>\tau(r_{k-1})}f_{k-1}(\lda_{i,t_{k}})\\
 & \leq ne^{-\Abs{\log t_{k}}^{3.5}}+5\E F_{k-1,t_{k}}\leq ne^{-\Abs{\log t_{k}}^{3.5}}+5ne^{-\abs{\log t_{k-1}}^{2}}\,.
\end{align*}
where we used the induction hypothesis for $k-1$. By the Type II
bound, for $t\in[t_{k},t_{k+1}]$,
\[
\E F_{k,t}\leq t_{k}^{-1000}\,(ne^{-\Abs{\log t_{k}}^{3.5}}+5ne^{-\abs{\log t_{k-1}}^{2}})\leq ne^{-\log^{2}t_{k}}\,.
\]

\section{The Thin-shell Conjecture \label{sec:thinshell}}

Recall that the KLS conjecture implies the thin-shell conjecture,
which in turn implies the slicing conjecture. Using a combination
of the backdoor approach and the SL argument, we addressed the strongest
of these (KLS) and even resolved the weakest one (slicing). For the
intermediate conjecture (thin-shell), however, previous approaches
required more sophisticated developments in spectral and semigroup
theory. In fact, we reviewed Guan's $O(\log\log n)$-bound on the
thin-shell constant in the first version of our survey (which we moved
to the appendix in this version for interested readers).

The thin-shell conjecture was resolved in 2025 by Klartag and Lehec
\cite{klartag2025thin}. In this section, we review their approach
based on parallel coupling of exponentially-tilted measures and a
refinement of Guan's bound. This approach does not heavily rely on
spectral theory, and instead works with optimal transport and parallel
coupling (from non-linear filtering theory). The starting point is
Barthe and Klartag's $H^{-1}$-inequality \cite{barthe2019spectral}:
\begin{lem}
Let $\pi$ be a logconcave distribution over $\Rn$. For any smooth
function $f$ such that $f,\norm{\nabla f}\in L^{2}(\pi)$ and $\int\nabla f\,\D\pi=0$,
\[
\var_{\pi}f\leq\norm{\nabla f}_{H^{-1}(\pi)}^{2}:=\sum_{i=1}^{n}\norm{\de_{i}f}_{H^{-1}(\pi)}^{2}\,,
\]
where for $g\in L^{2}(\pi)$ with $\int g\,\D\pi=0$, its $H^{-1}$-norm
is defined as 
\begin{align*}
\norm g_{H^{-1}(\pi)} & =\sup\Bbrace{\int g\vphi\,\D\pi:\textup{locally Lipschitz }\vphi\in L^{2}(\pi)\ \textup{such that }\int\norm{\nabla\vphi}^{2}\leq1}\\
 & =\sup\Bbrace{\int g\vphi\,\D\pi:\vphi\in C_{c}^{\infty}(\Rn),\,\int\norm{\nabla\vphi}^{2}\,\D\pi\leq1}\,.
\end{align*}
\end{lem}

The last line follows from the fact \cite{barthe2019spectral} that
the space $C_{c}^{\infty}(\Rn)$ of smooth, compactly-supported functions
in $\Rn$ is a dense subspace of $H^{1}(\pi)$. Here, $H^{1}(\pi)$
is the collection of all functions $f\in L^{2}(\pi)$ with weak partial
derivatives in $L^{2}(\pi)$, equipped with the norm 
\[
\norm f_{H^{1}(\pi)}^{2}:=\int\abs f^{2}\,\D\pi+\int\norm{\nabla f}^{2}\,\D\pi\,.
\]

The $H^{-1}$-inequality looks quite similar to the Poincar\'e inequality.
A benefit of this inequality is obvious when $f(\cdot)=\norm{\cdot}^{2}$
and $\pi$ is isotropic:
\begin{equation}
\var_{\pi}(\norm X^{2})\leq4\,\norm x_{H^{-1}(\pi)}^{2}=4\sum_{i=1}^{n}\norm{x_{i}}_{H^{-1}(\pi)}^{2}\,,\label{eq:hinverse-for-thin}
\end{equation}
so the thin-shell parameter $\sigma_{\pi}^{2}=n^{-1}\var_{\pi}(\norm X^{2})$
is bounded by $n^{-1}\sum_{i=1}^{n}\norm{x_{i}}_{H^{-1}(\pi)}^{2}$
(i.e., the average of the $H^{-1}$-norm across the coordinate functions).

\subsection{High-level ideas}

It suffices to show the following.
\begin{thm}
Let $\pi$ be an isotropic, logconcave probability measure in $\Rn$.
Then,
\[
\sum_{i=1}^{n}\norm{x_{i}}_{H^{-1}(\pi)}^{2}=O(n)\,.
\]
\end{thm}

We take a top-down approach to proving this result, assuming the following
two results. We then provide proof sketches for these results and
give the full details in the next subsection.
\begin{itemize}
\item Let $\pi$ be compactly-supported, centered, logconcave probability
measure in $\Rn$. Then for any $t>0$,
\begin{equation}
\sum_{i=1}^{n}\norm{x_{i}}_{H^{-1}(\pi)}^{2}\leq\frac{1}{t^{2}}\,\E\Bbrack{\sum_{i=1}^{n}\exp\Bpar{2\int_{0}^{t}\lda_{i}(s)\,\D s}}\,,\label{eq:target-bound}
\end{equation}
where $\lda_{1}(t)\geq\cdots\geq\lda_{n}(t)>0$ are the eigenvalues
of $\Sigma_{t}$, the covariance matrix of $\pi_{t}$. In particular,
$\sum_{i=1}^{n}\norm{x_{i}}_{H^{-1}(\pi)}^{2}\leq\E[\sum_{i=1}^{n}\exp(2\int_{0}^{1}\lda_{i}(s)\,\D s)]$
by choosing $t=1$.
\item Let $\tau$ be a stopping time with respect to the natural filtration
of the Brownian motion $(W_{t})_{t\geq0}$. Then for any $t>0$, 
\[
\sum_{i=1}^{n}\P\bpar{\lda_{i}(t\wedge\tau)\geq3}\lesssim n\exp(-t^{-1/8})\,.
\]
\end{itemize}
Note that it suffices to show the main theorem when $\pi$ is \emph{compactly-supported},
isotropic, and logconcave. This follows from the standard fact that
any isotropic logconcave probability measure is the weak limit of
a sequence of compactly-supported isotropic logconcave probability
measures.

Earlier we focused on when the largest eigenvalue $\lda_{1}(t)$ becomes
large (e.g., $\lda_{1}(t)\geq3$) during the SL process, but we now
analyze when this happens for \emph{all} eigenvalues. To this end,
we define the following stopping time:
\[
\tau_{k}:=\inf\{t>0:\lda_{k}(t)\geq3\}\,.
\]
Using $\Sigma_{t}\preceq t^{-1}I_{n}$,
\[
\int_{0}^{1}\lda_{k}(t)\,\D t=\int_{0}^{\tau_{k}\wedge1}\lda_{k}(t)\,\D t+\int_{\tau_{k}\wedge1}^{1}\lda_{k}(t)\,\D t\leq3\,(\tau_{k}\wedge1)-\log(\tau_{k}\wedge1)\,.
\]
Hence,
\[
\exp\Bpar{2\int_{0}^{1}\lda_{k}(t)\,\D t}\leq\exp\bpar{6\,(\tau_{k}\wedge1)-2\log(\tau_{k}\wedge1)}\leq e^{6}\,(\tau_{k}^{-2}+1)\,.
\]
Thus, bounding the RHS of \eqref{eq:target-bound} reduces to bounding
\[
\sum_{k=1}^{n}\E[\tau_{k}^{-2}]=O(n)\,.
\]

For this, we need an in-depth understanding of $\{\lda_{k}(t)\}_{k\in[n]}$.
This is where the second result comes into play.
\begin{lem}
For $k\in[n]$ and $t>0$,
\[
\P(\tau_{k}\leq t)\lesssim\frac{n}{k}\,\exp(-t^{-1/8})\,,\qquad\&\qquad\E[\tau_{k}^{-2}]\lesssim\bpar{1+\log\frac{n}{k}}^{16}\,.
\]
\end{lem}

\begin{proof}
Note that if $\tau_{k}\leq t$, then at time $\tau_{k}=t\wedge\tau_{k}$,
the $k$ largest eigenvalues are at least $3$ by the definition of
$\tau_{k}$. Namely, for $i\in[k]$,
\[
\P(\tau_{k}\leq t)\leq\P\bpar{\lda_{i}(t\wedge\tau_{k})\geq3}\,.
\]
Summing this inequality over $i\in[k]$ and using the second main
result, 
\[
k\,\P(\tau_{k}\leq t)\leq\sum_{i=1}^{k}\P\bpar{\lda_{i}(t\wedge\tau_{k})\geq3}\leq\sum_{i=1}^{n}\P\bpar{\lda_{i}(t\wedge\tau_{k})\geq3}\lesssim n\exp(-t^{-1/8})\,,
\]
which proves the first claim.

For the second claim, fix $x_{0}>0$ to be determined later. Let $\alpha=1/8$.
Then,
\begin{align*}
\E[\tau_{k}^{-2}] & =\int_{0}^{\infty}2x\,\P(\tau_{k}^{-1}\geq x)\,\D x=x_{0}^{2}+\int_{x_{0}}^{\infty}2x\,\P(\tau_{k}^{-1}\geq x)\,\D x\\
 & \lesssim x_{0}^{2}+\frac{n}{k}\int_{x_{0}}^{\infty}x\exp(-x^{-\alpha})\,\D x\\
 & \leq x_{0}^{2}+\frac{n}{k}\,\exp\bpar{-\half\,x_{0}^{-\alpha}}\int_{0}^{\infty}x\exp\bpar{-\half\,x^{-\alpha}}\,\D x\\
 & \leq x_{0}^{2}+O(1)\,,
\end{align*}
where the last line follows from the choice of $x_{0}=2^{1/\alpha}\,(\log\frac{n}{k})^{1/\alpha}$
and $\int_{0}^{\infty}x\exp(-\half\,x^{-\alpha})\,\D x=O(1)$.
\end{proof}
We are only left to show $\sum_{k}\E[\tau_{k}^{-2}]=O(n)$. This easily
follows from
\[
\sum_{k=1}^{n}\E[\tau_{k}^{-2}]\lesssim\sum_{k}\bpar{1+\log\frac{n}{k}}^{16}\lesssim\int_{0}^{n}\bpar{1+\log\frac{n}{x}}^{16}\,\D x=n\int_{0}^{1}\bpar{1+\log\frac{1}{y}}^{16}\,\D y=O(n)\,,
\]
which completes the proof of the thin-shell conjecture.

\paragraph{Recovering Guan's uniform trace bound.}

Using these results, we can recover Guan's uniform trace bound quite
easily. Note that

\[
\lda_{k}^{2}(t)=\lda_{k}^{2}(t)\,\ind[t\leq\tau_{k}]+\lda_{k}^{2}(t)\,\ind[\tau_{k}\leq t]\leq9+\tau_{k}^{-2}\,.
\]
Hence,
\[
\E\tr(\Sigma_{t}^{2})=\sum_{k}\E[\lda_{k}^{2}(t)]\leq O(n)+\sum_{k}\E[\tau_{k}^{-2}]=O(n)\,.
\]

\subsection{Bounding $H^{-1}$-norm by Wasserstein distance with parallel coupling}

Recall the first result. For any $t>0$, 
\[
\sum_{i=1}^{n}\norm{x_{i}}_{H^{-1}(\pi)}^{2}\leq\frac{1}{t^{2}}\,\E\Bbrack{\sum_{i=1}^{n}\exp\Bpar{2\int_{0}^{t}\lda_{i}(s)\,\D s}}\,.
\]
KL starts with a version of the inequality in \cite{Klartag09Berry},
which relates the infinitesimal Wasserstein distance to the $H^{-1}$-norm.
Recall that for any two Borel probability measures $\mu,\nu$ in $\Rn$
and any $p\in[1,\infty)$, the $p$-Wasserstein distance $\mc W_{p}(\mu,\nu)$
between $\mu$ and $\nu$ is defined as 
\[
\mc W_{p}^{p}(\mu,\nu)=\inf_{\gamma\in\Gamma(\mu,\nu)}\int\norm{x-y}^{p}\,\gamma(\D x,\D y)\,,
\]
where $\Gamma(\mu,\nu)$ is the family of \emph{couplings} of $(\mu,\nu)$
(i.e., the probability measure $\gamma\in\Gamma(\mu,\nu)$ in $\Rn\times\Rn$
has marginals $\mu$ and $\nu$, respectively).
\begin{lem}
\label{lem:Hinverse2W2} Let $\pi$ be a centered, compactly-supported
probability measure on $\Rn$. Then, for any $v\in\Rn$,
\[
\norm{\inner{x,v}}_{H^{-1}(\pi)}\leq\limsup_{\veps\to0^{+}}\frac{\mc W_{2}(\pi,\pi_{\veps v})}{\veps}\,,
\]
where $\D\pi_{\veps v}(x)\propto e^{\inner{\veps v,x}}\,\D\pi(x)$
is an exponentially-tilted measure.
\end{lem}

\begin{proof}
WLOG, we may assume that $\mc W_{2}(\pi,\pi_{\veps v})=o(\veps^{1/2})$;
otherwise the RHS becomes infinity. Since $\pi$ is centered, for
the LLT $\Lambda_{t}$ of $\pi$,
\[
\Lambda_{0}(\veps v)=\Lambda_{0}(0)+\inner{\nabla\Lambda_{0}(0),\veps v}+o(\veps)=o(\veps)\,.
\]
Given a smooth compactly-supported function $\vphi:\Rn\to\R$ with
$\norm{\nabla\vphi}_{L^{2}(\pi)}\leq1$, 
\begin{align*}
\int\inner{x,v}\vphi(x)\,\pi(\D x) & =\int\frac{e^{\inner{x,\veps v}}-1}{\veps}\,\vphi(x)\,\pi(\D x)+o(1)=\int\frac{e^{\inner{x,\veps v}-\Lambda_{0}(\veps v)}-1}{\veps}\,\vphi(x)\,\pi(\D x)+o(1)\\
 & =\frac{1}{\veps}\,\Babs{\int\vphi\,\D\pi_{\veps v}-\int\vphi\,\D\pi}+o(1)\,.
\end{align*}
Since $\vphi\in C_{c}^{\infty}(\Rn)$, we can take $R>0$ such that
\[
\abs{\vphi(y)-\vphi(x)}\leq\norm{\nabla\vphi(x)}\norm{y-x}+R\,\norm{y-x}^{2}\,.
\]
Take an arbitrary coupling $(X,Y)\sim(\pi,\pi_{\veps v})$. Then,
\begin{align*}
\Babs{\int\vphi\,\D\pi_{\veps v}-\int\vphi\,\D\pi} & \leq\E\abs{\vphi(X)-\vphi(Y)}\leq\E[\norm{\nabla\vphi(X)}\norm{Y-X}]+R\,\E[\norm{Y-X}^{2}]\\
 & \leq(\E[\norm{X-Y}^{2})^{1/2}+R\,\E[\norm{Y-X}^{2}]\,.
\end{align*}
Taking infimum over all couplings of $(\pi,\pi_{\veps v})$, we deduce
that 
\[
\Babs{\int\vphi\,\D\pi_{\veps v}-\int\vphi\,\D\pi}\leq\mc W_{2}(\pi,\pi_{\veps v})+R\mc W_{2}^{2}(\pi,\pi_{\veps v})\,.
\]
Therefore, 
\[
\int\inner{x,v}\vphi(x)\,\pi(\D x)=\veps^{-1}\mc W_{2}(\pi,\pi_{\veps v})+\veps^{-1}R\mc W_{2}^{2}(\pi,\pi_{\veps v})+o(1)\,.
\]
The claim then follows from $\mc W_{2}(\pi,\pi_{\veps v})=o(\veps^{1/2})$.
\end{proof}
Applying this result with $v=e_{1},\dots,e_{n}$, it suffices to bound
$\mc W_{2}(\pi,\pi_{\veps e_{i}})$. To bound the Wasserstein distance
in a \emph{meaningful} way, we need a nontrivial coupling between
$\pi$ and $\pi_{\veps e_{i}}$ (rather than a trivial coupling such
as $\pi\otimes\pi_{\veps e_{i}}$).

For any given $\pi_{x}$ and $\pi_{y}$, how can we find a sufficiently
good coupling? The idea is to construct a \emph{parallel} coupling
of their stochastically localized measures, and then pull this coupling
to that of $\pi_{x}$ and $\pi_{y}$. This coupling yields the RHS
bound in \eqref{eq:target-bound} for the $\mc W_{2}$-distance. We
will go over the idea behind this shortly.

\subsubsection{Parallel coupling of exponentially-tilted measures}

We recall the \emph{logarithmic Laplace transform} (LLT), exponential
tilts, and notation for SL.
\begin{defn}
[LLT and exponential tilts] Let $\pi$ be a compactly-supported probability
measure in $\Rn$. For $t\geq0$ and $\theta\in\Rn$, the LLT of $\pi$
is defined as
\[
\Lambda_{t}(\theta)=\log\int\exp\bpar{\inner{\theta,x}-\frac{t}{2}\,\norm x^{2}}\,\pi(\D x)\,.
\]
This defines the probability measure $\mu_{t,\theta}$ as follows:
\[
\D\pi_{t,\theta}(x)=\exp\bpar{\inner{\theta,x}-\frac{t}{2}\,\norm x^{2}-\Lambda_{t}(\theta)}\,\D\pi(x)\,.
\]
Then, $\pi_{\theta}:=\pi_{0,\theta}$ is called a \emph{log-affine
transformation} or \emph{exponential tilt} of $\pi$.
\end{defn}

Differentiating $\Lambda_{t}(\theta)$ w.r.t. $\theta$ leads to the
barycenter and covariance matrix of $\pi_{t,\theta}$:
\begin{align*}
b(t,\theta) & =\nabla\Lambda_{t}(\theta)=\int x\,\pi_{t,\theta}(\D x)\in\Rn\,,\\
\Sigma(t,\theta) & =\hess\Lambda_{t}(\theta)=\int x\otimes x\,\pi_{t,\theta}(\D x)-b(t,\theta)\otimes b(t,\theta)\in\Rnn\,.
\end{align*}

Now consider a stochastic process $(\theta_{t}^{x})_{t\geq0}$ given
by 
\[
\D\theta_{t}^{x}=b(t,\theta_{t}^{x})\,\D t+\D W_{t}\,,\quad\theta_{0}^{x}=x\,,\qquad t\geq0\,.
\]
For $\theta_{t}:=\theta_{t}^{0}$, observe that the SL is simply the
measure-valued process $(\pi_{t})_{t\geq0}$ defined as 
\[
\pi_{t}:=\pi_{t,\theta_{t}}\propto\exp\bpar{\inner{\theta_{t},x}-\frac{t}{2}\,\norm x^{2}}\,\pi(x)\,,
\]
and the barycenter and covariance matrix of $\pi_{t}$ are denoted
as 
\[
b_{t}=b(t,\theta_{t})\,,\qquad\Sigma_{t}=\Sigma(t,\theta_{t})\,.
\]
In summary, randomness in stochastically localized $\pi_{t,\theta_{t}^{x}}$
is captured by $\theta_{t}^{x}$ and driven by the Brownian motion
$W_{t}$. 

\paragraph{Construction of a coupling between $\pi_{x}$ and $\pi_{y}$.}

Our goal is to construct a sufficiently good coupling between $\pi_{x}$
and $\pi_{y}$, so that we could bound $\mc W_{2}(\pi,\pi_{\veps e_{i}})$
by taking $x=0$ and $y=\veps e_{i}$. We again take a backdoor approach
to constructing such a coupling by instead coupling the SL measures
of $\pi_{x}$ and $\pi_{y}$.

Perhaps, the most straightforward way to couple $\pi_{t,\theta_{t}^{x}}$
and $\pi_{t,\theta_{t}^{y}}$ (the SL measures of $\pi_{x}$ and $\pi_{y}$)
is to drive both processes with the same Brownian motion. That is,
for a fixed realization of the Brownian motion, we obtain an explicit
description of how $\theta_{t}^{x}$ and $\theta_{t}^{y}$ evolve
deterministically over time (Lemma~\ref{lem:parallel}). We then
relate how $\theta_{t}^{x}/t$ recovers the exponentially-tilted measure
$\pi_{x}$ in the limit (Lemma~\ref{lem:localization}). This allows
us to propagate the parallel coupling of $\pi_{t,\theta_{t}^{x}}$
and $\pi_{t,\theta_{t}^{y}}$ to the desired coupling of $\pi_{x}$
and $\pi_{y}$.

We begin with the first component by providing an explicit description
of $(\theta_{t}^{x})$ for a realized Brownian motion $w=(w_{t})_{t\geq0}\in C(\R_{\geq0},\Rn)$
(i.e., the collection of continuous paths in $\Rn$).
\begin{lem}
\label{lem:parallel} Fix $w=(w_{t})_{t\geq0}\in C([0,\infty),\Rn)$.
Then for any starting point $x\in\Rn$, there exists a unique solution
$(\theta_{t})_{t\geq0}$ to the integral equation
\begin{equation}
\theta_{t}=x+w_{t}+\int_{0}^{t}b(s,\theta_{s})\,\D s\,.\label{eq:integral-eq}
\end{equation}
It has the following regularity:
\begin{itemize}
\item The solution $\theta_{t}(x)$ is continuous in $(t,x)\in\R_{\geq0}\times\Rn$
and smooth in $x\in\Rn$ for any fixed $t\geq0$.
\item The derivative $M_{t}(x):=\theta_{t}'(x)\in\Rnn$ is continuous in
$(t,x)\in\R_{\geq0}\times\Rn$ and $C^{1}$-smooth in $t>0$. Also,
$M_{t}(x)$ is the unique solution of the linear differential equation
\[
M_{0}(x)=I_{n}\,,\qquad\de_{t}M_{t}(x)=\Sigma\bpar{t,\theta_{t}(x)}\,M_{t}(x)\quad\text{for }t\geq0\,.
\]
\end{itemize}
\end{lem}

To indicate the dependence of $\theta_{t}(x)$ on $w_{t}$, they denote
by $G=(G_{t,w}(x))_{t\geq0}$ the unique solution $\theta_{t}(x)$
to the integral equation. Equipping $C(\R_{\geq0},\Rn)$ with the
topology of uniform convergence on compact intervals and letting $(\mc F_{t})_{t\geq0}$
be the natural filtration of the coordinate process on $C(\R_{\geq0},\Rn)$,
KL show the following result:
\begin{lem}
\label{lem:G-ftn} The map $(x,w)\mapsto(G_{t,w}(x))_{t\geq0}\in C(\R_{\geq0},\Rn)$
is continuous. Moreover, for fixed $t\geq0$ and $x\in\Rn$, the map
$w\mapsto G_{t,w}(x)$ is $\mc F_{t}$-measurable. In particular,
$G_{t,w}:\Rn\to\Rn$ is $e^{C_{\pi}t}$-Lipschitz, where $C_{\pi}$
is a constant depending only on $\pi$.
\end{lem}

KL then inject randomness into the construction by treating $w_{t}$
as a realization of the standard Brownian motion $W=(W_{t})_{t\geq0}$.
Consider the stochastic process $(\bar{\theta}_{t}^{x})_{t\geq0}$
given by 
\[
\bar{\theta}_{t}^{x}=G_{t,W}(x)\qquad\text{for }t\geq0\,.
\]
By construction, one can interpret the integral equation \eqref{eq:integral-eq}
as a SDE:
\[
\D\bar{\theta}_{t}^{x}=b(t,\bar{\theta}_{t}^{x})\,\D t+\D W_{t}\,,\quad\bar{\theta}_{0}^{x}=x\,,\qquad t\geq0\,.
\]
Thus, we now have a pathwise description of $\theta_{t}^{x}=\bar{\theta}_{t}^{x}$!
Moreover, $\theta_{t}^{x}/t$ recovers the starting measure $\pi_{x}$
in the limit:
\begin{lem}
\label{lem:localization} Let $x\in\Rn$ and $X\sim\pi_{x}$ be independent
of the process $W=(W_{t})_{t\geq0}$. Then, the process $(G_{t,W}(x))_{t\geq0}$
has the same law as the process
\[
(x+W_{t}+tX)_{t\geq0}\,.
\]
\end{lem}

Notably, as $W_{t}/t\to0$ almost surely, we can conclude that 
\[
\lim_{t\to\infty}\frac{G_{t,W}(x)}{t}\sim\pi_{x}\,.
\]
The construction thus far implicitly defines a coupling between $\pi_{x}$
and $\pi_{y}$ for any $x,y\in\Rn$. Note that for a realized Brownian
motion $w_{t}$, the infinitesimal updates to $G_{t,W}(x)$ and $G_{t,W}(y)$
would look parallel since they are driven by the same Brownian motion,
and this is why KL calls this parallel coupling.

\subsubsection{Bounding $\protect\mc W_{2}$-distance}

We bound $\mc W_{2}(\pi,\pi_{\veps v})$ using the parallel coupling.
Let $\pi$ be a logconcave distribution in $\Rn$. We can first check
that for any $x,y\in\Rn$, 
\[
\frac{\norm{\theta_{t}^{x}-\theta_{t}^{y}}}{t}\text{ is non-increasing in }t\in\R_{>0}\,.
\]
To see this, recall from the integral equation \eqref{eq:integral-eq}
that $\theta_{t}^{x}=x+w_{t}+\int_{0}^{t}b(s,\theta_{s}^{x})\,\D s$,
so 
\[
\theta_{t}^{x}-\theta_{t}^{y}=x-y+\int_{0}^{t}\bpar{b(s,\theta_{s}^{x})-b(s,\theta_{s}^{y})}\,\D s\,.
\]
Thus,
\[
\frac{\D}{\D t}\,\norm{\theta_{t}^{x}-\theta_{t}^{y}}^{2}=2\inner{\theta_{t}^{x}-\theta_{t}^{y},b(t,\theta_{t}^{x})-b(t,\theta_{t}^{y})}\le\frac{2}{t}\,\norm{\theta_{t}^{x}-\theta_{t}^{y}}^{2}\,,
\]
where the last line follows from $b(t,\theta_{t}^{x})-b(t,\theta_{t}^{y})=\Sigma(t,\theta^{*})\,(\theta_{t}^{x}-\theta_{t}^{y})$
for some $\theta^{*}\in\Rn$, and from $\Sigma(t,\cdot)\preceq t^{-1}I_{n}$
due to the strong log-concavity of $\pi_{t}$. Hence,
\[
\frac{\D}{\D t}\,\frac{\norm{\theta_{t}^{x}-\theta_{t}^{y}}^{2}}{t^{2}}\leq t^{-2}\cdot\frac{2}{t}\,\norm{\theta_{t}^{x}-\theta_{t}^{y}}^{2}-2\,\frac{\norm{\theta_{t}^{x}-\theta_{t}^{y}}^{2}}{t^{3}}\leq0\,,
\]
which shows that $\norm{\theta_{t}^{x}-\theta_{t}^{y}}/t$ is indeed
non-increasing in $t$.

We are ready to bound the $\mc W_{2}$-distance as follows:
\begin{equation}
\mc W_{2}(\pi_{x},\pi_{y})\leq\frac{1}{t}\,(\E[\norm{\theta_{t}^{x}-\theta_{t}^{y}}^{2}])^{1/2}\,.\label{eq:w2-bound}
\end{equation}
To see this, for any $x,y\in\Rn$, take the parallel coupling between
$\lim_{t\to\infty}\frac{\theta_{t}^{x}}{t}\sim\pi^{x}$ and $\lim_{t\to\infty}\frac{\theta_{t}^{y}}{t}\sim\pi^{y}$.
Then, 
\[
\mc W_{2}^{2}(\pi_{x},\pi_{y})\leq\E\bbrack{\bnorm{\lim_{t\to\infty}\frac{\theta_{t}^{x}}{t}-\lim_{t\to\infty}\frac{\theta_{t}^{y}}{t}}^{2}}=\E\bbrack{\lim_{t\to\infty}\bnorm{\frac{\theta_{t}^{x}-\theta_{t}^{y}}{t}}^{2}}\,.
\]
Since $\norm{\theta_{t}^{x}-\theta_{t}^{y}}/t$ is non-increasing
in $t$, the claim follows.

Defining $M_{t}:=G_{t,W}'(0)\in\Rnn$, we next show that
\[
\limsup_{\veps\to0^{+}}\frac{\mc W_{2}(\pi,\pi_{\veps v})}{\veps}\leq\frac{(\E[\norm{M_{t}v}^{2}])^{1/2}}{t}\qquad\text{for any }v\in\Rn\,.
\]
To see this, recall from Lemma~\ref{lem:G-ftn} that $G_{t,w}(\cdot)$
is $e^{C_{\pi}t}$-Lipschitz. Hence,
\[
\frac{\norm{\theta_{t}^{0}-\theta_{t}^{\veps v}}}{\veps}\leq e^{C_{\pi}t}\,\norm v\,.
\]
Therefore, by the dominated convergence theorem,
\[
\lim_{\veps\to0^{+}}\E\bbrack{\frac{\norm{\theta_{t}^{0}-\theta_{t}^{\veps v}}^{2}}{\veps^{2}}}=\E\bbrack{\lim_{\veps\to0^{+}}\frac{\norm{\theta_{t}^{0}-\theta_{t}^{\veps v}}^{2}}{\veps^{2}}}=\E\bbrack{\lim_{\veps\to0^{+}}\frac{\norm{G_{t,w}(0)-G_{t,w}(\veps v)}^{2}}{\veps^{2}}}=\E[\norm{M_{t}v}^{2}]\,.
\]
Using the $\mc W_{2}$-bound in \eqref{eq:w2-bound} with $x=0$ and
$y=\veps v$, 
\[
\limsup_{\veps\to0^{+}}\frac{\mc W_{2}(\pi,\pi_{\veps v})}{\veps}\leq\frac{1}{t}\,\bpar{\limsup_{\veps\to0^{+}}\frac{\E[\norm{\theta_{t}^{0}-\theta_{t}^{\veps v}}^{2}]}{\veps^{2}}}^{1/2}=\frac{(\E[\norm{M_{t}v}^{2}])^{1/2}}{t}\,.
\]

Using this result with $v=e_{1},\dots,e_{n}$ and Lemma~\ref{lem:Hinverse2W2},
if $\pi$ is a centered, compactly-supported probability measure on
$\Rn$, we can conclude that
\begin{equation}
\sum_{i=1}^{n}\norm{x_{i}}_{H^{-1}(\pi)}^{2}\leq t^{-2}\sum_{i=1}^{n}\E[\norm{M_{t}e_{i}}^{2}]=t^{-2}\E[\norm{M_{t}}_{\hs}^{2}]\,.\label{eq:final-ineq}
\end{equation}

\subsubsection{Relating $\protect\norm{M_{t}}_{\protect\hs}^{2}$ to $\{\protect\lda_{i}(t)\}_{i\in[n]}$}

Let us bound $\E[\norm{M_{t}}_{\hs}^{2}]$. Recall that $M_{t}=G_{t,w}'(x)$
satisfies that for $\Sigma_{t}=\Sigma(t,G_{t,w}(x))$,
\[
M_{0}=I_{n}\,,\qquad\de_{t}M_{t}=\Sigma_{t}M_{t}\quad\text{for }t\geq0\,.
\]
When $n=1$, we have a closed-form formula as $M_{t}=\exp(\int_{0}^{t}\Sigma_{s}\,\D s)$,
so 
\[
\E[\norm{M_{t}}_{\hs}^{2}]=\exp\Bpar{2\int_{0}^{t}\lda_{1}(s)\,\D s}\,.
\]
However, this does not hold in higher dimension in general, so Klartag
and Lehec instead prove the following statement.
\begin{lem}
\label{lem:trace-bound} Let $\lda_{1}(t)\geq\lda_{2}(t)\geq\dots\geq\lda_{n}(t)>0$.
Then for any $t>0$,
\[
\E[\norm{M_{t}}_{\hs}^{2}]\leq\sum_{i=1}^{n}\exp\Bpar{2\int_{0}^{t}\lda_{i}(s)\,\D s}\,.
\]
\end{lem}

To this end, we need one algebraic lemma.
\begin{lem}
Let $\mu_{1}(t),\dots,\mu_{n}(t)$ and $\lambda_{1}(t)\ge\cdots\ge\lambda_{n}(t)$
be nonnegative, continuous functions on $[0,\infty)$. Assume that
for all $t\ge0$ and $k=1,\dots,n$, 
\begin{equation}
\sum_{i=1}^{k}\mu_{i}(t)\le\sum_{i=1}^{k}\Bpar{1+2\int_{0}^{t}\mu_{i}(s)\lambda_{i}(s)\,\D s}.\label{eq:assump45}
\end{equation}
Then for all $t\ge0$ and $k=1,\dots,n$, 
\begin{equation}
\sum_{i=1}^{k}\mu_{i}(t)\le\sum_{i=1}^{k}\exp\Bpar{2\int_{0}^{t}\lambda_{i}(s)\,\D s}\,.\label{eq:goal46}
\end{equation}
\end{lem}

\begin{proof}
Define, for $i=1,\dots,n$,
\begin{equation}
\nu_{i}(t):=1+2\int_{0}^{t}\mu_{i}(s)\lambda_{i}(s)\,\D s,\qquad E_{i}(t):=\exp\Bpar{2\int_{0}^{t}\lambda_{i}(s)\,\D s}\,.\label{eq:nu_def}
\end{equation}
By \eqref{eq:assump45}, for all $k$,
\begin{equation}
\sum_{i=1}^{k}\mu_{i}(t)\le\sum_{i=1}^{k}\nu_{i}(t)\,.\label{eq:mu_le_nu_partial}
\end{equation}

We prove \eqref{eq:goal46} by induction on $k$. When $k=1$, by
\eqref{eq:nu_def} and \eqref{eq:mu_le_nu_partial} with $k=1$, 
\[
\nu_{1}'(t)=2\mu_{1}(t)\lambda_{1}(t)\le2\nu_{1}(t)\lambda_{1}(t)\,.
\]
Hence $\de_{t}(\nu_{1}(t)e^{-2\int_{0}^{t}\lambda_{1}})\le0$, so
$\nu_{1}(t)\le E_{1}(t)$. Using $\mu_{1}(t)\le\nu_{1}(t)$ gives
$\mu_{1}(t)\le E_{1}(t)$.

We now fix $k\ge2$ and assume \eqref{eq:goal46} holds for all $1,\dots,k-1$.
We first recall a summation-by-parts claim: \emph{Let $w_{1}\ge\cdots\ge w_{k-1}\ge0$
and $a_{i},b_{i}\ge0$ satisfy $\sum_{i=1}^{m}a_{i}\le\sum_{i=1}^{m}b_{i}$
for all $m\le k-1$. Then $\sum_{i=1}^{k-1}a_{i}w_{i}\le\sum_{i=1}^{k-1}b_{i}w_{i}$}.
To use this, set $w_{i}(t):=\lambda_{i}(t)-\lambda_{k}(t)$ for $i\in[k-1]$.
Then $w_{1}\ge\cdots\ge w_{k-1}\ge0$. Applying the above claim with
$a_{i}=\mu_{i}(t)$ and $b_{i}=E_{i}(t)$ (using the induction hypothesis
for all partial sums $m\le k-1$) yields 
\begin{equation}
\sum_{i=1}^{k-1}\mu_{i}(t)\,\bpar{\lambda_{i}(t)-\lambda_{k}(t)}\le\sum_{i=1}^{k-1}E_{i}(t)\,\bpar{\lambda_{i}(t)-\lambda_{k}(t)}\,.\label{eq:weighted_bound}
\end{equation}

Let $F_{k}(t):=\sum_{i=1}^{k}\nu_{i}(t)$. Then, using $\sum_{i=1}^{k}\mu_{i}(t)\le F_{k}(t)$
from \eqref{eq:mu_le_nu_partial},
\begin{align*}
F_{k}'(t) & =2\sum_{i=1}^{k}\mu_{i}(t)\lambda_{i}(t)=2\sum_{i=1}^{k-1}\mu_{i}(t)\,\bpar{\lambda_{i}(t)-\lambda_{k}(t)}+2\lambda_{k}(t)\sum_{i=1}^{k}\mu_{i}(t)\\
 & \leq2\sum_{i=1}^{k-1}E_{i}(t)\,\bpar{\lambda_{i}(t)-\lambda_{k}(t)}+2\lambda_{k}(t)F_{k}(t)\,.
\end{align*}

Define $I_{k}(t)=\exp(-2\int_{0}^{t}\lambda_{k}(s)\,\D s)$. Multiplying
$I_{k}(t)$ to both sides of the above gives
\[
\de_{t}\bpar{I_{k}(t)F_{k}(t)}\le2\sum_{i=1}^{k-1}I_{k}(t)E_{i}(t)\,\bpar{\lambda_{i}(t)-\lambda_{k}(t)}=\sum_{i=1}^{k-1}\de_{t}\exp\Bpar{2\int_{0}^{t}\bpar{\lambda_{i}(s)-\lambda_{k}(s)}\,\D s}\,,
\]
where the last equality follows from $I_{k}(t)E_{i}(t)=\exp(2\int_{0}^{t}(\lambda_{i}(s)-\lambda_{k}(s))\,\D s)$.
Integrating from $0$ to $t$ and using $\nu_{i}(0)=1$ (hence $F_{k}(0)=k$)
yields 
\[
I_{k}(t)F_{k}(t)\le k+\sum_{i=1}^{k-1}(e^{2\int_{0}^{t}(\lambda_{i}-\lambda_{k})}-1)=1+\sum_{i=1}^{k-1}e^{2\int_{0}^{t}(\lambda_{i}-\lambda_{k})}\,.
\]
Multiplying by $I_{k}(t)^{-1}$, we conclude
\[
F_{k}(t)\le e^{2\int_{0}^{t}\lambda_{k}}+\sum_{i=1}^{k-1}e^{2\int_{0}^{t}\lambda_{i}}=\sum_{i=1}^{k}E_{i}(t)\,.
\]
Finally, \eqref{eq:mu_le_nu_partial} gives $\sum_{i=1}^{k}\mu_{i}(t)\le F_{k}(t)\le\sum_{i=1}^{k}E_{i}(t)$,
which is \eqref{eq:goal46} for this $k$.
\end{proof}
We can now prove Lemma~\ref{lem:trace-bound}.
\begin{proof}
[Proof of Lemma~\ref{lem:trace-bound}] Let $S_{t}:=M_{t}^{\T}M_{t}\succeq0$
and its eigenvalues be $\mu_{1}(t)\ge\cdots\ge\mu_{n}(t)\ge0$. Since
$\de_{t}M_{t}=\Sigma_{t}M_{t}$, we have $S'_{t}=2M_{t}^{\T}\Sigma_{t}M_{t}$.
For $k\in[n]$, define the Ky Fan $k$-sum, 
\[
f_{k}(t):=\sum_{i=1}^{k}\mu_{i}(t)\,.
\]
By the variational characterization of Ky Fan $k$-sums,
\[
f_{k}(t)=\max\{\tr(PS_{t}):P^{2}=P=P^{\T},\,\rank(P)=k\}\,.
\]
Recall Danskin's theorem that $f(x)=\max_{y}g(x,y)$ with some regularity
implies that $f'(x)=\max_{y\in\arg\max_{y}g(x,y)}\de_{x}g(x,y)$.
Using this and taking a maximizing projector $P_{t}$ (which exists
a.e. $t$),
\[
f_{k}'(t)=\de_{t}\tr(P_{t}S_{t})=\tr(P_{t}S'_{t})=2\tr(P_{t}M_{t}^{\T}A_{t}M_{t})=2\tr(\Sigma_{t}M_{t}P_{t}M_{t}^{\T})\,.
\]
Let $B_{t}:=M_{t}P_{t}M_{t}^{\T}\succeq0$ (so $\rank(B_{t})\le k$),
and since $P_{t}$ projects onto the top-$k$ eigenspace of $S_{t}$,
the nonzero eigenvalues of $B_{t}$ are precisely $\mu_{1}(t),\dots,\mu_{k}(t)$.
By von Neumann's trace inequality for PSD matrices, 
\[
\tr(\Sigma_{t}B_{t})\le\sum_{i=1}^{k}\lambda_{i}(t)\mu_{i}(t)\,.
\]
Therefore, for a.e. $t$,
\[
f_{k}'(t)\le2\sum_{i=1}^{k}\lambda_{i}(t)\mu_{i}(t)\,.
\]
Integrating and using $M_{0}=I$ (so $S_{0}=I$ and $f_{k}(0)=k$),
\[
f_{k}(t)\le k+2\int_{0}^{t}\sum_{i=1}^{k}\lambda_{i}(s)\mu_{i}(s)\,\D s\,.
\]
Using the algebraic lemma above,
\[
f_{k}(t)\le\sum_{i=1}^{k}\exp\Bpar{2\int_{0}^{t}\lambda_{i}(s)\,\D s}\,.
\]
Taking $k=n$ concludes the proof.
\end{proof}
Combining this lemma with \eqref{eq:final-ineq}, we obtain the first
desired result:
\[
\sum_{i=1}^{n}\norm{x_{i}}_{H^{-1}(\pi)}^{2}\leq\frac{1}{t^{2}}\,\E\Bbrack{\sum_{i=1}^{n}\exp\Bpar{2\int_{0}^{t}\lda_{i}(s)\,\D s}}\,.
\]

\subsection{A stopping-time version of Guan's result}

In this section, we sketch a proof of the second main technical result:
for any stopping time $\tau$ and $t>0$, 
\[
\sum_{i=1}^{n}\P\bpar{\lda_{i}(t\wedge\tau)\geq3}\lesssim n\exp(-t^{-1/8})\,.
\]
Note that this result recovers Guan's bound when $\tau=\infty$, so
this is essentially a stopping-time version of Guan's result.

KL first prove a stopping-time version of the It\^o derivative of
eigenvalues in \cite{Klartag21yuansi}: for any stopping time $\tau$
and any $C^{2}$-smooth function $f:[0,\infty)\to\R$, 
\[
\de_{t}\E\tr f(\Sigma_{t\wedge\tau})=-\sum_{i=1}^{n}\E\bbrack{\lda_{i}^{2}f'(\lda_{i})\cdot\ind[t<\tau]}+\half\sum_{i,j=1}^{n}\E\bbrack{\norm{H_{ij}}^{2}\,\frac{f'(\lda_{i})-f'(\lda_{j})}{\lda_{i}-\lda_{j}}\cdot\ind[t<\tau]}\,.
\]
As seen earlier, Guan showed that for some increasing function $f:[0,\infty)\to[0,\infty)$
such that $f(x)=x^{2}$ for $x\geq r$ and $f''(x)\leq D^{2}f(x)$
for $x\geq0$,
\[
\sum_{i,j}\norm{H_{ij}}^{2}\frac{f'(\lda_{i})-f'(\lda_{j})}{\lda_{i}-\lda_{j}}\lesssim\bpar{\frac{1}{t}+\frac{D^{2}}{t^{1/2}}}\sum_{i}f(\lda_{i})\,.
\]
Multiplying $\ind[t<\tau]$ to both sides above and taking expectation
leads to
\begin{align*}
\de_{t}\E\tr f(\Sigma_{t\wedge\tau}) & \lesssim\bpar{\frac{1}{t}+\frac{D^{2}}{t^{1/2}}}\sum_{i}\E\bbrack{f\bpar{\lda_{i}(t)}\,\ind[t<\tau]}\leq\bpar{\frac{1}{t}+\frac{D^{2}}{t^{1/2}}}\sum_{i}\E\bbrack{f\bpar{\lda_{i}(t\wedge\tau)}}\\
 & =\bpar{\frac{1}{t}+\frac{D^{2}}{t^{1/2}}}\,\E\tr f(\Sigma_{t\wedge\tau})\,.
\end{align*}

Using a similar function as in Guan's paper, KL consider a $C^{2}$-smooth
function $f_{D,r}:\R_{\geq0}\to\R_{\geq0}$ with two parameters $D>1$
and $r\in[2,3]$ such that $f''(x)\leq(12D)^{2}f(x)$ and 
\[
f(x)=f_{D,r}(x)=\begin{cases}
e^{D(x-r)} & x\leq r-D^{-1}\,,\\
x^{2} & x\geq r\,.
\end{cases}
\]
Then the proof can be summarized as follows:
\begin{itemize}
\item As in Guan's paper, construct sequences $r_{k},D_{k},t_{k}(=t/2^{8k})$
properly such that solving the differential inequality above leads
to 
\[
\E\tr f_{k}(\Sigma_{t_{k}\wedge\tau})\leq\bpar{\frac{t_{k}}{t_{k+1}}}^{O(1)}\E\tr f_{k}(\Sigma_{t_{k+1}\wedge\tau})\,.
\]
\item Define $g_{k}(x)=x^{2}\,\ind[x\geq r_{k}]$ (so $g_{k}\leq f_{k}$)
and show $f_{k}\leq\frac{9}{4}\,g_{k+1}+\exp(-t_{k}^{-1/8})$.
\item For $F_{k}:=\E\tr g_{k}(\Sigma_{t_{k}\wedge\tau})$, using the two
items above, we can deduce that 
\[
F_{k}\leq2^{O(1)}\,\bpar{F_{k+1}+n\exp(-t_{k}^{-1/8})}\,.
\]
By recursion, for any $k\geq1$ and constant $C>0$,
\[
F_{0}\leq2^{Ck}F_{k}+n\sum_{i=0}^{k-1}2^{C\,(i+1)}\exp(-t_{i}^{-1/8})\,.
\]
Using easy inequalities (following from the construction of $t_{k}$),
we obtain that 
\[
F_{0}\leq t_{k}^{-C}F_{k}+O(n)\,\exp(-t^{-1/8})\,.
\]
\item We now show $F_{k}/t_{k}^{C}\to0$ as $k\to\infty$. Since $r_{k}\geq2$,
\begin{align*}
F_{k} & =\E\tr g_{k}(\Sigma_{t_{k}\wedge\tau})\leq\E\Bbrack{\sum_{i}\lda_{i}^{2}(t_{k}\wedge\tau)\,\ind[\lda_{i}(t_{k}\wedge\tau)\geq2]}\leq\E\bbrack{\norm{\Sigma_{t_{k}\wedge\tau}}_{\hs}^{2}\,\ind[\norm{\Sigma_{t_{k}\wedge\tau}}\geq2]}\\
 & \lesssim_{n,C_{\pi}}\P(\norm{\Sigma_{t_{k}}}\geq2)\,,
\end{align*}
where the last inequality follows from the compact-support of $\pi_{t}$.
Since $t_{k}\to0$ as $k\to\infty$, and $\P(\norm{\Sigma_{t}}\geq2)\leq\exp(-O(t^{-1}))$
for $t\lesssim\log^{-2}n$, the bound on $F_{k}$ is dominated by
$t_{k}^{-C}$ as $k\to\infty$. Thus,
\[
F_{0}\lesssim n\exp(-t^{-1/8})\,.
\]
\item Lastly, due to $\ind[x\geq3]\leq x^{2}\,\ind[x\geq3]=g_{0}(x)$ (here
$r_{0}=3$), the second main claim follows from 
\[
\sum_{i=1}^{n}\P\bpar{\lda_{i}(t\wedge\tau)\geq3}\leq\E\tr g_{0}(\Sigma_{t\wedge\tau})=F_{0}\lesssim n\exp(-t^{-1/8})\,.
\]
\selectlanguage{american}%
\end{itemize}

\selectlanguage{english}%

\section{Discussion}

There are several applications of localization in recent years that
are not covered in this survey. We have already mentioned localization
schemes for analyzing Markov chains \cite{chen2022localization}.
There are also algorithmic applications \cite{AMS22sampling} and
information-theoretic perspectives \cite{AM22information}.

The KLS conjecture, a principal motivation for localization methods,
remains unresolved. In fact, the lower bound in Theorem \ref{thm:cov-bound}
suggests that a new idea is needed to affirmatively resolve the conjecture.
Besides its mathematical appeal, settling the full conjecture is also
motivated by a few different applications such as hypercontractivity
for clustering mixtures of logconcave distributions and bounding the
spectral norm of a matrix with a joint logconcave distribution.

We mention the following question, which could be viewed as an analog
of the Kadison--Singer problem for convex bodies: given an isotropic
convex body, does there exist a hyperplane partition into two convex
bodies such that both parts have covariance matrices with spectral
norm strictly less than $1$?
\begin{acknowledgement*}
The authors are indebted to Ravi Kannan, Laci Lov\'asz, Yin Tat Lee,
Assaf Naor, Daniel Dadush, Ben Cousins, Boaz Klartag, Ronen Eldan,
Mark Rudelson, Emanuel Milman, Matthieu Fradelizi, Andre Wibisono,
Sinho Chewi, Aditi Laddha, He Jia and Yuansi Chen. This work was supported
in part by NSF award CCF-2106444 and a Simons Investigator award.
\end{acknowledgement*}
\bibliographystyle{alpha}
\bibliography{main}

\appendix
\selectlanguage{american}%

\section{Semigroups and Spectral Analysis\label{sec:semigroup}}

\inputencoding{latin9}In this appendix we review relevant semigroup
approaches with spectral analysis. We refer interested readers to
\cite{van14probability,BGL14analysis,chewi25log} for semigroup theory.

\subsection{Semigroups and infinitesimal generators}

For a time-homogeneous Markov process $(X_{t})_{t\geq0}$, one can
study its properties through its \emph{semigroup} and \emph{generator},
which serve as handy tools for studying the Markov process.
\begin{defn}
[Semigroup and generator] For a time-homogeneous Markov process $(X_{t})_{t\geq0}$,
its associated \emph{semigroup} $(P_{t})_{t\geq0}$ is the family
of functional operators defined as 
\[
P_{t}f(x)\defeq\E[f(X_{t})|X_{0}=x]\,.
\]
Then, the \emph{(infinitesimal) generator} $\mc L$ associated with
the semigroup $(P_{t})_{t\geq0}$ is a functional operator defined
by 
\[
\mc Lf\defeq\lim_{t\to0}\frac{P_{t}f-f}{t}
\]
for all functions $f$ for which the above limit exists.
\end{defn}

Why do we bother considering these functional operators? The semigroup
and generator carry essential information on the original Markov process,
so the process can be studied through these notions in handy ways.
For example, we will see a very natural connection between (\ref{eq:PI})
and generators.

In this note, all we need are just two Markov semigroups --- the
\emph{heat} semigroup and the \emph{Langevin} semigroup.

\subsubsection{Heat semigroup}

The simplest Markov process would be the Brownian motion (or the Wiener
process) defined by 
\[
\D X_{t}=\D W_{t}\,.
\]
From the direct computation, one can check that
\[
P_{t}f=f*\gamma_{t}\,,
\]
where $\gamma_{t}$ denotes the Gaussian $\mc{\msf N}(0,tI_{n})$.
Then, its associated generator is simply
\[
\mc L=\half\,\Delta\,.
\]
Computational details will follow shortly.

\subsubsection{Langevin semigroups}

Another fundamental Markov process is the Langevin diffusion: for
a potential function $V:\Rn\to\R$,
\[
\D X_{t}=-\nabla V(X_{t})\,\D t+\sqrt{2}\,\D W_{t}\,.
\]
One can show that its associated generator $\mc L:L^{2}(\pi)\to L^{2}(\pi)$
(precisely, the domain is the subset of $L^{2}(\pi)$ where $\mc L$
is well-defined) is defined by 
\[
\mc L(\cdot)=\Delta(\cdot)-\inner{\nabla V,\nabla\cdot}\,.
\]
To see this, we use It\^o's formula to obtain that for a smooth test
function $f$,
\begin{align*}
\D f(X_{t}) & =\inner{\nabla f(X_{t}),\,\D X_{t}}+\half\inner{\nabla^{2}f(X_{t}),2I_{n}}\,\D t\\
 & =\inner{\nabla f(X_{t}),-\nabla V(X_{t})\,\D t+\sqrt{2}\,\D W_{t}}+\tr\nabla^{2}f(X_{t})\,\D t\\
 & =\bigl(\Delta f(X_{t})-\inner{\nabla V(X_{t}),\nabla f(X_{t})}\bigr)\,\D t+\sqrt{2}\,\inner{\nabla f(X_{t}),\D W_{t}}\,.
\end{align*}
Hence,
\[
f(X_{t})-f(X_{0})=\int_{0}^{t}\bpar{\Delta f(X_{s})-\inner{\nabla V(X_{s}),\nabla f(X_{s})}}\,\D s+\sqrt{2}\int_{0}^{t}\inner{\nabla f(X_{s}),\D W_{s}}\,.
\]
For $X_{0}=x$, as the It\^o integral term is a martingale (so its
expectation vanishes),
\[
\frac{\E[f(X_{t})|X_{0}=x]-f(x)}{t}=\frac{1}{t}\int_{0}^{t}\E\bigl[\bpar{\Delta f(X_{s})-\inner{\nabla V(X_{s}),\nabla f(X_{s})}}\,|\,X_{0}=x\bigr]\,\D s\,,
\]
and thus
\[
\mc Lf(x)=\lim_{t\to0}\frac{\E[f(X_{t})|X_{0}=x]-f(x)}{t}=\Delta f(x)-\inner{\nabla V(x),\nabla f(x)}\,.
\]
The generator associated with the heat semigroup can be computed in
a similar way with $V=0$ (and different scaling for $W_{t}$).

\subsection{Semigroup preliminaries}

We briefly go over the notions of stationary measure, reversibility,
and carr\'e du champ operator, which are essential for spectral analyses
of $\mc L$.

\subsubsection{Stationary distributions}

As a given Markov process $(X_{t})_{t\geq0}$ evolves, it is natural
to study how $\pi_{t}:=\law X_{t}$ evolves over time $t$. A \emph{stationary
distribution} $\pi$ of the Markov process is a probability measure
such that if $\pi=\law X_{0}$, then for all $t\ge0$,
\[
\pi_{t}=\pi\,.
\]
It is a standard fact that $\D\pi\propto\exp(-V)\,\D x$ is the stationary
distribution of the Langevin diffusion (under mild conditions on $V$).

\subsubsection{Reversibility and integration by parts}

Now consider a Markov semigroup $(P_{t})_{t\geq0}$ with its generator
$\mc L$ and stationary distribution $\pi$. We now study properties
of $P_{t}$ and $\mc L$ on the Hilbert space $L^{2}(\pi)$, which
has a natural inner product structure:
\[
\inner{f,g}_{L^{2}(\pi)}:=\int fg\,\D\pi\,.
\]
This inner-product suggests a ``symmetric'' structure of $P_{t}$
and $\mc L$, called \emph{reversibility}.
\begin{defn}
[Reversibility] The Markov semigroup is reversible with respect to
the stationary distribution $\pi$ if for all $f,g\in L^{2}(\pi)$
and $t\geq0$,
\[
\inner{P_{t}f,g}_{L^{2}(\pi)}=\inner{f,P_{t}g}_{L^{2}(\pi)}\,.
\]
Equivalently, for all $f$ and $g$ (for which $\mc Lf$ and $\mc Lg$
exist),
\[
\inner{\mc Lf,g}_{L^{2}(\pi)}=\inner{f,\mc Lg}_{L^{2}(\pi)}\,.
\]
\end{defn}

\begin{rem}
Note that $P_{t}f\in L^{2}(\pi)$ as $P_{t}$ is a \emph{contraction}
operator: for $f\in L^{2}(\pi)$,
\begin{align*}
\norm{P_{t}f}_{L^{2}(\pi)}^{2} & =\int(P_{t}f)^{2}\,\D\pi\underset{\text{Jensen}}{\leq}\int P_{t}(f^{2})\,\D\pi=\E_{X_{0}\sim\pi}\bigl[\E[f^{2}(X_{t})|X_{0}]\bigr]\\
 & =\E_{X_{t}}[f^{2}(X_{t})]=\E_{\pi}[f^{2}]=\norm f_{L^{2}(\pi)}^{2}\,,
\end{align*}
where we used the stationarity of $\pi$ in the penultimate equality.
\end{rem}

Recall that one can define the $L^{2}(\pi)$\emph{-adjoint} $\mc L^{*}$
of $\mc L$ through
\[
\inner{\mc Lf,g}_{L^{2}(\pi)}=\inner{f,\mc L^{*}g}_{L^{2}(\pi)}\quad\text{for all }f,g\in(\text{subset of})\ L^{2}(\pi)\,.
\]
Essentially, this is a concept that extends the idea of the transpose
of a matrix to functional operators. The $L^{2}(\pi)$\emph{-}adjoint
$P_{t}^{*}$ of semigroup $P_{t}$ can be defined in a similar way.
In terms of the adjoint, the reversibility simply means that the generator
$\mc L$ is \emph{self-adjoint} in $L^{2}(\pi)$ (i.e., $\mc L^{*}=\mc L$). 

We need one more notion to establish an important property of reversible
semigroups and generators.
\begin{defn}
[Carr\'e du champ] The \emph{carr\'e du champ} operator $\Gamma$
is defined as
\[
\Gamma(f,g)\defeq\half\,\bpar{\mc L(fg)-f\,\mc Lg-g\,\mc Lf}\,,
\]
and the Dirichlet energy $\mc E(f,g)$ is a functional defined as
\[
\mc E(f,g)=\int\Gamma(f,g)\,\D\pi\,.
\]
\end{defn}

\begin{example}
[Reversibility of Langevin semigroup] We now check the reversibility
of the Langevin semigroup, computing its carr\'e du champ operator.
For any functions $f$ and $g$, noting that
\begin{align*}
\nabla(fg) & =g\nabla f+f\nabla g\,,\\
\Delta(fg) & =\nabla\cdot\nabla(fg)=2\,\inner{\nabla f,\grad g}+g\Delta f+f\Delta g\,,
\end{align*}
we can deduce that
\begin{align*}
2\Gamma(f,g) & =\mc L(fg)-f\,\mc Lg-g\,\mc Lf\\
 & =\Delta(fg)-\inner{\nabla V,\nabla(fg)}-f\,(\Delta g-\inner{\nabla V,\nabla g})-g\,(\Delta f-\inner{\nabla V,\nabla f})\\
 & =2\,\inner{\nabla f,\grad g}+g\Delta f+f\Delta g-\inner{\nabla V,g\nabla f+f\nabla g}\\
 & \qquad-f\,(\Delta g-\inner{\nabla V,\nabla g})-g\,(\Delta f-\inner{\nabla V,\nabla f})\\
 & =2\,\inner{\nabla f,\grad g}\,.
\end{align*}
Thus,
\[
\Gamma(f,g)=\inner{\nabla f,\nabla g}\,,\qquad\Gamma(f,f)=\norm{\nabla f}^{2}\,,\qquad\mc E(f,f)=\E_{\pi}[\norm{\nabla f}^{2}]\,.
\]
The reversibility follows from (\ref{eq:IBP}) which is introduced
right below.
\end{example}

We introduce our first helper lemma, which is an equivalent formulation
of reversibility:
\begin{lem}
[Integration by parts] Suppose that a Markov process is reversible
w.r.t. $\pi$, and let $\mc L$ and $\Gamma$ be the associated generator
and carr\'e du champ operator. Then, for any $f$ and $g$,
\begin{equation}
\inner{f,(-\mc L)g}_{L^{2}(\pi)}=\inner{(-\mc L)f,g}_{L^{2}(\pi)}=\mc E(f,g)=\int\Gamma(f,g)\,\D\pi\,.\tag{\ensuremath{\msf{IBP}}}\label{eq:IBP}
\end{equation}
\end{lem}

\begin{proof}
Observe that
\[
\int\Gamma(f,g)\,\D\pi=\half\int\bpar{\mc L(fg)-f\mc Lg-g\mc Lf}\,\D\pi\underset{(i)}{=}-\half\int(f\mc Lg+g\mc Lf)\,\D\pi\underset{(ii)}{=}\inner{f,(-\mc L)g}_{\ltwo(\pi)}\,,
\]
where $(i)$ follows from $\E_{\pi}[\mc Lh]=0$ for any $h$, due
to
\[
\E_{\pi}[\mc Lh]=\inner{\mc Lh,1}_{\ltwo(\pi)}=\inner{h,\mc L1}_{\ltwo(\pi)}=0\,,
\]
and $(ii)$ follows from reversibility.
\end{proof}
We now arrive at a key observation about a reversible Markov process:
\begin{prop}
[$\mathcal{L}$ is negative] For a reversible Markov process, $-\mc L\geq0$,
i.e., for any $f$, 
\[
\inner{f,(-\mc L)f}_{\ltwo(\pi)}\geq0\,.
\]
\end{prop}

\begin{proof}
By (\ref{eq:IBP}), 
\[
\inner{f,\mc Lf}_{\ltwo(\pi)}=-\int\Gamma(f,f)\,\D\pi\,,
\]
so it suffices to show that for any $f$,

\[
0\leq2\Gamma(f,f)=\mc L(f^{2})-2f\mc Lf\,.
\]
To this end, note that
\[
P_{t}(f^{2})=\E[f^{2}(X_{t})|X_{0}]\geq\bpar{\E[f(X_{t})|X_{0}]}^{2}=(P_{t}f)^{2}\,,
\]
and thus
\[
\mc L(f^{2})=\lim_{t\to0}\frac{P_{t}(f^{2})-f^{2}}{t}\geq\lim_{t\to0}\frac{(P_{t}f)^{2}-f^{2}}{t}=2f\mc Lf\,,
\]
which completes the proof.
\end{proof}

\subsubsection{Spectral theory and Rayleigh quotient}

We have just observed that $-\mc L$ is, as an operator, positive
semi-definite (PSD). Just as a PSD matrix has non-negative real eigenvalues,
this spectral theorem can be generalized to this functional-operator
setup (under some conditions). Thus, we may consider the eigenvalues
of $-\mc L$ as follows:
\[
0=\lda_{0}(-\mc L)<\lda_{1}(-\mc L)\leq\cdots
\]
with corresponding eigenfunctions $\vphi_{0},\vphi_{1},\dotsc$, where
$\vphi_{0}=1$. We also know that the Rayleigh quotient can be related
to eigenvalues through 
\[
\lda_{k}(-\mc L)=\inf_{f\perp\{\vphi_{i}\}_{0\leq i\leq k-1}}\frac{\inner{f,(-\mc L)f}_{\ltwo(\pi)}}{\norm f_{2}^{2}}\,.
\]
Perhaps, the most important eigenvalue is the first non-zero eigenvalue
$\lda_{1}$ called the \emph{spectral gap}. Since $f\perp\vphi_{0}(=1)$
is equivalent to $\E_{\pi}f=0$,
\[
\lda_{1}(-\mc L)=\inf_{f\neq0,\,\E_{\pi}f=0}\frac{\inner{f,(-\mc L)f}_{\ltwo(\pi)}}{\norm f_{2}^{2}}\,.
\]

\begin{example}
[Spectral gap of Langevin] As $\inner{f,(-\mc L)f}_{\ltwo(\pi)}=\mc E(f,f)=\E_{\pi}[\norm{\grad f}^{2}]$
for the Langevin diffusion, the spectral gap can be written as
\[
\lda_{1}(-\mc L)=\inf_{f\neq0,\,\E_{\pi}f=0}\frac{\inner{f,(-\mc L)f}_{\ltwo(\pi)}}{\norm f_{2}^{2}}=\inf_{f\neq0,\,\E_{\pi}f=0}\frac{\E_{\pi}[\norm{\nabla f}^{2}]}{\var_{\pi}f}=\frac{1}{\cpi(\pi)}\,,
\]
where the last equality follows from (\ref{eq:PI}). This expression
shows how spectral analyses of $-\mc L$ become relevant in the study
of the KLS conjecture ($\cpi(\pi)\asymp\norm{\cov\pi}$).
\end{example}

\subsubsection{Bochner formula}

In addition to the first helper (\ref{eq:IBP}), we need another one
called the \emph{Bochner formula}.
\begin{lem}
[Bochner formula] For any function $f$,
\begin{equation}
\half\,\Delta(\norm{\nabla f}^{2})=\inner{\Delta\nabla f,\nabla f}+\norm{\nabla^{2}f}_{\hs}^{2}\,.\tag{\ensuremath{\msf{Bochner}}}\label{eq:Bochner}
\end{equation}
\end{lem}

To better understand this, we define the iterated carr\'e du champ:
\begin{defn}
[Iterated carr\'e du champ] The \emph{iterated carr\'e du champ}
operator $\Gamma_{2}$ is defined by 
\[
\Gamma_{2}(f,g)\defeq\half\,\bpar{\mc L\Gamma(f,g)-\Gamma(f,\mc Lg)-\Gamma(g,\mc Lf)}\,,
\]
which essentially replaces the multiplication in the carr\'e du champ
operator with $\Gamma$.
\end{defn}

\begin{example}
[$\Gamma_2$ for Langevin] Using $\Gamma(f,g)=\inner{\nabla f,\nabla g}$,
we can show that
\begin{align}
\Gamma_{2}(f,f) & =\half\,\bpar{\mc L\Gamma(f,f)-2\Gamma(f,\mc Lf)}=\half\,\bpar{\mc L(\norm{\nabla f}^{2})-2\inner{\nabla f,\nabla\mc Lf}}\label{eq:gamma2-generator}\\
 & =\half\,\bigl\{\Delta(\norm{\nabla f}^{2})-\inner{\nabla V,\nabla(\norm{\nabla f}^{2})}-2\bigl<\nabla f,\nabla(\Delta f-\inner{\nabla V,\nabla f})\bigr>\bigr\}\nonumber \\
 & =\half\,\Delta(\norm{\grad f}^{2})-\inner{\nabla V,\nabla^{2}f\,\nabla f}-\inner{\nabla f,\nabla\Delta f}+\bigl<\grad f,\nabla\inner{\nabla V,\nabla f}\bigr>\nonumber \\
 & \underset{(i)}{=}\inner{\Delta\nabla f,\nabla f}+\norm{\nabla^{2}f}_{\hs}^{2}-\inner{\nabla V,\nabla^{2}f\,\nabla f}-\inner{\nabla f,\nabla\Delta f}+\inner{\grad f,\hess V\,\nabla f+\nabla^{2}f\,\nabla V}\nonumber \\
 & =\norm{\nabla^{2}f}_{\hs}^{2}+\inner{\grad f,\hess V\,\nabla f}\,,\nonumber 
\end{align}
where we used (\ref{eq:Bochner}) in $(i)$. Therefore, the iterated
carr\'e du champ for the Langevin semigroup is
\[
\Gamma_{2}(f,f)=\norm{\nabla f}_{\hess V}^{2}+\norm{\hess f}_{\hs}^{2}\,.
\]
\end{example}

We are ready to prove the \emph{integrated Bochner formula} for the
Langevin generator:
\begin{lem}
[Integrated Bochner formula] For the Langevin generator $\mc L=\mc L_{V}$
and function $f$,
\begin{equation}
\int(\mc Lf)^{2}\,\D\pi=\int\norm{\nabla f}_{\hess V}^{2}\,\D\pi+\int\norm{\nabla^{2}f}_{\hs}^{2}\,\D\pi\,.\tag{\ensuremath{\msf{Int}}-\ensuremath{\msf{Bochner}}}\label{eq:int-Bochner}
\end{equation}
\end{lem}

\begin{proof}
Using (\ref{eq:IBP}) and (\ref{eq:gamma2-generator}),
\begin{align*}
\int(\mc Lf)^{2}\,\D\pi & =-\int\inner{\nabla f,\nabla\mc Lf}\,\D\pi=\int\bpar{\Gamma_{2}(f,f)-\half\,\mc L(\norm{\nabla f}^{2})}\,\D\pi\\
 & =\int\Gamma_{2}(f,f)\,\D\pi=\int\bpar{\norm{\nabla f}_{\hess V}^{2}+\norm{\hess f}_{\hs}^{2}}\,\D\pi\,,
\end{align*}
where we used $\int\mc Lh\,\D\pi=0$ for any function $h$ (which
follows from (\ref{eq:IBP})).
\end{proof}

\subsection{Improved Lichnerowicz inequality}

Using the tools developed above, we can prove the improved Lichnerowicz
inequality \cite{Klartag23log}.
\begin{thm}
[Improved Lichnerowicz inequality] Let $\pi\propto\exp(-V)$ be
a $t$-strongly log-concave probability measure over $\Rn$. Then,
\begin{equation}
\cpi(\pi)\leq\sqrt{\frac{\norm{\cov\pi}}{t}}\,.\tag{{\ensuremath{\msf{ILI}}}}\label{eq:improved-LI}
\end{equation}
\end{thm}

\begin{rem}
It is well known that $\cpi(\pi)\leq1/t$ in the same setting, which
is the so-called \emph{Lichnerowicz inequality} (or an implication
of the Brascamp--Lieb inequality). Since $\norm{\cov\pi}\leq\cpi(\pi)$
in general, the improved LI recovers the classical LI as a special
case. Moreover, it provides a tighter bound than the classical version.

The bound above can be viewed as the geometric average of the KLS
conjecture (i.e., $\cpi(\pi)\lesssim\norm{\cov\pi}$) and the LI (i.e.,
$\cpi(\pi)\leq1/t$).
\end{rem}

\begin{proof}
We relate an upper and lower bound on $\norm{\int\nabla\vphi}$. Let
$\mc L$ denote the Langevin generator associated with the potential
$V$.

As for a lower bound, let $\vphi_{1}$ be an eigenfunction of $\lda_{1}=\cpi(\pi)^{-1}$
with unit norm, which satisfies $\lda_{1}=\inner{\vphi_{1},(-\mc L)\vphi_{1}}_{\ltwo(\pi)}=\E_{\pi}[\norm{\nabla\vphi_{1}}^{2}]$.
Using (\ref{eq:int-Bochner}) with a test function $\vphi=\vphi_{1}$,
and hiding $\D\pi$ in $\int$,
\begin{align*}
\lda_{1}^{2} & =\int(\mc L\vphi)^{2}=\int\norm{\nabla\vphi}_{\hess V}^{2}+\int\norm{\nabla^{2}\vphi}_{\text{F}}^{2}\geq t\int\norm{\nabla\vphi}^{2}+\sum_{i=1}^{n}\underbrace{\int\norm{\nabla(\de^{i}\vphi)}^{2}}_{\text{Use PI to each }\de^{i}\vphi}\\
 & \geq t\int\norm{\nabla\vphi}^{2}+\sum_{i=1}^{n}\lda_{1}\Bpar{\int(\de^{i}\vphi)^{2}-\Bpar{\int\de^{i}\vphi}^{2}}=t\underbrace{\int\norm{\nabla\vphi}^{2}}_{=\lda_{1}}+\lda_{1}\Bpar{\int\norm{\nabla\vphi}^{2}-\Bnorm{\int\nabla\vphi}^{2}}\\
 & =\lda_{1}t+\lda_{1}^{2}-\lda_{1}\Bnorm{\int\nabla\vphi}^{2}\,.
\end{align*}
Rearranging terms,
\[
\Bnorm{\int\nabla\vphi}^{2}\geq t\,.
\]
We now move to an upper bound. For some unit vector $\theta\in\mbb S^{n-1}$,
\begin{align*}
\Bnorm{\int\nabla\vphi} & =\int\nabla\vphi\cdot\theta=\sum_{i=1}^{n}\int\inner{\nabla\vphi,\nabla(x_{i}-\mu_{i})}\,\theta_{i}\underset{\eqref{eq:IBP}}{=}-\int\mc L\vphi\,(x-\mu)\cdot\theta\\
 & \leq\sqrt{\int(\mc L\vphi)^{2}}\sqrt{\int\bpar{(x-\mu)\cdot\theta}^{2}}\leq\lda_{1}\sqrt{\norm{\cov\pi}}\,.
\end{align*}
Relating the upper and lower bound on $\norm{\int\nabla\vphi}$, we
have
\[
t\leq\lda_{1}^{2}\,\norm{\cov\pi}\Longleftrightarrow\cpi(\pi)\leq\sqrt{\frac{\norm{\cov\pi}}{t}}\,,
\]
which completes the proof.\selectlanguage{english}%
\end{proof}

\section{A Connection between Stochastic Localization and the Proximal Sampler\label{sec:PS-SL}}

 Here we interpret stochastic localization (SL) using the $\ps$
($\PS$) \cite{LST21structured} for ease of exposition. The $\PS$
can be viewed as a \emph{realization} of a static version of SL. We
then develop semigroup machinery for studying static SL, and finally
move to SL. For additional material, we refer to \cite{KP23spectral}
and \cite[\S8]{chewi25log}.

\subsection{Proximal sampler}

Let $\pi^{X}\propto\exp(-V)$ be a target distribution we aim to sample
from. The $\ps$ with step size $h$ considers the augmented distribution
\[
\bs{\pi}_{h}(x,y)\propto\exp\bpar{-V(x)-\frac{1}{2h}\,\norm{x-y}^{2}}\,.
\]
With the abbreviation $\bs{\pi}=\bs{\pi}_{h}$, and writing $\pi^{X|Y}=\bs{\pi}^{X|Y}$
and $\pi^{Y|X}=\bs{\pi}^{Y|X}$, one iteration consists of two steps---the
forward step and backward step:
\begin{itemize}
\item {[}Forward{]} Sample $y_{k+1}\sim\pi^{Y|X=x_{k}}=\msf N(x_{k},hI_{n})$
\item {[}Backward{]} Sample $x_{k+1}\sim\pi^{X|Y=y_{k+1}}\propto\exp\bpar{-V(\cdot)-\frac{1}{2h}\,\norm{\cdot-y_{k+1}}^{2}}$
\end{itemize}
Essentially, the $\PS$ runs Gibbs sampling over $\bs{\pi}$, so $\bs{\pi}$
is clearly its stationary distribution. Thus, if the $\PS$ starts
from the target distribution (i.e., $x_{0}\sim\pi^{X}$), the forward
step sends $\pi^{X}$ to $\pi^{Y}=\pi^{X}*\gamma_{h}$, and the backward
step takes $\pi^{Y}$ back to $\pi^{X}$ through the backward channel
$\pi^{X|Y=y}$:
\[
\pi^{X}=\int\pi^{X|Y=y}\,\D\pi^{Y}(y)=\E_{y\sim\pi^{Y}}\pi^{X|Y=y}=:\E_{\pi^{Y}}\pi^{X|Y}\,.
\]
In other words, $\pi^{X}$ can be viewed as \textbf{a mixture of the
conditional distributions $\pi^{X|Y=y}$ with $y\sim\pi^{Y}$}.

\paragraph{SDE interpretation.}

Suppose we want to understand how fast the $\PS$ converges to $\pi^{X}$.
This leads us to study the forward and backward processes. The forward
step can be interpreted as simulating a heat flow for time $h$:

\[
\D Z_{t}=\D W_{t}\,,\quad Z_{0}\sim X_{k}\Longrightarrow Z_{h}\sim Y_{k+1}\,.
\]
Note that $(Z_{0},Z_{h})\sim\bs{\pi}$ when $Z_{0}\sim\pi^{X}$. Hence,
the heat semigroup $(Q_{h})_{h\geq0}$ naturally arise when studying
the forward step, and it is natural to ask whether one can introduce
a corresponding semigroup structure for the backward step as well.
It turns out that the \emph{backward heat flow} can be made precise:
\[
\D Z_{t}^{\leftarrow}=\grad\log(\pi^{X}*\gamma_{h-t})(Z_{t}^{\leftarrow})\,\D t+\D W_{t}\,,\quad Z_{0}^{\leftarrow}\sim Y_{k+1}\Longrightarrow Z_{h}^{\leftarrow}\sim X_{k+1}\,.
\]
In words, the backward step corresponds to simulating the backward
heat flow for time $h$. For example, if $Z_{0}^{\leftarrow}\sim\pi^{Y}$,
then $Z_{h}^{\leftarrow}\sim\pi^{X}$. Since the backward step acts
as a \emph{channel} $\pi^{X|Y=y}$, it is natural to guess that if
$Z_{0}^{\leftarrow}=\delta_{y}$, then $Z_{h}^{\leftarrow}=\pi^{X|Y=y}$
(though justifying this rigorously is non-trivial). Therefore, one
may introduce a semigroup structure for the backward heat flow via
\begin{equation}
Q_{h}^{\leftarrow}f(y)\defeq\E[f(Z_{h}^{\leftarrow})|Z_{0}^{\leftarrow}=y]=\E_{\pi^{X|Y=y}}f\,.\label{eq:ps-adjoint}
\end{equation}

As we will see in the next section, Klartag and Putterman \cite{KP23spectral}
defined the same object as an \emph{adjoint of the heat semigroup}
and developed its corresponding ``generator'' along with calculus
rules, thereby providing a framework for carrying out spectral analyses
of the backward step.
\begin{rem}
[Non-homogeneous backward process] One may notice that the semigroup
and adjoint associated with the backward heat flow are handled quite
differently compared to those for the heat flow and Langevin flow.
In the case of the heat flow or Langevin flow, we have a single SDE
(with no external parameter like $h$), from which the semigroup and
generator are derived. In contrast, the backward heat flow involves
an additional parameter $h$, and as $h$ varies, the corresponding
SDE changes accordingly.
\end{rem}

\begin{rem}
[Connection to SL] The backward step of the $\PS$ with step size
$h$ is encoded by $\pi^{X|Y}=\pi_{h}^{X|Y}$. Let us examine two
extreme cases, $h=0$ and $h=\infty$. Below, let $X_{0}\sim\pi^{X}$.

As $h\to0$, we have $Y_{1}=X_{0}+\msf N(0,hI_{n})\to X_{0}\sim\pi^{X}$,
and $\pi_{h}^{X|Y_{1}}$ becomes concentrated, vanishing unless $X=X_{0}$.
Thus, we expect that 
\[
\pi_{h}^{X|Y}\to\delta_{X_{0}}\quad\text{with }X_{0}\sim\pi^{X}\,.
\]
When $h\to\infty$, the density of $\pi^{X|Y}$ suggests that it converges
to $\exp(-V)\propto\pi^{X}$.

SL provides an interpolation between these two extreme cases when
viewed using ``time'' $t=1/h$. In this sense, SL provides a \emph{dynamic
perspective} on the backward step of the $\PS$. Vice versa, the $\PS$
backward step with fixed $h$ can be regarded as a snapshot of SL
at time $t=1/h$. 
\end{rem}

\subsection{Adjoint of the heat semigroup}

We provide a gentle introduction to the heat adjoint, as defined by
Klartag and Putterman \cite{KP23spectral}, through the lens of the
$\PS$ framework.

Let $\D Z_{t}=\D W_{t}$ with $Z_{0}\sim\pi=\pi^{X}$. Then the heat
semigroup $(Q_{h})_{h\geq0}$ is defined as
\[
Q_{h}f(x)=\E[f(Z_{h})\,|\,Z_{0}=x]=(f*\gamma_{h})(x)\,.
\]
Denoting $\pi_{h}:=\pi*\gamma_{h}=\pi^{Y}$ (which corresponds to
the forward step of the $\PS$), we note that $Q_{h}:L^{2}(\pi_{h})\to\ltwo(\pi)$:
for $f\in\ltwo(\pi_{h})$,
\[
\norm{Q_{h}f}_{\ltwo(\pi)}^{2}\leq\int Q_{h}(f^{2})\,\D\pi=\E_{\pi}\bbrack{\E[f^{2}(Z_{h})|Z_{0}=x]}=\E_{\pi_{h}}[f^{2}(Z_{h})]=\norm f_{\ltwo(\pi_{h})}^{2}\,.
\]

\subsubsection{Heat adjoint\label{subsec:heat-adjoint}}

We consider its adjoint $Q_{h}^{*}:L^{2}(\pi)\to L^{2}(\pi_{h})$
defined in the way that for any $f\in\ltwo(\pi)$ and $g\in L^{2}(\pi_{h})$,
\[
\inner{f,Q_{h}g}_{\ltwo(\pi)}=\inner{Q_{h}^{*}f,g}_{\ltwo(\pi_{h})}\,.
\]
The RHS is simply 
\[
\int Q_{h}^{*}f\cdot g\,\D\pi_{h}\,,
\]
and the LHS equals
\[
\inner{f,Q_{h}g}_{\ltwo(\pi)}=\E\bbrack{f(Z_{0})\,\E[g(Z_{h})|Z_{0}]}=\E[f(Z_{0})\,g(Z_{h})]=\E_{Z_{h}\sim\pi_{h}}\bbrack{g(Z_{h})\,\E[f(Z_{0})|Z_{h}]}\,.
\]
Relating the two equations above,
\[
Q_{h}^{*}f(y)=\E[f(Z_{0})|Z_{h}=y]\,.
\]
Due to $(Z_{0},Z_{h})\sim\bs{\pi}(x,y)$, we have $\law(Z_{0}|Z_{h}=y)=\pi^{X|Y=y}$,
so
\[
Q_{h}^{*}f(y)=\E_{\pi^{X|Y=y}}f\,.
\]
This is exactly the same as (\ref{eq:ps-adjoint})---$Q_{h}^{\leftarrow}f(y)=\E_{\pi^{X|Y=y}}f$.
Therefore, we can view the heat adjoint as the backward step of the
$\PS$.

Just as $Q_{h}f=f*\gamma_{h}$, we can also derive a formula for the
heat adjoint: using $\pi^{X|Y}\pi^{Y}=\bs{\pi}$ and noting $\bs{\pi}(x,y)=\pi(x)\,\gamma_{h}(y-x)$,
\begin{align*}
Q_{h}^{*}f(y) & =\int f\,\D\pi^{X|Y=y}=\frac{\int f(x)\,\bs{\pi}(x,y)\,\D x}{\pi^{Y}}=\frac{\int f(x)\,\pi(x)\,\gamma_{h}(y-x)\,\D x}{\pi*\gamma_{h}}=\frac{(f\pi)*\gamma_{h}}{\pi*\gamma_{h}}\\
 & =\frac{Q_{h}(f\pi)}{Q_{h}\pi}=\frac{(f\pi)_{h}}{\pi_{h}}\,.
\end{align*}

\subsubsection{Dynamic backward step $\equiv$ Stochastic localization}

For $X\sim\pi=\pi^{X}$, the forward step of the $\PS$ sends it to
$Y=X+W_{h}$. Using a standard fact that $\frac{1}{t}W_{t}\overset{d}{=}W_{1/t}'$
($W_{t}'$ is another BM), we know that for $t=1/h$
\[
tY=tX+tW_{1/t}\overset{d}{=}tX+W_{t}'\,.
\]
Hence, $\bar{\theta}_{t}:=tX+W_{t}'$ has the same law as $tY$. Since
$Y\overset{d}{=}\bar{\theta}_{t}/t$, the backward step of the $\PS$
can be thought of as 
\begin{align}
\pi^{X|Y} & \propto\exp\bpar{-V(x)-\frac{1}{2h}\,\bnorm{x-\frac{\bar{\theta}_{t}}{t}}^{2}}\propto\pi(x)\,\exp\bpar{-\frac{t}{2}\,\bnorm{x-\frac{\bar{\theta}_{t}}{t}}^{2}}\nonumber \\
 & \propto\pi(x)\,\exp\bpar{\inner{\bar{\theta}_{t},x}-\frac{t}{2}\,\norm x^{2}}\,.\label{eq:ps-backward-theta}
\end{align}
By comparing the last line with (\ref{eq:SL-pdf}), it is plausible
to view SL as a dynamic version of the backward step of the $\PS$.

\section{Guan's $\log\log n$-bound via Klartag and Lehec's spectral analysis}

Guan's bound is obtained by combining his uniform trace bound with
an improved spectral analysis. Overall, Guan follows the approach
of Klartag and Lehec  (KL) toward proving $\sigma_{\pi}\lesssim\log^{4}n$
\cite{KL22Bourgain}. It also starts with Barthe and Klartag's $H^{-1}$-inequality,
and Guan focuses on bounding $\sum_{i=1}^{n}\norm{x_{i}}_{H^{-1}(\pi)}^{2}$
(see (\ref{eq:hinverse-for-thin})).

It is a standard fact that for a centered $L^{2}(\pi)$ function $f$
(i.e., $\E_{\pi}f=0$), 
\[
\norm f_{H^{-1}(\pi)}^{2}=\int_{\lda_{1}}^{\infty}\frac{\inner{E_{\lda}f,f}}{\lda^{2}}\,\D\lda\,,
\]
where $\lda_{1}$ is the spectral gap of the Langevin generator $\mc L$
associated with $\pi$ (i.e., $\lda_{1}^{-1}=\cpi(\pi)$), and $E_{\lda}$
is the \emph{spectral projection} below level $\lda$. In terms of
a matrix analogy, one can think of it as the orthogonal projection
onto the subspace spanned by eigenvectors whose eigenvalues are at
most $\lda$. We will return to this later in more detail.

Putting these together,
\begin{equation}
\sigma_{\pi}^{2}\leq\frac{4}{n}\sum_{i=1}^{n}\int_{\lda_{1}}^{\infty}\frac{\inner{E_{\lda}x_{i},x_{i}}}{\lda^{2}}\,\D\lda=:4\int_{\lda_{1}}^{\infty}\frac{F(\lda)}{\lda^{2}}\,\D\lda\,,\label{eq:sigmasquare-bound}
\end{equation}
where $F(\lda):=\frac{1}{n}\sum_{i=1}^{n}\inner{E_{\lda}x_{i},x_{i}}=\frac{1}{n}\sum_{i=1}^{n}\norm{E_{\lda}x_{i}}_{L^{2}(\pi)}^{2}$
can be interpreted as the average spectral mass of the coordinate
functions below level $\lda$. This allows KL to focus on bounding
$F(\lda)$. Then, KL related $F(\lda)$ to $\norm{Q_{h}^{*}x}_{L^{2}(\pi Q_{h})}^{2}:=\sum_{i=1}^{n}\norm{Q_{h}^{*}x_{i}}_{L^{2}(\pi Q_{h})}^{2}$
for $\pi Q_{h}:=\pi*\gamma_{h}$ as follows:
\begin{itemize}
\item \textbf{LB}: $\norm{Q_{h}^{*}E_{\lda}x_{i}}_{L^{2}(\pi Q_{h})}^{2}\geq\exp\bpar{-h\int\abs{\nabla E_{\lda}x_{i}}^{2}\,\D\pi}\,\norm{E_{\lda}x_{i}}_{L^{2}(\pi)}^{2}$.
\item \textbf{UB}: $\norm{Q_{h}^{*}E_{\lda}x_{i}}_{L^{2}(\pi Q_{h})}^{2}\leq\norm{Q_{h}^{*}x_{i}}_{L^{2}(\pi Q_{h})}^{2}+\text{terms involving }Q_{h}^{*}\,(I-E_{\lda})x_{i}$.
\end{itemize}
Summing over $i\in[n]$ and chaining these two bounds, KL established
\[
F(\lda)\lesssim\frac{1}{n}\sum_{i=1}^{n}\norm{Q_{h}^{*}x_{i}}_{L^{2}(\pi Q_{h})}^{2}+\lda h\,.
\]

Since SL  is a dynamic version of the heat adjoint, we can introduce
SL  to carry out a finer analysis of $\norm{Q_{h}^{*}x}_{L^{2}(\pi Q_{h})}^{2}$.
Precisely, for $t=1/h$, $\pi=\pi^{X}$, $\pi^{Y}=\pi^{X}*\gamma_{h}$,
and SL $(\pi_{t})_{t\geq0}$ with $\pi_{0}=\pi$,
\begin{align*}
\E[\norm{b_{t}}^{2}] & =\E\Bbrack{\Bnorm{\int x\,\D\pi_{t}}^{2}}=\sum_{i=1}^{n}\E\Bbrack{\Bpar{\int x_{i}\,\D\pi_{t}}^{2}}\\
 & =\sum_{i=1}^{n}\E_{\pi^{Y}}[(\E_{\pi^{X|Y}}x_{i})^{2}]=\sum_{i=1}^{n}\E_{\pi^{X}*\gamma_{h}}[(Q_{h}^{*}x_{i})^{2}]\\
 & =\sum_{i=1}^{n}\norm{Q_{h}^{*}x_{i}}_{L^{2}(\pi Q_{h})}^{2}=\norm{Q_{h}^{*}x}_{L^{2}(\pi Q_{h})}^{2}\,,
\end{align*}
where we used the language of the $\ps$. Then, using $\de_{t}\E[\norm{b_{t}}^{2}]=\E\tr(\Sigma_{t}^{2})$
from (\ref{eq:derivative-bt}) and known bounds on $\tr(\Sigma_{t}^{2})$
at that time, KL managed to bound $F(\lda)$, which in turn implied
the $\O(\log^{4}n)$-bound on $\sigma_{\pi}$.

Guan analyzes the \textbf{UB} part more carefully while leveraging
the uniform trace bound (i.e., $\E\tr(\Sigma_{t}^{2})=\O(n)$ from
Theorem~\ref{thm:trace-control}) and an improved spectral analysis.
The main components are integration by parts, the integrated Bochner
formula, and the properties of the dynamic $\Gamma$-calculus developed
by Klartag and Putterman \cite{KP23spectral}. By chaining the LB
with the improved UB, Guan establishes 
\[
e^{-h\lda}F(\lda)\lesssim h^{-1}+\sqrt{\lda F(\lda)\,(\lceil\log_{2}h\rceil+1)}+\lda hF(\lda)\,.
\]
Taking $h=(\lda F(\lda))^{-1/2}$, the paper deduces
\[
F(\lda)\lesssim\lda\,\abs{\log\lda}\,.
\]
Putting these back to (\ref{eq:sigmasquare-bound}), Guan concludes
that $\sigma_{\pi}\lesssim\abs{\log\lda_{1}}\asymp\log\cpi(\pi)\lesssim\log\log n$
as desired. 

\subsection{Calculus on adjoint of heat semigroup}

We now delve into the adjoint $Q_{h}^{*}$ of heat semigroups in more
details. We recommend readers read \S\ref{sec:semigroup} if they
are not familiar with this topic. In the following two subsections,
$\pi_{t}$ denotes $\pi*\gamma_{t}$, not the SL $\pi_{t}$.

\subsubsection{Forward equation and box operator}

Markov semigroups in general satisfy that for test functions $f$,
\[
\de_{t}P_{t}f=\mc LP_{t}f=P_{t}\mc Lf\,,
\]
as seen in

\[
\de_{t}P_{t}f=\lim_{\delta\to0}\frac{P_{t+\delta}f-P_{t}f}{\delta}=P_{t}\lim_{\delta\to0}\frac{P_{\delta}f-f}{\delta}=P_{t}\mc Lf\,.
\]
Thus, for the heat semigroup $Q_{t}:L^{2}(\pi_{t})\to L^{2}(\pi)$,
as its generator is $\mc L=\half\,\Delta$, we have that
\[
\de_{t}Q_{t}f=\half\,\Delta Q_{t}f\,.
\]

Recall that we derived in \S\ref{subsec:heat-adjoint} that the heat
adjoint $Q_{t}^{*}:L^{2}(\pi)\to L^{2}(\pi_{t})$ satisfies 
\[
Q_{t}^{*}f=\frac{Q_{t}(f\pi)}{Q_{t}\pi}\,.
\]
Taking gradients of both sides of $Q_{t}\pi\cdot Q_{t}^{*}f=Q_{t}(f\pi)$,
we obtain
\begin{align*}
\nabla Q_{t}(f\pi) & =Q_{t}^{*}f\,\nabla Q_{t}\pi+Q_{t}\pi\,\nabla Q_{t}^{*}f\,,\\
\Delta Q_{t}(f\pi) & =Q_{t}^{*}f\,\Delta Q_{t}\pi+2\inner{\nabla Q_{t}^{*}f,\nabla Q_{t}\pi}+Q_{t}\pi\,\Delta Q_{t}^{*}f\,.
\end{align*}
Using this, we can directly compute that
\begin{align*}
\de_{t}Q_{t}^{*}f & =\frac{\de_{t}Q_{t}(f\pi)}{Q_{t}\pi}-\frac{Q_{t}(f\pi)}{(Q_{t}\pi)^{2}}\,\de_{t}Q_{t}\pi\\
 & =\half\,\frac{\Delta Q_{t}(f\pi)}{Q_{t}\pi}-\half\,\Delta Q_{t}\pi\,\frac{Q_{t}(f\pi)}{(Q_{t}\pi)^{2}}\\
 & =\half\,\frac{\Delta Q_{t}(f\pi)}{Q_{t}\pi}-\half\,\frac{Q_{t}^{*}f\,\Delta Q_{t}\pi}{Q_{t}\pi}\\
 & =\half\,\Delta Q_{t}^{*}f+\inner{\nabla Q_{t}^{*}f,\nabla\log Q_{t}\pi}\\
 & =\half\,\Delta Q_{t}^{*}f+\inner{\nabla Q_{t}^{*}f,\nabla\log\pi_{t}}
\end{align*}
Therefore, comparing this with $\de_{t}Q_{t}f=\mc LQ_{t}f=\half\,\Delta Q_{t}f$,
we are led to define a sort of ``generator'' of $Q_{t}^{*}$ (called
the box operator in \cite{KP23spectral}) as follows:
\begin{equation}
\square_{t}f=\half\,\Delta f+\inner{\nabla\log\pi_{t},\nabla f}\,,\tag{\ensuremath{\msf{box}}}\label{eq:box-generator}
\end{equation}
which looks quite similar to the Langevin operator associated with
$\pi_{t}$, given as 
\[
\mc L_{t}f=\Delta f+\inner{\nabla\log\pi_{t},\nabla f}\,.
\]
Note that 
\[
\de_{t}Q_{t}^{*}f=\square_{t}Q_{t}^{*}f\,.
\]

\subsubsection{A dynamic version of $\Gamma$-calculus}

Given a Langevin or heat-flow generator $\mc L$, we defined the (iterated)
carr\'e du champ as 
\begin{align*}
\Gamma_{0}(f,g) & :=fg\,,\\
\Gamma_{i+1}(f,g) & :=\half\,\bpar{\mc L\Gamma_{i}(f,g)-\Gamma_{i}(f,\mc Lg)-\Gamma_{i}(\mc Lf,g)}\ \text{for }i\geq0\,.
\end{align*}
Inspired by this, we can define this $\Gamma$-notion for the box
operator $\square_{t}$:
\begin{align*}
\Gamma_{0}(f,g) & :=fg\,,\\
\Gamma_{i+1}(f,g) & :=\square_{t}\Gamma_{i}(f,g)-\Gamma_{i}(\square_{t}f,g)-\Gamma_{i}(f,\square_{t}g)-\frac{\D\Gamma_{i}}{\D t}(f,g)\,.
\end{align*}
We note that unlike the Langevin case (which admits a stationary distribution),
this has dependence on the time $t$.

\paragraph{The formulas of $\Gamma_{1}$ and $\Gamma_{2}$.}

Let us compute this:
\begin{align*}
\Gamma_{1}(f,g) & =\square_{t}(fg)-g\,\square_{t}f-f\,\square_{t}g\\
 & =\mc L_{t}(fg)-g\,\mc L_{t}f-f\,\mc L_{t}g-\half\,\Delta(fg)+\half\,g\Delta f+\half\,f\Delta g\\
 & =2\inner{\nabla f,\nabla g}-\half\,(f\,\Delta g+2\inner{\nabla f,\nabla g}+g\,\Delta f)+\half\,g\Delta f+\half\,f\Delta g\\
 & =\inner{\nabla f,\nabla g}\,,
\end{align*}
so it is actually the same as the carr\'e du champ of the heat and
Langevin generators. 

Using this, we can immediately obtain that
\[
\Gamma_{2}(f,g)=\square_{t}\inner{\nabla f,\nabla g}-\inner{\nabla\square_{t}f,\nabla g}-\inner{\nabla f,\nabla\square_{t}g}\,.
\]
Let $\Gamma_{2}^{L}$ be the iterated carr\'e du champ of the Langevin
generator (associated with $\pi$), which satisfies
\[
\Gamma_{2}^{L}(f,f)=\norm{\nabla^{2}f}_{\hs}^{2}-\inner{\nabla f,\nabla^{2}\log\pi\,\nabla f}\,.
\]
Due to $\mc L_{t}=\square_{t}+\half\,\Delta$, 
\begin{align}
\Gamma_{2}(f,f) & =\square_{t}(\norm{\nabla f}^{2})-2\inner{\nabla\square_{t}f,\nabla f}\nonumber \\
 & =2\Gamma_{2}^{L}(f,f)-\half\,\Delta(\norm{\nabla f}^{2})+\inner{\nabla\Delta f,\nabla f}\nonumber \\
 & =\norm{\nabla^{2}f}_{\hs}^{2}-2\inner{\nabla f,\nabla^{2}\log\pi_{t}\,\nabla f}\,,\label{eq:dynamic-cd-champ}
\end{align}
where we used (\ref{eq:Bochner}), $\half\,\Delta(\norm{\nabla f}^{2})=\inner{\Delta\nabla f,\nabla f}+\norm{\nabla^{2}f}_{\hs}^{2}$.

\paragraph{A handy lemma.}

All the things developed lead to the following lemma.
\begin{lem}
Under some regularity conditions on $f$, for $i=0,1$,
\begin{equation}
\de_{t}\int\Gamma_{i}(Q_{t}^{*}f,Q_{t}^{*}f)\,\D(Q_{t}\pi)=-\int\Gamma_{i+1}(Q_{t}^{*}f,Q_{t}^{*}f)\,\D(Q_{t}\pi)\,.\tag{\ensuremath{\msf{iter}}-\ensuremath{\msf{diff}}}\label{eq:iterated-diff-Gamma}
\end{equation}
\end{lem}

\begin{proof}
Let $f_{t}:=Q_{t}^{*}f$ and $\pi_{t}=Q_{t}f$. Then, using $\de_{t}f_{t}=\square_{t}f_{t}$
and $\de_{t}\pi_{t}=\mc L\pi_{t}=\half\,\Delta\pi_{t}$,
\begin{align*}
\de_{t}\int\Gamma_{i}(Q_{t}^{*}f,Q_{t}^{*}f)\,\D(Q_{t}\pi) & =\de_{t}\int\Gamma_{i}(f_{t},f_{t})\,\pi_{t}\\
 & =\int\bpar{\de_{t}\Gamma_{i}(f_{t},f_{t})\,\pi_{t}+2\Gamma_{i}(\square_{t}f_{t},f_{t})\,\pi_{t}+\Gamma_{i}(f_{t},f_{t})\,\frac{\Delta\pi_{t}}{2}}\\
 & \underset{\eqref{eq:IBP}}{=}\int\bpar{2\Gamma_{i}(\square_{t}f_{t},f_{t})+\half\,\Delta\Gamma_{i}(f_{t},f_{t})+\de_{t}\Gamma_{i}(f_{t},f_{t})}\,\pi_{t}\\
 & =\int\bpar{-\square_{t}\Gamma_{i}(f_{t},f_{t})+2\Gamma_{i}(\square_{t}f_{t},f_{t})+\de_{t}\Gamma_{i}(f_{t},f_{t})}\,\pi_{t}\\
 & \qquad+\int\bpar{\square_{t}\Gamma_{i}(f_{t},f_{t})+\half\,\Delta\Gamma_{i}(f_{t},f_{t})}\,\pi_{t}\\
 & =-\int\Gamma_{i+1}(f_{t},f_{t})\,\pi_{t}+\int\mc L_{t}\Gamma_{i}(f_{t},f_{t})\,\pi_{t}\\
 & =-\int\Gamma_{i+1}(f_{t},f_{t})\,\D\pi_{t}\,,
\end{align*}
where we used $\int\mc L\vphi\,\D\pi=0$ for (regular) test functions
$\vphi$ (e.g., shown by (\ref{eq:IBP})).
\end{proof}

\subsubsection{Decreasing Rayleigh quotient}

Recall that the Rayleigh quotient with respect to a function $f$
is defined as 
\[
R_{f}=\frac{\int\norm{\nabla f}^{2}\,\D\pi}{\int f^{2}\,\D\pi}\,.
\]
We now consider its time-varying version, defined as
\[
R_{f}(t)=\frac{\int\norm{\nabla Q_{t}^{*}f}^{2}\,\D(Q_{t}\pi)}{\int(Q_{t}^{*}f)^{2}\,\D(Q_{t}\pi)}=\frac{\int\norm{\nabla f_{t}}^{2}\,\D\pi_{t}}{\int f_{t}^{2}\,\D\pi_{t}}\,,
\]
and we show that this is non-increasing in $t$ when $\pi$ is a logconcave
distribution.
\begin{thm}
[Decreasing Rayleigh] Let $\pi$ be an absolutely continuous and
logconcave probability measure over $\Rn$. Given a regular $f$ (in
the Sobolev space $H^{1}(\pi)$), the Rayleigh quotient 
\[
R_{f}(t)=\frac{\int\norm{\nabla f_{t}}^{2}\,\D\pi_{t}}{\int f_{t}^{2}\,\D\pi_{t}}
\]
is non-increasing in $t\geq0$. In particular, the function $\norm{f_{t}}_{L^{2}(\pi_{t})}$
is log-convex in $t\geq0$.
\end{thm}

\begin{proof}
Let us differentiate the quotient in $t$. Using (\ref{eq:iterated-diff-Gamma})
for $i=0,1$,
\begin{align*}
\de_{t}R_{f}(t) & =\de_{t}\frac{\int\Gamma_{1}(f_{t},f_{t})\,\D\pi_{t}}{\int\Gamma_{0}(f_{t},f_{t})\,\D\pi_{t}}\\
 & =-\frac{\int\Gamma_{2}(f_{t},f_{t})\,\D\pi_{t}}{\int\Gamma_{0}(f_{t},f_{t})\,\D\pi_{t}}+\Bpar{\frac{\int\Gamma_{1}(f_{t},f_{t})\,\D\pi_{t}}{\int\Gamma_{0}(f_{t},f_{t})\,\D\pi_{t}}}^{2}\\
 & =\frac{\bpar{\int\norm{\nabla f_{t}}^{2}\,\D\pi_{t}}^{2}}{\norm{f_{t}}_{L^{2}(\pi_{t})}^{4}}-\frac{\int(\norm{\nabla^{2}f_{t}}_{\hs}^{2}-2\inner{\nabla f_{t},\nabla^{2}\log\pi_{t}\,\nabla f_{t}})\,\D\pi_{t}}{\norm{f_{t}}_{L^{2}(\pi_{t})}^{2}}\\
 & \underset{(i)}{=}\frac{\bpar{\int\norm{\nabla f_{t}}^{2}\,\D\pi_{t}}^{2}-\int(\mc L_{t}f_{t})^{2}\,\D\pi_{t}\,\norm{f_{t}}_{L^{2}(\pi_{t})}^{2}}{\norm{f_{t}}_{L^{2}(\pi_{t})}^{4}}+\frac{\int\inner{\nabla f_{t},\nabla^{2}\log\pi_{t}\,\nabla f_{t}}\,\D\pi_{t}}{\norm{f_{t}}_{L^{2}(\pi_{t})}^{2}}\\
 & \leq\frac{\bpar{\int\norm{\nabla f_{t}}^{2}\,\D\pi_{t}}^{2}-\int(\mc L_{t}f_{t})^{2}\,\D\pi_{t}\,\norm{f_{t}}_{L^{2}(\pi_{t})}^{2}}{\norm{f_{t}}_{L^{2}(\pi_{t})}^{4}}\leq0
\end{align*}
where in $(i)$ we used (\ref{eq:int-Bochner}), and the last inequality
follows from
\[
\int\norm{\nabla f_{t}}^{2}\,\D\pi_{t}=\int\Gamma_{1}(f_{t},f_{t})\,\D\pi_{t}=\int(-\mc L_{t}f_{t})f_{t}\,\D\pi_{t}\leq\sqrt{\int(\mc L_{t}f_{t})^{2}\,\D\pi_{t}}\sqrt{\int f_{t}^{2}\,\D\pi_{t}}\,.
\]

As for the second claim, we note that
\[
\de_{t}\log\norm{f_{t}}_{L^{2}(\pi_{t})}^{2}=\frac{\de_{t}\int\Gamma_{0}(f_{t},f_{t})\,\D\pi_{t}}{\norm{f_{t}}_{L^{2}(\pi_{t})}^{2}}=-\frac{\int\Gamma_{1}(f_{t},f_{t})\,\D\pi_{t}}{\norm{f_{t}}_{L^{2}(\pi_{t})}^{2}}=-R_{f}(t)\,,
\]
so $\de_{t}^{2}(\log\norm{f_{t}}_{L^{2}(\pi_{t})}^{2})=-\de_{t}R_{f}(t)\geq0$.
\end{proof}
This gives us an important observation that will be used to study
the thin-shell conjecture.
\begin{cor}
[Relating $Q_h^*f$ and $f$] \label{cor:adjoint-norm-change} Let
$f\in H^{1}(\pi)$. Then for all $h>0$,
\[
\frac{\norm{Q_{h}^{*}f}_{L^{2}(\pi_{h})}^{2}}{\norm f_{L^{2}(\pi)}^{2}}\geq e^{-hR_{f}(0)}\,.
\]
\end{cor}

One might notice a similarity with, for the Langevin semigroup $P_{t}$,
\[
\frac{\norm{P_{t}f}_{L^{2}(\pi)}^{2}}{\norm f_{L^{2}(\pi)}^{2}}\leq\exp\bpar{-\frac{2t}{\cpi(\pi)}}\,.
\]

\begin{proof}
From the computation above, 
\[
\de_{t}\log\norm{Q_{h}^{*}f}_{L^{2}(\pi_{h})}^{2}=-R_{f}(t)\geq-R_{f}(0)\,,
\]
so 
\[
\log\frac{\norm{Q_{h}^{*}f}_{L^{2}(\pi_{h})}^{2}}{\norm f_{L^{2}(\pi)}^{2}}\geq-hR_{f}(0)\,,
\]
which completes the proof. 
\end{proof}
It also give another useful insight into the adjoint:
\begin{cor}
[Gradient inequality] \label{cor:gradient-ineq} Let $f\in H^{1}(\pi)$.
Then for all $h>0$,
\[
\int\norm{\nabla Q_{h}^{*}f}^{2}\,\D\pi_{h}\leq\int\norm{\nabla f}^{2}\,\D\pi\,.
\]
\end{cor}

\begin{proof}
Recall that 
\[
R_{f}(h)=\frac{\E_{\pi_{h}}[\norm{\nabla Q_{h}^{*}f}^{2}]}{\E_{\pi_{h}}[\norm{Q_{h}^{*}f}^{2}]}\leq\frac{\E_{\pi}[\norm{\nabla f}^{2}]}{\E_{\pi}[\norm f^{2}]}\,.
\]
This implies that 
\[
\E_{\pi_{h}}[\norm{\nabla Q_{h}^{*}f}^{2}]\leq\frac{\E_{\pi_{h}}[\norm{Q_{h}^{*}f}^{2}]}{\E_{\pi}[\norm f^{2}]}\,\E_{\pi}[\norm{\nabla f}^{2}]\leq\E_{\pi}[\norm{\nabla f}^{2}]\,,
\]
which completes the proof.
\end{proof}

\subsection{Bound on the thin-shell constant}

We are now ready to prove $\sigma_{n}\lesssim\log\psi_{\kls}(n)$.
Pick an isotropic logconcave $\pi$ that saturates the thin-shell
constant (i.e., $\sigma_{n}/2\leq\sigma_{\pi}$). It suffices to show
that
\[
\sigma_{\pi}\lesssim\log\psi_{\kls}(\pi)\asymp\log\cpi(\pi)=\Abs{\log\lda_{\pi}}\,,
\]
where $\lda_{\pi}$ denotes the spectral gap of $-\mc L$ associated
with $\pi$.

Recall that 
\[
\sigma_{\pi}^{2}\lesssim\int_{\lda_{\pi}}^{\infty}\frac{F(\lda)}{\lda^{2}}\,\D\lda=\int_{e^{-1}}^{\infty}\frac{F(\lda)}{\lda^{2}}\,\D\lda+\int_{\lda_{\pi}\wedge e^{-1}}^{e^{-1}}\frac{F(\lda)}{\lda^{2}}\,\D\lda\,,
\]
where $F(\lda)=n^{-1}\sum_{i=1}^{n}\norm{E_{\lda}x_{i}}_{L^{2}(\pi)}^{2}$
is the average spectral mass of the coordinate functions below level
$\lda$. Note that $F(\lda)\leq1$ for $\lda>0$, as seen in 
\[
F(\lda)=\frac{1}{n}\sum_{i=1}^{n}\norm{E_{\lda}x_{i}}_{L^{2}(\pi)}^{2}\leq\frac{1}{n}\sum_{i=1}^{n}\norm{x_{i}}_{L^{2}(\pi)}^{2}=\frac{1}{n}\int\norm x^{2}\,\D\pi=1\,.
\]
Thus,
\begin{equation}
\sigma_{\pi}^{2}\lesssim\int_{e^{-1}}^{\infty}\frac{1}{\lda^{2}}\,\D\lda+\int_{\lda_{\pi}\wedge e^{-1}}^{e^{-1}}\frac{F(\lda)}{\lda^{2}}\,\D\lda\lesssim1+\int_{\lda_{\pi}\wedge e^{-1}}^{e^{-1}}\frac{F(\lda)}{\lda^{2}}\,\D\lda\,.\label{eq:sigma-decomp}
\end{equation}

\paragraph{Proof outline.}

We can now focus on bounding $F(\lda)$ when $\lda\leq e^{-1}$. As
mentioned earlier, our approach is to translate the bound on $\norm{Q_{h}^{*}x}_{L^{2}(\pi_{h})}^{2}$
into one on $F(\lda)$. This starts with the observation made in Corollary~\ref{cor:adjoint-norm-change}:
\[
\norm{Q_{h}^{*}E_{\lda}x_{i}}_{L^{2}(\pi_{h})}^{2}\geq e^{-hR_{E_{\lda}x_{i}}(0)}\,\norm{E_{\lda}x_{i}}_{L^{2}(\pi)}^{2}\,.
\]
As $R_{E_{\lda}x_{i}}(0)\leq\lda$ (due to the spectral projection
$E_{\lda}$), it follows that $\norm{Q_{h}^{*}E_{\lda}x_{i}}_{L^{2}(\pi_{h})}^{2}\geq e^{-h\lda}\,\norm{E_{\lda}x_{i}}_{L^{2}(\pi)}^{2}$,
so
\[
\frac{1}{n}\sum_{i=1}^{n}\norm{Q_{h}^{*}E_{\lda}x_{i}}_{L^{2}(\pi_{h})}^{2}\geq e^{-h\lda}\,\frac{1}{n}\sum_{i=1}^{n}\norm{E_{\lda}x_{i}}_{L^{2}(\pi)}^{2}=e^{-h\lda}\,F(\lda)\,.
\]
Therefore, we can introduce $\sum\norm{Q_{h}^{*}x_{i}}_{L^{2}(\pi_{h})}^{2}$
simply by, for $E_{\lda}^{\perp}:=\id-E_{\lda}$,
\[
\sum_{i}\norm{Q_{h}^{*}E_{\lda}x_{i}}_{L^{2}(\pi_{h})}^{2}\leq\sum_{i}\norm{Q_{h}^{*}x_{i}}_{L^{2}(\pi_{h})}^{2}+2\underbrace{\Abs{\sum_{i=1}^{n}\int(Q_{h}^{*}E_{\lda}^{\perp}x_{i})\,(Q_{h}^{*}E_{\lda}x_{i})\,\D\pi_{h}}}_{=:\msf A}=\norm{Q_{h}^{*}x}_{L^{2}(\pi_{h})}^{2}+2\msf A\,.
\]
Hence, it suffices to bound the first term and $\msf A$ (and then
simply relate this upper bound to the lower bound on the LHS, which
is $e^{-h\lda}F(\lda)$).

\paragraph{Bound on $\protect\norm{Q_{h}^{*}x}_{L^{2}(\pi_{h})}^{2}$.}

Let us bound $\norm{Q_{h}^{*}x}_{L^{2}(\pi_{h})}^{2}=\E[\norm{b_{t}}^{2}]$.
Since $\de_{t}\E[\norm{b_{t}}^{2}]=\E\tr(\Sigma_{t}^{2})\lesssim n$,
we have 
\[
\E[\norm{b_{t}}^{2}]\lesssim nt\,.
\]
Moreover,
\[
\E[\norm{b_{t}}^{2}]=\E\Bbrack{\Bnorm{\int x\,\D\pi_{t}}^{2}}\leq\E\Bbrack{\int\norm x^{2}\,\D\pi_{t}}=\int\norm x^{2}\,\D\pi_{0}=n\,.
\]
Combining these two bounds, for some constant $c>e$,
\[
\norm{Q_{h}^{*}x}_{L^{2}(\pi_{h})}^{2}\le cn\,(h^{-1}\wedge1)\,.
\]

\paragraph{Bound on $\protect\msf A$.}

We now bound the term $\msf A$: using (\ref{eq:iterated-diff-Gamma}),
\begin{align*}
 & \Babs{\sum_{i=1}^{n}\int(Q_{h}^{*}E_{\lda}^{\perp}x_{i})\,(Q_{h}^{*}E_{\lda}x_{i})\,\D\pi_{h}}\\
 & =\Babs{\sum_{i=1}^{n}\int_{0}^{h}\de_{s}\int(Q_{s}^{*}E_{\lda}^{\perp}x_{i})\,(Q_{s}^{*}E_{\lda}x_{i})\,\D\pi_{s}\D s}\\
 & =\Babs{\sum_{i=1}^{n}\int_{0}^{h}\int\inner{\nabla Q_{s}^{*}E_{\lda}^{\perp}x_{i},\nabla Q_{s}^{*}E_{\lda}x_{i}}\,\D\pi_{s}\D s}\\
 & =\Babs{\sum_{i=1}^{n}\int_{0}^{h}\int\inner{\nabla Q_{s}^{*}x_{i}-\nabla Q_{s}^{*}E_{\lda}x_{i},\nabla Q_{s}^{*}E_{\lda}x_{i}}\,\D\pi_{s}\D s}\\
 & \leq\Babs{\underbrace{\sum_{i=1}^{n}\int_{0}^{h}\int\inner{\nabla Q_{s}^{*}x_{i},\nabla Q_{s}^{*}E_{\lda}x_{i}}\,\D\pi_{s}\D s}_{=:\msf I}}+\underbrace{\sum_{i=1}^{n}\int_{0}^{h}\int\norm{\nabla Q_{s}^{*}E_{\lda}x_{i}}^{2}\,\D\pi_{s}\D s}_{=:\msf{II}}\,.
\end{align*}

We bound $\msf I$ and $\msf{II}$, respectively. Let us denote by
$\mc L_{s}$ the Langevin generator associated with $\pi_{s}=\pi*\gamma_{s}$.
As for $\msf I$, using (\ref{eq:IBP}) and Cauchy-Schwarz,
\begin{align*}
\sum_{i=1}^{n}\int\inner{\nabla Q_{s}^{*}x_{i},\nabla Q_{s}^{*}E_{\lda}x_{i}}\,\D\pi_{s} & =-\sum_{i=1}^{n}\int Q_{s}^{*}x_{i}\cdot\mc L_{s}Q_{s}^{*}E_{\lda}x_{i}\,\D\pi_{s}\\
 & \leq\Bpar{\sum_{i=1}^{n}\int(Q_{s}^{*}x_{i})^{2}\,\D\pi_{s}}^{1/2}\Bpar{\sum_{i=1}^{n}\int(\mc L_{s}Q_{s}^{*}E_{\lda}x_{i})^{2}\,\D\pi_{s}}^{1/2}\\
 & \leq c^{1/2}n^{1/2}\,(s^{-1/2}\wedge1)\,\Bpar{\sum_{i=1}^{n}\int(\mc L_{s}Q_{s}^{*}E_{\lda}x_{i})^{2}\,\D\pi_{s}}^{1/2}\,.
\end{align*}
Letting $\Gamma_{2,s}(f):=\Gamma_{2,s}(f,f)$ be the dynamic iterated
carr\'e du champ (\ref{eq:dynamic-cd-champ}) and using (\ref{eq:int-Bochner}),
\begin{align*}
\int(\mc L_{s}Q_{s}^{*}E_{\lda}x_{i})^{2}\,\D\pi_{s} & =\int(\norm{\hess Q_{s}^{*}E_{\lda}x_{i}}_{\hs}^{2}-\inner{\nabla^{2}\log\pi_{s}\,\nabla Q_{s}^{*}E_{\lda}x_{i},Q_{s}^{*}E_{\lda}x_{i}})\,\D\pi_{s}\\
 & =\int\bpar{\Gamma_{2,s}(Q_{s}^{*}E_{\lda}x_{i})+\inner{\nabla^{2}\log\pi_{s}\,\nabla Q_{s}^{*}E_{\lda}x_{i},Q_{s}^{*}E_{\lda}x_{i}}}\,\D\pi_{s}\\
 & \leq\int\Gamma_{2,s}(Q_{s}^{*}E_{\lda}x_{i})\,\D\pi_{s}=-\de_{s}\int\Gamma_{1,s}(Q_{s}^{*}E_{\lda}x_{i})\,\D\pi_{s}\,.
\end{align*}
Putting this back above,
\[
\Babs{\sum_{i=1}^{n}\int\inner{\nabla Q_{s}^{*}x_{i},\nabla Q_{s}^{*}E_{\lda}x_{i}}\,\D\pi_{s}}\leq c^{1/2}n^{1/2}\,(s^{-1/2}\wedge1)\,\Bpar{-\de_{s}\sum_{i=1}^{n}\int\Gamma_{1,s}(Q_{s}^{*}E_{\lda}x_{i})\,\D\pi_{s}}^{1/2}\,.
\]
Therefore,
\begin{align*}
\msf I & =\Babs{\int_{0}^{h}\sum_{i=1}^{n}\int\inner{\nabla Q_{s}^{*}x_{i},\nabla Q_{s}^{*}E_{\lda}x_{i}}\,\D\pi_{s}\D s}\\
 & \leq c^{1/2}n^{1/2}\int_{0}^{h}(s^{-1/2}\wedge1)\,\Bpar{-\de_{s}\underbrace{\sum_{i=1}^{n}\int\Gamma_{1,s}(Q_{s}^{*}E_{\lda}x_{i})\,\D\pi_{s}}_{=:\vphi(s)}}^{1/2}\,\D s\,.
\end{align*}
To bound this, we use a simple algebraic lemma:
\begin{lem}
[{\cite{guan2024note}}]Let $\vphi\geq0$ be a differentiable and
non-increasing on $\R_{\geq0}$. Then, for any integer $N\geq0$,
\[
\int_{0}^{2^{N}}(s^{-1/2}\wedge1)\,\bigl(-\vphi'(s)\bigr)^{1/2}\,\D s\le\sqrt{\vphi(0)\,(N+1)}\,.
\]
\end{lem}

To use this, we bound
\begin{align}
\vphi(0) & =\sum_{i=1}^{n}\int\Gamma_{1,0}(E_{\lda}x_{i})\,\D\pi=\sum_{i=1}^{n}\int\norm{\nabla E_{\lda}x_{i}}^{2}\,\D\pi\label{eq:repeat}\\
 & =\sum_{i}\int(E_{\lda}x_{i})(-\mc LE_{\lda}x_{i})\,\D\pi\leq\lda\sum_{i}\int\norm{E_{\lda}x_{i}}^{2}\,\D\pi\nonumber \\
 & =\lda n\,\frac{1}{n}\sum_{i=1}^{n}\norm{E_{\lda}x_{i}}_{L^{2}(\pi)}^{2}=\lda nF(\lda)\,.\nonumber 
\end{align}
Using this lemma with this bound and assuming $h\geq1$ for a while,
\begin{align*}
\msf I & \leq c^{1/2}n^{1/2}\int_{0}^{h}(s^{-1/2}\wedge1)\,\Bpar{-\de_{s}\sum_{i=1}^{n}\int\Gamma_{1,s}(Q_{s}^{*}E_{\lda}x_{i})\,\D\pi_{s}}^{1/2}\,\D s\\
 & \leq c^{1/2}\lda^{1/2}nF^{1/2}(\lda)\,(\lceil\log_{2}h\rceil+1)^{1/2}\,.
\end{align*}

As for $\msf{II}$, using the gradient inequality (Corollary~\ref{cor:gradient-ineq})
and reproducing (\ref{eq:repeat}),
\[
\sum_{i=1}^{n}\int_{0}^{h}\int\norm{\nabla Q_{s}^{*}E_{\lda}x_{i}}^{2}\,\D\pi_{s}\D s\leq\sum_{i=1}^{n}\int_{0}^{h}\int\norm{\nabla E_{\lda}x_{i}}^{2}\,\D\pi\D s\leq\lda hnF(\lda)\,.
\]

Combining the bounds on $\msf I$ and $\msf{II}$, the term $\msf A$
is bounded as, if $h\geq1$ 
\[
\Babs{\sum_{i=1}^{n}\int(Q_{h}^{*}E_{\lda}^{\perp}x_{i})\,(Q_{h}^{*}E_{\lda}x_{i})\,\D\pi_{h}}\leq c^{1/2}\lda^{1/2}nF^{1/2}(\lda)\,(\lceil\log_{2}h\rceil+1)^{1/2}+\lda hnF(\lda)\,.
\]

\paragraph{Putting together.}

We now provide a bound on $F(\lda)$ on $\lda\in(0,e^{-1})$ (so $\Abs{\log\lda}\geq1$)
as 
\[
F(\lda)\lesssim\lda\,\abs{\log\lda}\,.
\]
We may assume that 
\[
\frac{F(\lda)}{\lda}\geq e^{-4}\,.
\]
Otherwise, the desired inequality clearly holds on $\lda\in(0,e^{-1})$.

Recall that 

\[
\frac{1}{n}\sum_{i=1}^{n}\norm{Q_{h}^{*}E_{\lda}x_{i}}_{L^{2}(\pi_{h})}^{2}\geq e^{-h\lda}\,F(\lda)\,,
\]
and 
\[
\sum_{i}\norm{Q_{h}^{*}E_{\lda}x_{i}}_{L^{2}(\pi_{h})}^{2}\leq\norm{Q_{h}^{*}x}_{L^{2}(\pi_{h})}^{2}+2\msf A\,.
\]
Combining these, if $h\geq1$,
\begin{align*}
e^{-h\lda}F(\lda) & \leq c\,(h^{-1}\wedge1)+2c^{1/2}\lda^{1/2}F^{1/2}(\lda)\,(\lceil\log_{2}h\rceil+1)^{1/2}+2\lda hF(\lda)\\
 & =ch^{-1}+2c^{1/2}\lda^{1/2}F^{1/2}(\lda)\,(\lceil\log_{2}h\rceil+1)^{1/2}+2\lda hF(\lda)\,.
\end{align*}
Take
\[
h=\lda^{-1/2}F^{-1/2}(\lda)\,,
\]
under which $1/2\leq e^{-\lda^{1/2}F^{-1/2}}$ holds due to $F/\lda\geq e^{-4}$.
Thus,
\[
\half\,F\lesssim\lda^{1/2}F^{1/2}+\lda^{1/2}F^{1/2}\,(\lceil\log_{2}h\rceil+1)^{1/2}+\lda^{1/2}F^{1/2}\,,
\]
and
\[
F\lesssim\lda\,(\lceil\log_{2}h\rceil+1)\,.
\]
One can readily check that $\lceil\log_{2}h\rceil\lesssim\abs{\log\lda}$,
so the desired inequality follows.

If $h\leq1$, we could have simply worked with $2^{|\log\lda|}$ instead
of $h$ in the algebraic lemma, so 
\[
F\lesssim\lda\,\abs{\log\lda}\,,
\]
which completes the proof of the desired inequality.

Substituting the desired inequality back to (\ref{eq:sigma-decomp}),
\[
\sigma_{\pi}^{2}\lesssim1+\int_{e^{-1}\wedge\lda_{\pi}}^{e^{-1}}\frac{F(\lda)}{\lda^{2}}\,\D\lda\lesssim1+\int_{e^{-1}\wedge\lda_{\pi}}^{e^{-1}}\frac{\abs{\log\lda}}{\lda}\,\D\lda\lesssim\abs{\log(\lda_{\pi}\wedge e^{-1})}^{2}\,,
\]
and thus
\[
\sigma_{\pi}^{2}\lesssim|\log\lda_{\pi}|^{2}\,.
\]

\end{document}